\documentclass[11pt]{amsart}

\usepackage{amssymb,epsfig,mathrsfs}
\usepackage{amsmath,verbatim,psfrag}
\usepackage[all]{xy}
\usepackage{pinlabel,url}

\usepackage[hmargin=1.375in, vmargin=1.125in]{geometry}
\usepackage[font=small,format=plain,labelfont=bf,up,textfont=it,up]{caption}

\newcommand{\nc}{\newcommand}
\nc{\nt}{\newtheorem}
\nc{\dmo}{\DeclareMathOperator}

\nt{theorem}{Theorem}[section]
\nt{prop}[theorem]{Proposition}
\nt{lemma}[theorem]{Lemma}
\nt{fact}[theorem]{Fact}
\nt{cor}[theorem]{Corollary}
\nt{conjecture}[theorem]{Conjecture}
\nt{question}[theorem]{Question}
\nt{problem}[theorem]{Problem}
\nt*{metaconjecture}{Metaconjecture}
\nt*{genmetaconjecture}{Generalized Metaconjecture}
\nt*{main}{Main Theorem}
\nt*{remain}{Main Theorem (Restated)}

\dmo{\MCG}{Mod^\pm}
\dmo{\Mod}{Mod}
\dmo{\Out}{Out}
\dmo{\PMCG}{PMod^\pm}
\dmo{\PMod}{PMod}
\dmo{\PMF}{PMF}
\dmo{\Homeo}{Homeo}
\dmo{\Span}{Span}
\dmo{\Supp}{Supp}
\dmo{\Hom}{Hom}
\dmo{\Exchange}{Ex}

\dmo{\Comm}{Comm}

\dmo{\Aut}{Aut}
\dmo{\Inn}{Inn}

\dmo{\Lk}{Lk}

\nc{\overbar}[1]{\mkern 1.5mu\overline{\mkern-1.5mu#1\mkern-1.5mu}\mkern 1.5mu}

\nc{\C}{\mathcal{C}}
\nc{\D}{\mathcal{D}}
\nc{\N}{\mathcal{N}}
\nc{\G}{\mathcal{G}}
\nc{\DS}{\mathcal{D}}
\nc{\A}{\mathcal{A}}
\nc{\M}{\mathcal{M}}
\nc{\X}{\mathcal{X}}
\nc{\K}{\mathcal{K}}
\nc{\T}{\mathcal{T}}

\newcommand{\BigFreeProd}{\mathop{\mbox{\Huge{$\ast$}}}}

\dmo{\RC}{RC}
\dmo{\BRC}{BRC}

\dmo{\GL}{GL}

\nc{\Side}{\mathscr{S}}
\nc{\B}{\mathscr{B}}
\nc{\R}{\mathcal{R}}

\nc{\I}{\mathcal{I}}

\nc{\Z}{\mathbb{Z}}
\nc{\Q}{\mathbb{Q}}

\nc{\p}[1]{\medskip\paragraph{{\bf #1}}}
\nc{\margin}[1]{\marginpar{\scriptsize #1}}

\nc{\bl}{ \begin{list}{$\cdot$}{
\setlength{\leftmargin}{.5in}
\setlength{\rightmargin}{.5in}
\setlength{\parsep}{0.5ex plus .2ex minus 0ex}
\setlength{\itemsep}{0.2ex plus 0.2ex minus 0ex}
}
}

\nc{\el}{\end{list}}

\title[Normal subgroups of mapping class groups]{Normal subgroups of mapping class groups and the metaconjecture of Ivanov}

\begin{document}
	
\input{epsf.sty}

\author{Tara E. Brendle}

\author{Dan Margalit}

\address{Tara E. Brendle \\School of Mathematics \& Statistics \\ University Place \\ University of Glasgow \\ G12 8SQ \\ tara.brendle@glasgow.ac.uk}

\address{Dan Margalit \\ School of Mathematics\\ Georgia Institute of Technology \\ 686 Cherry St. \\ Atlanta, GA 30332 \\  margalit@math.gatech.edu}

\thanks{This material is based upon work supported by the EPSRC under grant EP/J019593/1 and the National Science Foundation under Grant Nos. DMS - 1057874 and DMS - 1510556.}



\begin{abstract}
We prove that if a normal subgroup of the extended mapping class group of a closed surface has an element of sufficiently small support then its automorphism group and abstract commensurator group are both isomorphic to the extended mapping class group.  The proof relies on another theorem we prove, which states that many simplicial complexes associated to a closed surface have automorphism group isomorphic to the extended mapping class group.  These results resolve the metaconjecture of N.V. Ivanov, which asserts that any ``sufficiently rich'' object associated to a surface has automorphism group isomorphic to the extended mapping class group, for a broad class of such objects.  As applications, we show: (1) right-angled Artin groups and surface groups cannot be isomorphic to normal subgroups of mapping class groups containing elements of small support, (2) normal subgroups of distinct mapping class groups cannot be isomorphic if they both have elements of small support, and (3)~distinct normal subgroups of the mapping class group with elements of small support are not isomorphic. Our results also suggest a new framework for the classification of normal subgroups of the mapping class group.
\end{abstract}

\maketitle

\section{Introduction}
\label{sec:intro}

The mapping class group $\Mod(S_g)$ of a closed, orientable surface $S_g$ of genus $g$ is the group of homotopy classes of orientation-preserving homeomorphisms of $S_g$.  The extended mapping class group $\MCG(S_g)$ is the group of homotopy classes of all homeomorphisms of $S_g$.

\p{Normal subgroups.} The main result of this paper, Theorem~\ref{main:normal}, gives a general condition for a normal subgroup of $\Mod(S_g)$ or $\MCG(S_g)$ to have automorphism group and abstract commensurator group isomorphic to $\MCG(S_g)$.  Previously this result was only known for very specific subgroups, namely, the Torelli group and its variants.  Our general condition, which is that the normal subgroup contains an element of ``small'' support, is easy to verify and applies to most natural normal subgroups, including the Torelli groups and their variants, as well as many others. 

Farb has posed the problem of computing the abstract commensurators for various subgroups of $\Mod(S_g)$ \cite[Problem 2.2]{farb}.  Our theorem solves this problem in many cases. It also addresses the so-called metaconjecture of Ivanov; see below.

An observation of L. Chen further implies that each normal subgroup with an element of small support is unique in that no other normal subgroup of $\MCG(S_g)$ is isomorphic to it; see Corollary~\ref{cor:chen}.  So, for example, the terms of the Johnson filtration, the terms of the Magnus filtration, the level $m$ congruence subgroups together form a collection of pairwise non-isomorphic subgroups of the mapping class group.

Two further applications of our theorem are restrictions on the isomorphism types of subgroups of mapping class groups.  For example, if $g \neq g'$ a normal subgroup of $\Mod(S_g)$ with an element of small support cannot be isomorphic to a normal subgroup of $\Mod(S_{g'})$ with an element of small support; see Corollary~\ref{cor:notiso}.  Also, a normal subgroup of $\Mod(S_g)$ with an element of small support cannot be isomorphic to a surface group or a right-angled Artin group; see Corollary~\ref{cor:raag}.  The key idea for both of these applications is to use the automorphism group as an invariant of the isomorphism class of a group.

Our results suggest a dichotomy for the normal subgroups of mapping class groups, namely, into those that have automorphism group isomorphic to the extended mapping class group and those that do not contain elements of small support; see Conjecture~\ref{conj:dichotomy} below.  As discussed by Farb \cite{farb}, a traditional classification theorem for normal subgroups of mapping class groups, in the form of a complete list of isomorphism types, is almost certainly out of reach.  However, our conjecture provides a new framework for a coarser classification of normal subgroups of mapping class groups.

\p{Simplicial complexes.} We prove our main result about normal subgroups of mapping class groups by reducing it to a problem about automorphisms of simplicial complexes.  To this end, we consider simplicial complexes whose vertices correspond to connected subsurfaces of the ambient surface and whose edges correspond to disjointness.  Our main theorem about automorphisms of simplicial complexes, Theorem~\ref{main:complex}, gives general conditions for such a simplicial complex to have automorphism group isomorphic to the extended mapping class group.

Our result applies to many natural simplicial complexes associated to a surface, including some that were already known to have automorphism group the extended mapping class group.  Our theorem is the first to address infinitely many complexes with a single argument, and indeed it applies to a wide class.

\p{Ivanov's metaconjecture} Our work in this paper has its origins in the seminal work of N. V. Ivanov \cite{ivanov}.  He proved that for $g$ at least 3, the automorphism group of the complex of curves $\C(S_g)$ is isomorphic to $\MCG(S_g)$ (see also \cite{korkmaz,luo}).  As one application, he proved that the automorphism group of $\Mod(S_g)$---and also the abstract commensurator group of $\Mod(S_g)$---is isomorphic to $\MCG(S_g)$ (cf. \cite{mccarthy,tchangang}).  Ivanov's work inspired a number of theorems of the following form:
\begin{enumerate}
\item the automorphism group of some particular simplicial complex associated to a surface $S$ is isomorphic to $\MCG(S)$, and
\item the automorphism group of some particular normal subgroup of $\Mod(S)$ is isomorphic to $\MCG(S)$.
\end{enumerate}
Many are found among the 96 (and counting) citations on MathSciNet for the aforementioned paper of Ivanov. In response, Ivanov posed the following \cite{ivanov15}.

\begin{metaconjecture}
Every object naturally associated to a surface $S$ and having a sufficiently rich structure has $\MCG(S)$ as its group of automorphisms.  Moreover, this can be proved by a reduction to the theorem about the automorphisms of $\C(S)$.
\end{metaconjecture}
There are many results supporting Ivanov's metaconjecture, some quite classical, going back to the work of Dehn \cite[Paper 8]{dehn} and Nielsen \cite{nielsen} in the 1920s.  Also, in the 1930s Teichm\"uller's showed that the group of automorphisms of the universal curve over Teichm\"uller space is the extended mapping class group \cite{TeichTrans,teich}; this theorem was put into a more general framework by Grothendieck \cite[Theorem 3.1]{grothendieck}.  For an overview of other related results, see our survey paper \cite{survey} or the one by McCarthy--Papadopoulos \cite{mcpap}.  In our survey we suggest a generalization of Ivanov's metaconjecture, from surfaces to other spaces.

Our results may be viewed as a resolution of Ivanov's metaconjecture for a wide class of normal subgroups of $\Mod(S_g)$ and a wide class of simplicial complexes associated to $S_g$.  Ivanov's metaconjecture is deliberately vague: the terms ``object,''  ``naturally,'' and ``sufficiently rich'' are left open to interpretation.  In this paper we formulate his metaconjecture into two precise statements about normal subgroups and simplicial complexes (the objects at hand) by finding appropriate notions of sufficient richness in each case.

Our work follows Ivanov in the sense that we reduce our problem about normal subgroups of the mapping class group to a problem about automorphisms of simplicial complexes.  In his work, as well as in the other subsequent works, there is a single group being considered, and a single simplicial complex.  A central challenge in this paper, and one of the main departures from Ivanov's work, is that we consider many groups all at the same time, each requiring its own simplicial complex.  As such, we need to prove that all of these simplicial complexes have automorphism group isomorphic to the extended mapping class group.  Also, when it comes to the normal subgroups we consider we do not have full information about which elements are in, and are not in, our given subgroups; we are only given the information that there is a single element whose support is small (in the precise sense defined below).


\subsection{Results on normal subgroups}

In order to state our main theorem about normal subgroups of the mapping class group, Theorem~\ref{main:normal}, we require several definitions.

\p{Small components.} Let $f \in \Mod(S_g)$ be a pure mapping class.  Briefly, this means that $f$ is a product of partial pseudo-Anosov mapping classes and Dehn twists, all with disjoint supports; see Section~\ref{section:normal} for details.  We will define  a certain measure of complexity $\hat g(f)$ for $f$.  First, a \emph{region} in $S_g$ is a compact, connected subsurface with no boundary component homotopic to a point in $S_g$.  Next, a \emph{fitting region} for $f$ is a region $Q$ in $S_g$ so that some Nielsen--Thurston component of (some representative of) $f$ has support that is non-peripheral in $Q$ and so that the other Nielsen--Thurston components of $f$ have support disjoint from $Q$.  Finally we define $\hat g(f)$ to be the smallest number $k$ so that there is a region of $S_g$ that has genus $k$ and connected (possibly empty) boundary and that contains a fitting region for $f$.

We will say that a pure element $f$ of $\Mod(S_g)$ has a \emph{small component} if $\hat g(f) < g/3$.  The main hypothesis of Theorem~\ref{main:normal} is that the normal subgroup under consideration has a nontrivial element with a small component.  

For example, if $f$ has a partial pseudo-Anosov Nielsen--Thurston component whose support is contained in a region of genus $k$ with connected boundary then $\hat g(f)$ is at most $k$.  Also if the entire support of $f$ is contained in a region of genus $k$ with connected boundary then $\hat g(f)$ is at most $k+1$ (for instance if $f$ is a Dehn twist about a separating curve of genus $k$ then $\hat g(f) = k+1$). 

For a subgroup $N$ of $\MCG(S_g)$, we define $\hat g(N)$ to be the minimum of $\hat g (f)$ for nontrivial pure $f$ in $N$ (by default pure elements lie in $\Mod(S_g)$).  See Section~\ref{section:normal} for the definitions of pure elements and Nielsen--Thurston components.

\p{Abstract commensurators.} The \emph{abstract commensurator group} of a group $G$ is the group of equivalence classes of isomorphisms between finite-index subgroups of $G$, where two isomorphisms are equivalent if they agree on some finite-index subgroup of $G$.

\p{Natural maps} Let $N$ be a normal subgroup of $\Mod(S_g)$.  There is a natural homomorphism $\Mod(S_g) \to \Aut N$ where $f \in \Mod(S_g)$ maps to the element of $\Aut N$ given by conjugation by $f$; there is a similar map if $N$ is normal in $\MCG(S_g)$.  Also, there is a natural homomorphism $\Aut N \to \Comm N$ where an automorphism maps to its equivalence class.  

There is one more natural homomorphism in the statement of Theorem~\ref{main:normal}.  If $N$ is a subgroup of $\MCG(S_g)$, there is a map $\MCG(S_g) \to \Hom(N,\MCG(S_g))$ where $f \in \MCG(S_g)$ maps to the homomorphism taking $n$ to $fnf^{-1}$.  If for each $f \in \MCG(S_g)$ there is a finite-index subgroup $H$ of $N$ so that $fHf^{-1}$ has finite index in $N$ then we may regard the map $\MCG(S_g) \to \Hom(N,\MCG(S_g))$ as a map $\MCG(S_g) \to \Comm N$.  When this map exists, we call it the natural map $\MCG(S_g) \to \Comm N$.

\p{Statement of our main theorem about normal subgroups.} Our first main theorem describes the automorphism group and the abstract commensurator groups of all normal subgroups of $\Mod(S_g)$ and $\MCG(S_g)$ that contain elements of small support.  So each of these subgroups remembers the structure of the full mapping class group.

\begin{theorem}
\label{main:normal} Let $N$ be a normal subgroup of either $\MCG(S_g)$ or $\Mod(S_g)$ with $g \geq 3 \hat g(N) + 1$.
\begin{enumerate}
\item\label{normal emcg} If $N$ is normal in $\MCG(S_g)$ then the natural maps 
\[ \MCG(S_g) \to \Aut N \to \Comm N  \]
are isomorphisms.
\item\label{normal mcg} If $N$ is normal in $\Mod(S_g)$ but not in $\MCG(S_g)$ then there is a natural map $\Comm N \to \MCG(S_g)$, the natural maps
\[ \Mod(S_g) \to \Aut N \to \Comm N \to \MCG(S_g) \]
are injective, the first map is an isomorphism, and the composition is the inclusion.  In particular $\Comm N$ is isomorphic to either $\Mod(S_g)$ or to $\MCG(S_g)$.  In the first case the second map is an isomorphism.  In the second case the inverse of the isomorphism $\Comm N \to \MCG(S_g)$ is the natural map $\MCG(S_g) \to \Comm N$.
\end{enumerate}
\end{theorem}

Most of the well-studied normal subgroups of the mapping class group---for instance the Torelli group and the terms of the Johnson filtration---are normal in the extended mapping class group, and so the first statement of Theorem~\ref{main:normal} applies.    We expect that there are subgroups $N$ that are normal in $\Mod(S_g)$ but not $\MCG(S_g)$ and that satisfy $\Comm N \cong \MCG(S_g)$.  Examples in the case $g=1$ were explained to us by Jones \cite{jones}. 

By Lemma~\ref{comm trick} below, any normal subgroup of $\Mod(S_g)$ or $\MCG(S_g)$ that has a pure element with a small component also has a pure element whose entire support is small, meaning that the support is contained as a nonperipheral subsurface in a subsurface of $S_g$ with connected boundary and genus $k < g/3$. Therefore, the hypothesis on $N$ in Theorem~\ref{main:normal} is equivalent to the hypothesis that $N$ has a nontrivial element with small support.

The hypothesis of small supports in Theorem~\ref{main:normal} is certainly not optimal.  Indeed, if we take $N=\Mod(S_g)$, then Theorem~\ref{main:normal}\eqref{normal emcg} implies that $\Aut \Mod(S_g)$ is isomorphic to $\MCG(S_g)$ when $g \geq 4$.  On the other hand, Ivanov already proved this result for $g \geq 3$.

\p{Exotic normal subgroups.}   One might hope that all normal subgroups of $\Mod(S_g)$ have automorphism group $\MCG(S_g)$, in other words that the hypothesis on $\hat g(N)$ in Theorem~\ref{main:normal} is not necessary.  However, this is certainly not the case: Dahmani, Guirardel, and Osin \cite{dgo} proved that there are normal subgroups of $\Mod(S_g)$ isomorphic to infinitely generated free groups; see also the recent work of Clay, Mangahas, and the second author \cite{cmm}.  Each nontrivial element in the Dahmani--Guirardel--Osin subgroups is pseudo-Anosov.  The hypotheses of Theorem~\ref{main:normal} exactly rule out this type of example, as $\hat g(f) = g$ for any pseudo-Anosov $f$.

\p{Prior results.}  Our Theorem~\ref{main:normal} recovers many previously known results.  After Ivanov's original work, Farb and Ivanov \cite{farbivanovannounce,farbivanov} proved that the automorphism group and the abstract commensurator group of the Torelli subgroup of $\Mod(S_g)$ is isomorphic to $\MCG(S_g)$, and the authors of this paper proved \cite{kg} that the automorphism group  and the abstract commensurator group of the Johnson kernel, an infinite index subgroup of the Torelli group, is isomorphic to $\MCG(S_g)$.  Bridson, Pettet, and Souto \cite{bps} then announced the following result: every normal subgroup of the extended mapping class group that is contained in the Torelli group and has the property that each subsurface of Euler characteristic $-2$ supports a non-abelian free subgroup has automorphism group and abstract commensurator group isomorphic to $\MCG(S_g)$.  In particular for $g \geq 4$ this applies to every term of the Johnson filtration of $\Mod(S_g)$.  The Johnson filtration is an infinite sequence of nested normal subgroups of $\Mod(S_g)$ whose intersection is the trivial subgroup; the first two groups in the sequence are the Torelli group and the Johnson kernel.     Theorem~\ref{main:normal} implies each of the above results.

\p{Applications and examples}   Many natural subgroups of $\Mod(S_g)$ and $\MCG(S_g)$ come in families, meaning that there is one normal subgroup $N_g$ for each $g$.  Also, it is often the case that $\hat g(N_g)$ does not depend on $g$, and so Theorem~\ref{main:normal} applies to all members of the family once $g$ is large enough.

As one example the Torelli group $\I(S_g)$ is the normal subgroup of $\MCG(S_g)$ defined as the kernel of the action of $\Mod(S_g)$ on $H_1(S_g;\Z)$.  The Johnson kernel $\K(S_g)$ is the infinite-index subgroup of $\I(S_g)$ generated by Dehn twists about separating curves.  For all $g$ we have $\hat g(\I(S_g)) = \hat g(\K(S_g)) = 2$.  Theorem~\ref{main:normal}\eqref{normal emcg} applies to both, thus recovering our earlier result and the result of Farb and Ivanov for $g \geq 7$.  

Similarly, Theorem~\ref{main:normal}\eqref{normal emcg} applies to the $k$th term $N_g^k$ of the Johnson filtration, which is the kernel of the action (by outer automorphisms) of $\Mod(S_g)$ on $\pi/\pi_k$ where $\pi=\pi_1(S_g)$ and $\pi_k$ is the $k$th term of its lower central series.  We have $\hat g(N_g^k) = 2$ \cite[Proof of Theorem 5.10]{farb}.  In particular for $g \geq 7$ our theorem recovers the results announced by Bridson--Pettet--Souto. 

Beyond this, the terms of the derived series for the Torelli group, the terms of the lower central series of the Torelli group, the kernel of the Chillingworth homomorphism, and the kernel of the Birman--Craggs--Johnson homomorphism each have $\hat g = 2$ and so Theorem~\ref{main:normal} applies for $g \geq 7$.

A further application of our theorem is to the Magnus filtration of the Torelli group, defined by McNeill \cite{taylor}.  The $k$th term $M_g^k$ is the subgroup of $\Mod(S_g)$ acting trivially on $\pi/\pi_k'$ where $\pi_k'$ is the $k$th term of the lower central series of $[\pi,\pi]$.  The first term $M_g^1$ is the kernel of the Magnus representation of $\Mod(S_g)$, defined in the 1930s by Magnus \cite{magnus}.  McNeill proves that $\hat g(M_g^k) \leq 3$ for all $g \geq 3$ and $k \geq 1$ \cite[Lemma 5.2]{taylor}, and so Theorem~\ref{main:normal}\eqref{normal emcg} applies for $g \geq 10$. (McNeill discusses surfaces with boundary, but the capping homomorphism to $\Mod(S_g)$ respects the Magnus filtration.)

One may readily construct many other examples of normal subgroups satisfying the hypotheses of Theorem~\ref{main:normal}, for instance the group generated by $k$th powers of Dehn twists ($\hat g = 1$), the group generated by $k$th powers of Dehn twists about separating curves ($\hat g = 2$), the terms of the lower central series of the Torelli group ($\hat g = 2$), the normal closure of any partial pseudo-Anosov element supported on a torus with one boundary component ($\hat g = 1$), and the normal closure of any multitwist ($\hat g \leq 2$).  In the last case, to make an example with $\hat g = 2$ we should choose the support to be a pants decomposition where each curve is nonseparating.

Any normal subgroup of $\Mod(S_g)$ or $\MCG(S_g)$ containing one of the above groups automatically satisfies the hypothesis of Theorem~\ref{main:normal}.  For instance, if $N$ is the kernel of the \textrm{SU}(2)-TQFT representations of the mapping class group (see e.g. Funar \cite{funartqft}), then $N$ contains the group generated by $k$th powers of Dehn twists, and hence Theorem~\ref{main:normal} applies for $g \geq 4$.  The same applies to the subgroup of $\Mod(S_g)$ generated by the $k$th powers of all elements.

Finally, any normal subgroup of $\Mod(S_g)$ or $\MCG(S_g)$ that has finite index in a group satisfying the hypothesis of Theorem~\ref{main:normal}.  This includes, for example, the level $m$ congruence subgroups of $\Mod(S_g)$ and also the congruence subgroups defined by Ivanov via characteristic covers of surfaces \cite[Problem 1]{ivanov15}.   

\p{Chen's corollary} Ivanov--McCarthy proved that any injective map $\MCG(S_g) \to \MCG(S_g)$ is an inner automorphism \cite[Theorem 1]{im}.  As observed by Chen \cite{lei}, this theorem has the following corollary: if $N$ is a normal subgroup of $\MCG(S_g)$ where the natural map  $\MCG(S_g) \to \Aut N$ is an isomorphism then $N$ is unique in the sense that every normal subgroup of $\MCG(S_g)$ isomorphic to $N$ is equal to $N$.  This applies, for example, to the Torelli group, as well as all of the other subgroups discussed above.

Indeed, suppose that $M$ is a normal subgroup of $\MCG(S_g)$ isomorphic to $N$.  Consider the composition
\[
\Xi : \MCG(S_g) \stackrel{}{\to} \Aut M \stackrel{} \to \Aut N \stackrel{}{\to} \MCG(S_g),
\]
where the first map is the natural map given by conjugation, the second map is the isomorphism induced by any isomorphism $M \to N$, and the third map is the inverse of the natural map $\MCG(S_g) \to \Aut N$, which is an isomorphism by assumption.  All of the maps are injective (cf. Lemma~\ref{inj}) and hence the composition $\Xi$ is an injective map from $\MCG(S_g)$ to itself.  By the Ivanov--McCarthy result, $\Xi$ is an inner automorphism of $\MCG(S_g)$.  From the definitions of the three maps we observe that
\[
M \mapsto \Inn M \mapsto \Inn N \mapsto N
\]
and so $\Xi(M) = N$.  Since $\Xi$ is inner and $M$ is normal it follows that $M=N$, as desired.

Combining Chen's corollary with our main theorem we obtain the following corollary of Theorem~\ref{main:normal}.

\begin{cor}
\label{cor:chen}
Suppose $N$ is a normal subgroup of $\Mod(S_g)$ with $g \geq 3 \hat g(N) + 1$.  Any normal subgroup of $\Mod(S_g)$ isomorphic to $N$ is equal to $N$.
\end{cor}

Here is a sample application of Corollary~\ref{cor:chen}.  Fix some $g \geq 4$.  For each natural number $k$ let $\T_k$ denote the subgroup of $\Mod(S_g)$ generated by all $k$th powers of Dehn twists.  For $k < \ell$ the subgroups $\T_k$ and $\T_{\ell}$ are not equal, since $\T_\ell$ lies in the level $\ell$ congruence subgroup of $\Mod(S_g)$ and $\T_k$ does not.  Thus by Corollary~\ref{cor:chen} the subgroup $\T_k$ is not isomorphic to $\T_\ell$.  So $\T_1,\T_2,\dots$ is an infinite sequence of pairwise non-isomorphic subgroups of $\Mod(S_g)$.  Similarly, the terms $\N_k(S_g)$ of the Johnson filtration are all non-isomorphic and each such term is not isomorphic to any $\T_\ell$, etc.

The reader should compare Corollary~\ref{cor:chen} with the theorem of Akin, which states that the point-pushing subgroup $\pi_1(S_g)$ is unique among normal subgroups of $\Mod(S_{g,1})$ in the same sense.  Akin's theorem is not implied by our corollary since our Theorem~\ref{main:normal} does not apply to punctured surfaces.  McLeay \cite{alan,alan2} has proved an analogue of Theorem~\ref{main:normal} for punctured surfaces; however Akin's group does not satisfy the hypotheses there.

\p{Application: non-commensurability of normal subgroups in different mapping class groups} One kind of application of Theorem~\ref{main:normal} is to show that certain normal subgroups of $\Mod(S_g)$ cannot be isomorphic to, or even commensurable to, certain normal subgroups of $\Mod(S_{g'})$ with $g \neq g'$.  Specifically we have the following corollary of Theorem~\ref{main:normal}.

\begin{cor}
\label{cor:notiso}
Suppose $N$ and $N'$ are normal subgroups of $\Mod(S_g)$ and $\Mod(S_{g'})$, respectively, with $3\hat g(N)+1 < g$ and $3 \hat g(N') +1 < g'$.  If $g \neq g'$ then $N$ is not abstractly commensurable to $N'$.  In particular, $N$ and $N'$ are not isomorphic. 
\end{cor}

Indeed, consider $N$ and $N'$ as in the corollary.  By Theorem~\ref{main:normal} we have that $\Aut N \cong \MCG(S_g)$ and $\Aut N' \cong \MCG(S_{g'})$.  Since $\MCG(S_g)$ is not isomorphic to $\MCG(S_{g'})$ when $g \neq g'$ (consider, for instance, the rank of a maximal abelian subgroup), it follows that $N$ and $N'$ are not isomorphic.  Moreover, since $\Comm N$ is isomorphic to $\Mod(S_g)$ or $\MCG(S_g)$ and $\Comm N'$ is isomorphic to $\Mod(S_{g'})$ or $\MCG(S_{g'})$ then, since the abstract commensurator group is an invariant of the abstract commensurability class, it similarly follows that $N$ is not commensurable to $N'$ (again use the ranks of maximal abelian subgroups).

To illustrate Corollary~\ref{cor:notiso}, consider the following normal subgroups of $\Mod(S_g)$ and $\Mod(S_{2g})$.  Let $N$ be the normal closure in $\Mod(S_g)$ of a partial pseudo-Anosov element supported on a torus with one boundary component and let $N'$ be the normal closure in $\Mod(S_{2g})$ of a partial pseudo-Anosov element supported on a subsurface of genus two with one boundary component.  If $g \geq 4$ then by Corollary~\ref{cor:notiso}, the groups $N$ and $N'$ are not abstractly commensurable.

We do not know a proof of the non-commensurability of such subgroups that is independent of Theorem~\ref{main:normal}.  For example, the groups $N$ and $N'$ above cannot be distinguished by their virtual cohomological dimensions or the maximal ranks of their abelian subgroups in any obvious way (both invariants are at least $g$ for both $N$ and $N'$, but their exact values seem hard to compute).

\p{Application: an obstruction theorem for normal subgroups} Another kind of application of Theorem~\ref{main:normal} is to rule out isomorphism types for certain normal subgroups of the mapping class group.  For example, if $N \trianglelefteq \Mod(S_g)$ contains a nontrivial pure element with a small component then $N$ cannot be isomorphic to---or even abstractly commensurable to---any group whose automorphism group or abstract commensurator group is not isomorphic to $\Mod(S_g)$ or to $\MCG(S_g)$.

There are many classes of groups where no member of the class has both its automorphism group and its abstract commensurator group isomorphic to a mapping class group of a closed surface.  For example, if $G = \pi_1(S_h)$ with $h \geq 2$ then $\Aut G$ is isomorphic to $\MCG(S_{h,1})$, the extended mapping class group of a punctured surface.  The group $\MCG(S_{h,1})$ is not isomorphic to any $\Mod(S_g)$ or $\MCG(S_g)$ (the rank of a maximal abelian subgroup is divisible by 3 in the closed case and not in the punctured case), and so $\Aut G$ is not isomorphic to any $\MCG(S_g)$ (even more, $\Comm G$ is quite large).

The same holds for all right-angled Artin groups.  The abstract commensurator group of an abelian right-angled Artin group is isomorphic to $\GL_n(\Q)$ for some $n$ (and anyway there are no infinite abelian normal subgroups of $\Mod(S_g)$).  Also, the abstract commensurator group of any non-abelian right-angled Artin group contains arbitrarily large finite groups, and $\MCG(S_g)$ does not have this property; see \cite{CLM}.  

We summarize the above discussion with the following corollary.

\begin{cor}
\label{cor:raag}
If $G$ is a group with $\Aut(G)$ or $\Comm(G)$ not isomorphic to $\Mod(S_g)$ or to $\MCG(S_g)$, and $N$ is a normal subgroup of $\Mod(S_g)$ with $g \geq 3 \hat g(N) + 1$, then $N$ is not isomorphic to $G$ and further $N$ is not abstractly commensurable to $G$.  In particular, this applies when $G$ is any surface group or right-angled Artin group.
\end{cor}

It was, for instance, a folk conjecture  that the normal subgroup $T_g^k$ of $\Mod(S_g)$ generated by all $k$th powers of Dehn twists is a right-angled Artin group \cite{funar}, but this is false for $g \geq 4$ since $\hat g(T_g^k) = 1$.

As a consequence of Corollary~\ref{cor:raag}, we see that all normal right-angled Artin subgroups of $\Mod(S_g)$ and all surface subgroups of $\Mod(S_g)$ must be like the Dahmani--Guirardel--Osin examples in that the support of every nontrivial Nielsen--Thurston component of every element must be large.  In this direction, Clay, Mangahas, and the second author of this paper have produced normal right-angled Artin groups of $\Mod(S_g)$ where the support of each element is large (but not all pseudo-Anosov as in the Dahmani--Guirardel--Osin examples) \cite{cmm}.  

\p{A conjectural sharpening of our theorem.} As mentioned, the hypothesis $\hat g(N) < g/3$ in Theorem~\ref{main:normal} is not optimal.  We conjecture that the $g/3$ can be improved to $g/2$.

\begin{conjecture}
If $N$ is a normal subgroup of $\MCG(S_g)$ with $g \geq 2 \hat g(N) + 1$ then the natural maps 
\[ \MCG(S_g) \to \Aut N \to \Comm N  \]
are isomorphisms.
\end{conjecture}

Even better, we expect that one can replace the hypothesis $\hat g(N) < g/2$ with the hypothesis that $N$ contains a nontrivial pure element with a component whose support takes up less than half of $S_g$ in the sense that it is homeomorphic to a proper subsurface of its complement.  

\p{A conjectural dichotomy.} Combining our Theorem~\ref{main:normal} with the exotic subgroups produced by Dahmani--Guirardel--Osin and those constructed by Clay, Mangahas, and the second author, we are led to a conjectural dichotomy for normal subgroups of the mapping class group.  

\begin{conjecture}
\label{conj:dichotomy}
Let $N$ be a normal subgroup of either $\Mod(S_g)$ or $\MCG(S_g)$.  Then either $\Aut N \cong \MCG(S_g)$ or $N$ contains an infinitely generated right-angled Artin group with finite index.  
\end{conjecture}

A further conjecture of Clay, Mangahas, and the second author \cite{cmm} is that if $N$ is a normal, right-angled Artin subgroup of $\Mod(S_g)$ then $N$ is isomorphic to a free product of groups from the following list:
\[ F_\infty \, , \qquad \displaystyle\BigFreeProd_\infty \left(F_\infty \times F_\infty\right)\, , \qquad \displaystyle\BigFreeProd_\infty \left(F_\infty \times \Z\right)\, ,\quad  \text{ and } \quad \displaystyle\BigFreeProd_\infty \left(F_\infty \times F_\infty \times \Z\right).\]

\p{Mapping class groups versus lattices} Ivanov's original motivation for the study of the automorphism group of the complex of curves stems from the analogous work about lattices and arithmetic groups.  For example, the fundamental theorem of projective geometry states that  for $k$ a field and $n \geq 3$ the automorphism group of the Tits building for $k^n$ (the poset of nontrivial proper subspaces) is the group of projective semilinear automorphisms of $k^n$ (see \cite{ftpg}).

Ivanov's work on abstract commensurators also is inspired by the theory of lattices.  By the work of Margulis, an irreducible lattice in a connected semisimple noncompact Lie group with finite center is arithmetic if and only if it has infinite index in its abstract commensurator \cite{margulis}.  As observed by Ivanov \cite{ivanov}, this implies that mapping class groups are not arithmetic as (for most surfaces) mapping class groups have finite index in their abstract commensurator.  Since arithmetic groups are not normal subgroups of their abstract commensurators, our Theorem~\ref{main:normal} gives a new point of contrast between arithmetic groups and normal subgroups of mapping class groups.


\subsection{Results on complexes of regions}\label{sec:complex}

Our next goal is to state our results about automorphisms of simplicial complexes associated to a surface.  We begin by describing a class of simplicial complexes first defined by McCarthy and Papadopoulos \cite{mcpap}. 

\p{Complexes of regions.}  By a \emph{subsurface} of a compact surface $S$ we will always mean a compact subsurface $R$ where no component of $\partial R$ is homotopic to a point in $S$.  And (as above) a \emph{region} is a connected, non-peripheral subsurface.  A \emph{complex of regions} for $S$ is any nonempty simplicial flag complex that has vertices corresponding to homotopy classes of regions in $S$ and edges corresponding to vertices with disjoint representatives and that admits an action of $\MCG(S)$.

For the purposes of this paper, the difference between a graph and a flag complex is purely cosmetic, since the automorphism group of a flag complex is completely determined by the 1-skeleton.  In other words, we could replace complexes of regions with graphs of regions without affecting the theory.  The only difference is that we will use the term ``simplex'' instead of ``clique'' etc.  On the other hand, it might be an interesting problem to understand the topological properties of the complexes of regions as defined.

Let $\R(S)$ be the set of $\MCG(S)$-orbits of homotopy classes of regions in $S$.  For any subset $A$ of $\R(S)$ we denote the associated complex of regions by $\C_A(S)$.  One can recover traditional complexes of curves in this context by using an annulus as a proxy for a curve: if $A$ consists of all orbits of annuli, then $\C_A(S)$ is isomorphic to the usual complex of curves.

\p{Prior results}  There are several examples of complexes of regions that have been shown to have automorphism group isomorphic to the extended mapping class group: the complex of nonseparating curves by Irmak \cite{irmak}, the complex of separating curves by the authors of this paper \cite{kg}, the truncated complex of domains by McCarthy and Papadopoulos \cite{mcpap}, the arc complex by Irmak and McCarthy and by Disarlo \cite{irmakmac,disarlo}, the arc and curve complex by Korkmaz and Papadopoulos \cite{korkpap}, and the complex of strongly separating curves by Bowditch \cite{bowditch}.  

There are also more general theorems characterizing isomorphisms---and injective maps---between different complexes; see for instance the work of Aramayona \cite{aramayona}, Aramayona--Leininger \cite{al}, Bavard--Dowdall--Rafi \cite{bdr}, Birman--Broaddus--Menasco \cite{bbm}, Hern\'andez \cite{jhh2,jhh3}, Irmak \cite{irmak1,irmak2}, and Shackleton \cite{shack}.


\p{Pathologies} In spite of all of the aforementioned positive results, there are many natural complexes of regions on which $\MCG(S)$ acts but where the full group of automorphisms is much larger than $\MCG(S_g)$.  There are two immediate problems:
\begin{enumerate}
\item $\C_A(S_g)$ might be disconnected, and
\item $\C_A(S_g)$ might admit an \emph{exchange automorphism}, that is, an automorphism that interchanges two vertices and fixes all others.
\end{enumerate}
Typically, a disconnected complex of regions has automorphism group larger than $\MCG(S_g)$.  For instance, if $\C_A(S_g)$ has infinitely many isomorphic components (like the complex of curves for the torus or the complex of nonseparating curves for the multi-punctured torus) then the automorphism group contains an infinite permutation group.  

Also, an element of $\MCG(S_g)$ cannot act on $\C_A(S_g)$ by an exchange automorphism.  Indeed, if a mapping class fixes all but finitely many vertices of $\C_A(S_g)$ then it must be the identity (cf. Lemma~\ref{inj} below).  McCarthy and Papadopoulos were the first to address the issue of exchange automorphisms; they showed that the complex of domains $\C_{\R(S)}(S)$ admits exchange automorphisms when $S$ has more than one boundary component.

We would like to rule out these two types of pathologies.  First we will list two situations that give rise to exchange automorphisms---holes and corks---and later in Section~\ref{sec:suff rich} we will prove that all exchange automorphisms arise in this way.

\p{Holes and corks} Let $\C_A(S_g)$ be a complex of regions.  First, we say that a vertex $v$ of $\C_A(S_g)$ has a \emph{hole} if a representative region $R$ has a complementary region $Q$ with the property that no vertex of $\C_A(S_g)$ is represented by a subsurface of $Q$ (we refer to $Q$ as the hole).  Note that annular vertices cannot have holes.  Indeed, if $R$ represents an annular vertex, then there is an annulus parallel to $R$ in every complementary region.

Next, we say a vertex $v$ of $\C_A(S_g)$ is a \emph{cork} if (1) $v$ is represented by an annulus $A$, (2) one complementary region $R$ of $A$ represents a vertex $w$ of $\C_A(S_g)$, and (3) no proper, non-peripheral subsurface of $R$ represents a vertex of $\C_A(S_g)$.  Any such pair $\{v,w\}$ will be referred to as a \emph{cork pair}.  

\begin{figure}
\includegraphics[scale=.15]{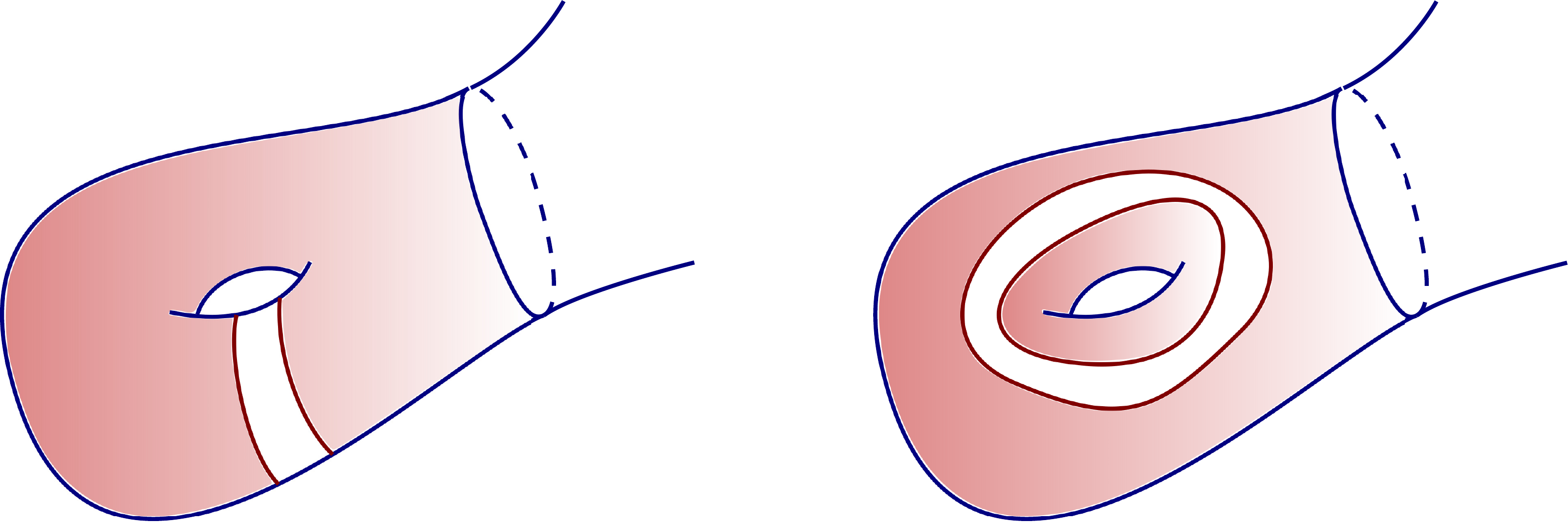}
\hspace*{.5in}
\includegraphics[scale=.15]{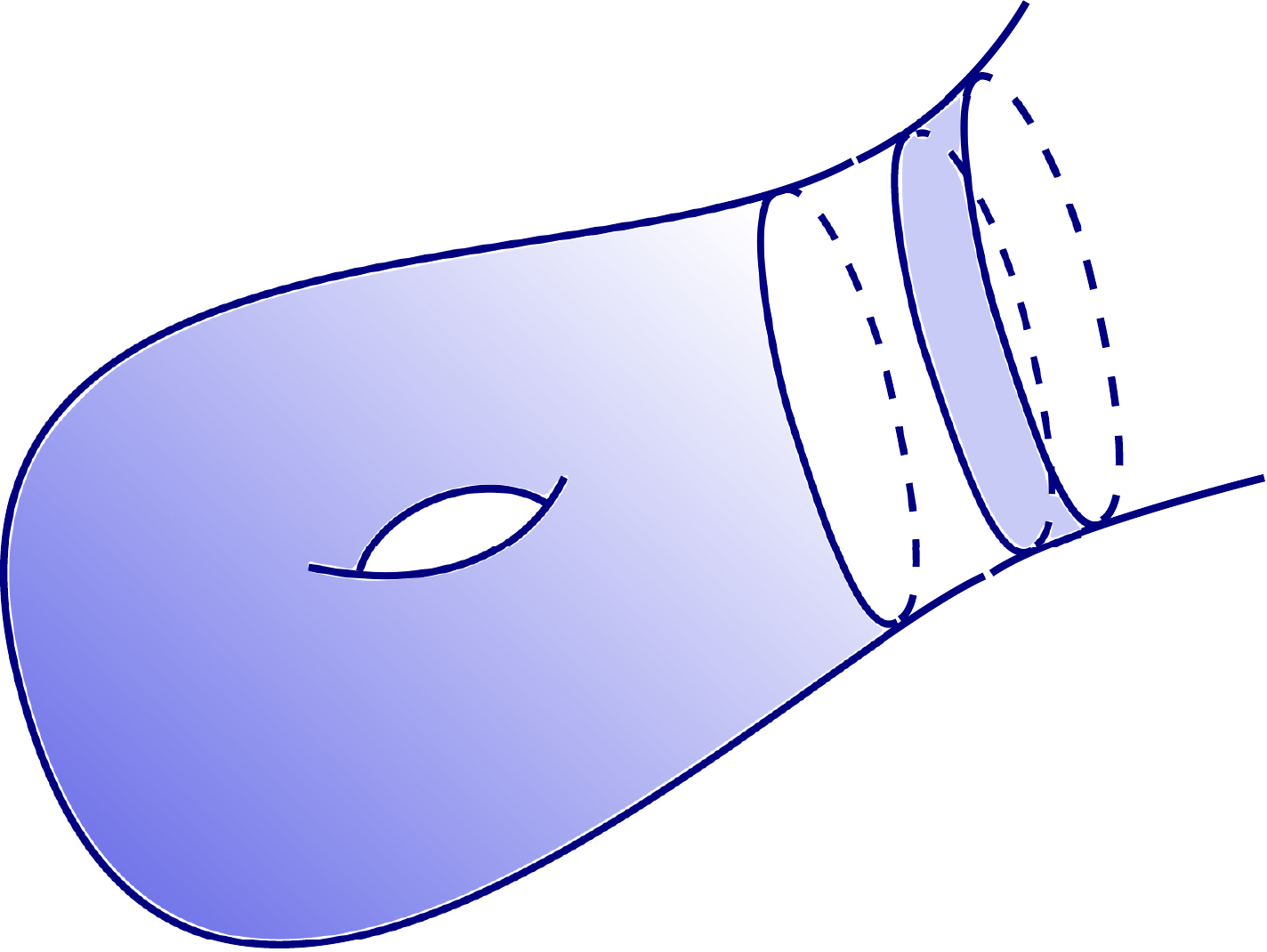}
\caption{\emph{Left:} vertices with holes; \emph{Right:} a cork pair}
\label{fig:bow}
\end{figure}

An example of a complex of regions that has a vertex with a hole is the complex whose vertices correspond to bow-legged pairs of pants, that is, pairs of pants that are embedded in such a way that two of the boundary components are parallel in the surface (the hole is the annulus between these two boundary components).  Any two vertices with representatives contained in the same handle (torus with one boundary component) can be exchanged by an automorphism of this complex; see the left-hand side of Figure~\ref{fig:bow}.  

An example of a complex of regions with a cork is the complex of regions that includes all regions except the nonseparating annuli and the bow-legged pairs of pants; in this case the corks are the separating annuli that cut off a handle.  These vertices can be exchanged with the vertices corresponding to the handles they cut off.

In Section~\ref{sec:suff rich} below we will show that a connected complex of regions admits an exchange automorphism if and only if it has a hole or a cork (meaning that it has a vertex with a hole or a vertex that is a cork).  

\p{Statement of the main theorem about complexes} For a subsurface $R$ of $S_g$ we define $\bar g(R)$ to be the smallest number $k$ so that $R$ is contained in a subsurface of $S_g$ of genus $k$ with connected boundary (we allow for the possibility that $\bar g(R) = g$).  We define $\bar g(A)$ to be the minimum of $\bar g(R)$ where $R$ represents an element of $A$.  The definitions of $\hat g(f)$ above and $\bar g(R)$ here are similar in spirit, although an important difference is that in the present case $R$ is not required to be non-peripheral in the subsurface of genus $k$.  As such, we use different notations to avoid confusion.

\begin{theorem}
\label{main:complex}
Let $\C_A(S_g)$ be a complex of regions that is connected and has no holes or corks and assume that $g \geq 3\bar g(A) + 1$.  Then the  natural map
\[ \MCG(S_g) \to \Aut \C_A(S_g) \]
is an isomorphism.
\end{theorem}

Again the hypothesis on $g$ is not sharp.  In the case where $\C_A(S_g)$ is the complex of curves, Theorem~\ref{main:complex} says that $\Aut \C_A(S_g) \cong \MCG(S_g)$ when $g \geq 4$, while on the other hand this isomorphism is known to hold for $g \geq 3$.  

\p{Applications.} Theorem~\ref{main:complex} is the first result to address infinitely many distinct complexes with a unified argument.  It covers many of the previously studied examples of simplicial complexes with automorphism group $\MCG(S_g)$, such as the complex of curves, the complex of separating curves, the complex of nonseparating curves, and the Bridson--Pettet--Souto complex of four-holed spheres and two-holed tori.  It is also easy to construct new examples, such as the complex of handles, the complex of separating curves of odd genus, and the complex of nonseparating seven-holed tori, etc.

\p{Automorphisms in the presence of holes and corks.} In the case where $\C_A(S_g)$ has a hole or a cork, we can still describe its automorphism group.  Let $\Exchange \C_A(S_g)$ denote the normal subgroup of $\Aut \C_A(S_g)$ generated by all exchange automorphisms.  

\begin{theorem}
\label{ex thm}
Let $g \geq 3$.  Let $\C_A(S_g)$ be a complex of regions that is connected and satisfies $g \geq 3 \bar g(A)+1$.  Then
\[
\Aut \C_A(S_g) \cong \Exchange \C_A(S_g) \rtimes \MCG(S_g). 
\]
\end{theorem}

We will prove Theorem~\ref{ex thm} in Section~\ref{sec:suff rich}.  McCarthy and Papadopoulos proved Theorem~\ref{ex thm} in the case where $\C_A(S_g)$ is the complex of domains \cite[Theorem 8.9]{mcpap}.

\p{A conjectural sharpening of the theorem.} We conjecture that the condition on $\hat g(A)$ in Theorem~\ref{main:complex} is not necessary.

\begin{conjecture}
\label{crconj}
Let $\C_A(S_g)$ be a complex of regions that is connected and has no holes or corks.  Then the natural map
\[ \MCG(S_g) \to \Aut \C_A(S_g) \]
is an isomorphism.
\end{conjecture}

\p{Further possible generalizations} There are other ways that one might extend our Theorem~\ref{main:complex}, for instance by generalizing the definition of a complex of regions.  There are several examples of simplicial complexes that do not satisfy our definition of a complex of regions but still have automorphism group isomorphic to the extended mapping class group.  For example, the systolic complex of curves, studied by Schmutz--Schaller \cite{pss} has edges that do not correspond to disjointness.  The Torelli complex, studied by Farb and Ivanov \cite{farbivanov} (see also \cite{kg}), has vertices corresponding to disconnected subsurfaces.  And the pants complex, studied by the second author \cite{pants}, has both deficiencies: its vertices correspond to disconnected subsurfaces and its edges do not correspond to disjointness.  On the other hand, all of these complexes have automorphism group the extended mapping class group.  

Other natural directions are to study the analogs for punctured surfaces, surfaces of infinite type, and other manifolds.  Work in these directions has already been done by McLeay \cite{alan,alan2} and Scott \cite{shane}.


\subsection{Plan of the paper} We now give a summary of the remaining five sections of the paper.  Along the way, we explain how the various sections fit together to prove our main results.  

\p{Exchange automorphisms} As mentioned, Section~\ref{sec:suff rich} is devoted to the classification of exchange automorphisms  and to the determination of the automorphism group of a connected complex with exchange automorphisms.  More precisely, we prove Theorem~\ref{theorem:suff rich} below, which states that all exchange automorphisms arise from holes and corks, and we prove Theorem~\ref{ex thm} above, which gives a semi-direct product decomposition of the group of automorphisms into the extended mapping class group and the group of exchange automorphisms.

\p{Complexes of separating curves} In Section~\ref{sec:sep} we extend previous work of the authors on the complex of separating curves. This section is the only part of the paper that closely parallels earlier work in the subject.

Recall that the complex of curves $\C(S_g)$ is the simplicial flag complex with vertices corresponding to homotopy classes of simple closed curves in $S_g$ and edges connecting vertices with disjoint representatives.  Vertices of $\C(S_g)$ can be separating or nonseparating, meaning that they have a representative that is such.

The \emph{genus} of a separating curve is the minimum of the genera of the two complementary regions of $S_g$.  Let $\C_k(S_g)$ be the subcomplex of $\C(S_g)$ spanned by all vertices represented by separating curves of genus at least $k$.  The first ingredient in our proof of Theorem~\ref{main:complex} is the following.

\begin{theorem}
\label{theorem:sep k}
Let $k \geq 1$ and let $g \geq 3k+1$. The natural map
\[ \MCG(S_g) \to \Aut(\C_k(S_g)) \]
is an isomorphism.
\end{theorem}

We prove Theorem~\ref{theorem:sep k} in Section~\ref{sec:sep}.  The proof proceeds by induction on $k$.  The base case is $k=1$, in which case $\C_k(S_g)$ is the complex of separating curves.  This case was proved in our earlier work \cite{kg,kgadd}.  

The bounds on genus in Theorems~\ref{main:normal} and~\ref{main:complex} are derived from the bound on $g$ in Theorem~\ref{theorem:sep k}.  If the bound here can be improved, one would obtain improved versions of those theorems.

\p{Complex of dividing sets} In Section~\ref{sec:div} we apply Theorem~\ref{theorem:sep k} in order to determine the automorphism group of a different complex, which is of a different nature and requires a specialized set of tools and techniques. To state the theorem, we require some definitions.

A \emph{dividing set} in $S_g$ is a disjoint union of essential simple closed curves that divides $S_g$ into exactly two regions in such a way that each curve lies in the boundary of both regions.  We allow for the possibility that one of the two regions associated to a dividing set is an annulus.  We say that two dividing sets are \emph{nested} if one is contained entirely in a single region defined by the other.

Let $\D(S_g)$ denote the set of $\MCG(S_g)$-orbits of isotopy classes of dividing sets in $S_g$.  For any $D \subseteq \D(S_g)$ we define $\C_D(S_g)$ to be the abstract simplicial flag complex whose vertices correspond to isotopy classes of dividing sets in $S_g$ representing elements of $D$ and whose vertices are connected by an edge when they have nested representatives.  

We define a partial order on $\D(S_g)$ as follows: we say that $a \preceq b$ if $a$ and $b$ have nested representatives $A$ and $B$ and, of the two regions of $S_g$ complementary to $B$, the dividing set $A$ lies in one with minimal genus.  

For any set $X$ with a partial order, an \emph{upper set} is a subset $Y \subseteq X$ with the property that $y \in Y$ and $y \preceq z$ implies $z \in Y$.

Finally, we define $\check g(D)$ to be the minimum of the genera of the separating curves corresponding to elements of $D$.

\begin{theorem}
\label{theorem:multi k}
Fix a nonempty upper set $D \subseteq \D(S_g)$ so that $\C_D(S_g)$ is connected and assume that $g \geq 3 \check g(D)+1$.  Then the natural map 
\[ \MCG(S_g) \to \Aut \C_D(S_g) \]
is an isomorphism.
\end{theorem}

We prove Theorem~\ref{theorem:multi k} in Section~\ref{sec:div}.  The first observation is that since $D$ is an upper set the complex $\C_D(S_g)$ has a subcomplex isomorphic to $\C_{\check g(D)}(S_g)$.  The proof then proceeds by showing that an automorphism of $\C_{D}(S_g)$ induces an automorphism of $\C_{\check g(D)}(S_g)$ and then applying Theorem~\ref{theorem:sep k}.

The proof of Theorem~\ref{theorem:multi k} is more subtle than the previous theorems about automorphisms of curve complexes.  The first major distinction is that edges in $\C_D(S_g)$ do not correspond to disjointness, and so the usual arguments do not apply.  On top of this, the hypotheses of the theorem do not specify which dividing sets do, and do not, correspond to vertices of $\C_D(S_g)$; they only specify that $D$ is an upper set (compare this with the hypotheses of our Theorem~\ref{main:normal} about normal subgroups of $\Mod(S_g)$).

\p{Complexes of regions}  In Section~\ref{sec:reg} we derive Theorem~\ref{main:complex} from Theorem~\ref{theorem:multi k}.  The starting point is a correspondence
\[
\left\{  \text{maximal joins in } \C_A(S_g) \right\} \longleftrightarrow \left\{  \text{vertices of } \C_D(S_g) \right\}.
\]
We use this correspondence to show that an automorphism of $\C_A(S_g)$ induces an automorphism of some $\C_D(S_g)$, and then apply Theorem~\ref{theorem:multi k}.  The main work is in showing that the induced map $\Aut \C_A(S_g) \to \Aut \C_D(S_g)$ is injective.  Again a difficulty is that we do not have an explicit list of vertices of the complex.  

\p{Normal subgroups} In Section~\ref{section:normal} we prove Theorem~\ref{main:normal} using Theorem~\ref{main:complex}.  The core idea is that to a normal subgroup $N$ of $\Mod(S_g)$ or $\MCG(S_g)$ we associate a complex of regions whose vertices correspond to the supports of certain subgroups of $N$.  To this end, we introduce the notion of a basic subgroup of $N$, which is a non-abelian subgroup of $N$ whose centralizer in $N$ is maximal among non-abelian subgroups of $N$.

We consider the complex of regions whose vertices are the supports of the basic subgroups of $N$.  Since $N$ has a pure element with a small component, the complex of regions has a small vertex.    This construction enables us to extract the necessary topological data from $N$ without knowing any specific information about its elements.  After possibly modifying the complex of regions so that it satisfies the other hypotheses of Theorem~\ref{main:complex}, we then show that an automorphism of $N$ gives rise to an automorphism of the complex of regions and then apply Theorem~\ref{main:complex}.

The proof then proceeds by analyzing separately the case where $N$ is normal in $\MCG(S_g)$ and the (harder) case where $N$ is normal in $\Mod(S_g)$.  

\p{Bird's-eye view of the proof} One interpretation of our proofs of Theorems~\ref{main:normal} and~\ref{main:complex} is that there is a sequence of maps 
\begin{align*}
 \Aut N \to \Aut &\ \C_A(S_g) \to \Aut \C_{D}(S_g) \to \Aut \C_k(S_g) \to \Aut \C_{k-1}(S_g) \to \\
  & \cdots \to \Aut \C_1(S_g) \to \Aut \C(S_g) \to \MCG(S_g)
\end{align*}
and that the appropriate compositions are inverse to the natural maps $\MCG(S_g) \to \Aut N$ and $\MCG(S_g) \to \Aut \C_A(S_g)$. 
Therefore, not only does our main theorem validate Ivanov's metaconjecture, but by going through Ivanov's original theorem the proof does as well.


\p{Acknowledgments} This project was begun in conjunction with the Mathematical Research Communities program on Geometric Group Theory in 2013.  We are grateful to the American Mathematical Society, the organizers, and the participants there, in particular Matthew Durham and Brian Mann, for helpful conversations and inspiration.  We would like to thank Justin Lanier, Alan McLeay, and Shane Scott for their comments on earlier drafts, and Martin Bridson, Lei Chen, Tom Church, Luis Paris, Alexandra Pettet, Juan Souto, Masaaki Suzuki, and Daniel Studenmud for helpful conversations.  We would like to thank Kevin Wortman for detailed discussions about arithmetic groups.   We are also grateful to Joan Birman, Benson Farb, and Andrew Putman for their encouragement on this project.  Finally we would like to express our deepest gratitude to Kathleen Margalit and Brendan Owens for their support throughout this project.


\section{Exchange automorphisms}
\label{sec:suff rich}

Before we proceed to the proofs of our main theorems, we prove two theorems that clarify the role that exchange automorphisms play in the theory of automorphisms of complexes of regions.  Theorem~\ref{theorem:suff rich} gives a complete characterization of exchange automorphisms in complexes of regions.  Theorem~\ref{ex thm} then describes the automorphism group of a connected complex of regions that has exchange automorphisms.

\subsection{Characterization of exchange automorphisms}\label{sec:sr} To state the first theorem about exchange automorphisms, we require a definition.  Let $\C_A(S_g)$ be a complex of regions, and let $v$ be a non-annular vertex.  Let $R$ be a representative of $v$ and let $Q_1,\dots,Q_n$ be the set of complementary regions of $R$ that do not contain representatives of vertices of $\C_A(S_g)$ (the $Q_i$ are the holes of $v$).  The \emph{filling} of $v$ is the homotopy class of regions represented by $R' = R \cup Q_1 \cup \cdots \cup Q_n$.  When $v$ has a hole (so $n > 0$), there are infinitely many vertices of $\C_A(S_g)$ with the same filling; these are the translates of $v$ under the elements of $\MCG(S_g)$ that preserve $R'$.  For convenience, we define the filling of an annular vertex $v$ to be $v$ itself.

\begin{theorem}
\label{theorem:suff rich}
Let $g \geq 3$.  Let $\C_A(S_g)$ be a complex of regions with no isolated vertices or edges.  Then $\C_A(S_g)$ admits an exchange automorphism if and only if it has a hole or a cork.  Moreover, two vertices can be interchanged by an exchange automorphism if and only if they are non-annular vertices with equal fillings or they form a cork pair.
\end{theorem}

\begin{proof}

The first statement follows from the second statement and the fact that if a vertex has a hole then there is another vertex (in fact infinitely many) with the same filling.  Thus, it suffices to prove the latter statement. 

Suppose first that $v$ and $w$ form a cork pair in $\C_A(S_g)$, and say that $v$ is the annular vertex in the pair.  Clearly any vertex connected by an edge to $w$ must also be connected to $v$.  If there were a vertex of $\C_A(S_g)$ that was connected to $v$ but not $w$, then it would be represented by a subsurface of a representative of $w$.  By the definition of a cork, no such vertex exists.  Thus, the stars of $v$ and $w$ in $\C_A(S_g)$ are equal and so the two vertices can be interchanged by an exchange automorphism.

Now suppose that $v$ is a vertex with a hole and that $v$ and $w$ have equal fillings.  Say that $R$ is a representative of the filling.  By the definition of a filling, any vertex of $\C_A(S_g)$ that is connected to $v$ by an edge must be represented in the complement of $R$.  But then this vertex is also connected to $w$ by an edge.  It then follows that $v$ and $w$ have equal links and can be interchanged by an exchange automorphism.

For the other direction of the theorem, we will show that two vertices of $\C_A(S_g)$ either have equal fillings, form a cork pair, or cannot be interchanged by an exchange automorphism.  Let $v$ and $w$ be two vertices of $\C_A(S_g)$ and say they are represented by regions $P$ and $Q$.  First we treat the case where $v$ and $w$ are connected by an edge, so $P$ and $Q$ are disjoint.   There are three subcases.  

The first subcase is where there is a component of the boundary of $P$ that is not parallel to the boundary of $Q$.  In this case there is a region $R$ of $S_g$ that contains $P$ as a proper, non-peripheral subsurface and is disjoint from $Q$.  From this it follows that there are $\MCG(S_g)$-translates of $v$ that are connected by an edge to $w$ but not to $v$.  Therefore $v$ and $w$ cannot be exchanged.

The second subcase is where $P$ and $Q$ are complementary regions in $S_g$.  If there is a vertex of $\C_A(S_g)$ corresponding to a proper, non-peripheral subsurface of either $P$ or $Q$ then this vertex would be connected by an edge to one of $v$ or $w$ but not the other, and we would have that $v$ and $w$ were not exchangeable.  So we may assume there is no such vertex.  It follows that $P$ and $Q$ are homeomorphic.  We may also assume that there is a vertex $u$ of $\C_A(S_g)$ corresponding to a component of the boundary of $P$, for otherwise $v$ and $w$ would span an isolated edge in $\C_A(S_g)$.  If $u$ is a nonseparating annular vertex then since there are no vertices of $\C_A(S_g)$ represented by proper, non-peripheral subsurfaces of $P$ or $Q$, it follows that $P$ and $Q$ are pairs of pants and so $g=2$, a contradiction.  If $u$ is a separating annular vertex, then $u$ and $v$ form a cork pair.  

The third and final subcase is where $P$ has connected boundary and $Q$ is an annulus parallel to the boundary of $P$.  If there is a vertex $u$ of $\C_A(S_g)$ represented by a proper, essential subsurface of $P$ then $v$ and $w$ cannot be exchanged ($u$ is connected by an edge to $w$ but not $v$).  If there is no such vertex $u$ then $v$ and $w$ form a cork pair.  

We now proceed to the case where $v$ and $w$ are not connected by an edge.  Here there are two subcases, according to whether or not one of the two vertices is annular.  Again let $P$ and $Q$ be representatives of $v$ and $w$; in this case, $P$ and $Q$ have essential intersection.

The first subcase is where $P$ is annular.  The complement of $P$ in $S_g$ consists of either one or two regions; the region $Q$ has essential intersection with each such region.  Since $v$ is not an isolated vertex of $\C_A(S_g)$, there is a vertex $u$ of $\C_A(S_g)$ represented in a region of $S_g$ complementary to $P$.  Some $\MCG(S_g)$-translate of $u$ is connected by an edge to $v$ but not $w$, and so $v$ and $w$ are not exchangeable.

The second and final subcase is where neither $P$ or $Q$ is annular.  Since $v$ and $w$ are distinct, there must be (after possibly renaming the vertices) a complementary region $R$ of $P$ that has essential intersection with $Q$.  If we assume that $v$ and $w$ are exchangeable, then $R$ must represent a hole for $P$ (otherwise we would find a vertex connected by an edge to $v$ but not $w$).  It follows that after filling $v$ and $w$, they do not intersect each other's complementary regions.  In other words, $v$ and $w$ have equal fillings.  This completes the proof.
\end{proof}

\subsection{Automorphism groups in the presence of exchange automorphisms} In this section we prove Theorem~\ref{ex thm} which describes the group of automorphisms of a general complex of regions, possibly with holes and corks.

We will require a sequence of lemmas about fillings (these lemmas will also be used in Section~\ref{section:normal}).  In what follows, the \emph{filling} of a complex of regions $\C_A(S_g)$ is the complex of regions $\C_{\bar A}(S_g)$ obtained by replacing each vertex of $\C_A(S_g)$ with its filling.

\begin{lemma}
\label{filling}
Let $\C_A(S_g)$ be a complex of regions. Then its filling $\C_{\bar A}(S_g)$ has no holes.
\end{lemma}

\begin{proof}

Let $R$ be a non-annular region of $S_g$ representing a vertex of $\C_{\bar A}(S_g)$.  Since the vertices of $\C_{\bar A}(S_g)$ are fillings of vertices of $\C_A(S_g)$, it follows that there is a non-annular region $Q$ in $R$ that represents a vertex of $\C_A(S_g)$ and so that the $R$-vertex of $\C_{\bar A}(S_g)$  is the filling of the $Q$-vertex of $\C_A(S_g)$.  Denote the complementary regions of $Q$ in $S_g$ by $P_1,\dots,P_n$ and say that $P_{m+1},\dots,P_n$ are the complementary regions corresponding to holes (i.e. the regions that do not support any vertex of $\C_A(S_g)$).  Then $R$ is represented by the union of $Q$ with $P_{m+1} \cup \cdots \cup P_n$ and so the complementary regions to $R$ are $P_{1},\dots,P_n$.  By assumption each of these regions supports a vertex of $\C_A(S_g)$.   

We must show that each of $P_1,\dots,P_n$ supports a vertex of $\C_{\bar A}(S_g)$.  Let $Q_{1}$ be a region in $P_{1}$ representing a vertex of $\C_A(S_g)$.  If $Q_1$ is annular then the filling of the $Q_1$-vertex is itself and there is nothing to do.  So we may assume $Q_1$ is not annular.  Since the complement of $P_1$ is connected, there is a single complementary region of $Q_{1}$ containing the complement of $P_1$.  Since this complementary region contains $Q$ it cannot represent a hole for $Q_{1}$.  Thus the filling of the $Q_{1}$-vertex is contained in $P_1$ as desired.  The same argument applies to $P_2,\dots,P_m$, and we have finished showing that $\C_{\bar A}(S_g)$ has no holes. 
\end{proof}

\begin{lemma}
\label{small filling}
Let $\C_A(S_g)$ be a complex of regions with a small vertex. Then its filling $\C_{\bar A}(S_g)$ has a small vertex.
\end{lemma}

\begin{proof}

Let $v$ be a small vertex of $\C_A(S_g)$.  Let $R$ be a subsurface of genus less than $g/3$ and with connected boundary that contains a representative of $v$.  Let $Q$ denote the region complementary to $R$.  Note that $Q$ lies in a single complementary region for (a representative of) $v$.  Also note that $Q$ does not lie in a hole for $v$ since there are $\MCG(S_g)$-translates of $v$ that are represented in $Q$.  Therefore the filling of $v$ is represented in $R$.  The lemma follows.
\end{proof}

We are ready now for the proof of Theorem~\ref{ex thm}.  In the proof we will refer to the set of vertices of $\C_A(S_g)$ with a given filling as an \emph{equal filling set}.  As mentioned, if an equal filling set has more than one element, then it has infinitely many.

\begin{proof}[Proof of Theorem~\ref{ex thm}]

Let $\C_{\bar A}(S_g)$ be the filling of $\C_A(S_g)$ and let $\C_{\bar A'}(S_g)$ be the complex of regions obtained from $\C_{\bar A}(S_g)$ by removing all corks.   

The proof is divided into two parts.  The first part is to show that $\C_{\bar A'}(S_g)$ satisfies the hypotheses of Theorem~\ref{main:complex} and hence has automorphism group $\MCG(S_g)$.  The second part to show that there is a split short exact sequence
\[ 
1 \to \Exchange \C_A(S_g) \to \Aut \C_A(S_g) \to \Aut \C_{\bar A'}(S_g) \to 1.
\]

For the first part of the proof there are four steps, namely, to show that $\C_{\bar A'}(S_g)$ is connected, that is has a small vertex, that it has no corks, and that it has no holes.  We treat these four steps in order.  

The first step is to show that $\C_{\bar A'}(S_g)$ is connected.  There are natural maps
\[
\C_A^{(0)}(S_g) \to  \C_{\bar A}^{(0)}(S_g) \to  \C_{\bar A'}^{(0)}(S_g)
\]
defined as follows.  Under the first map each vertex of $\C_A(S_g)$ is sent to its filling.  Under the second map, each vertex that is not a cork is sent to itself and each cork is assigned to have the same image as the vertex for which is it a cork.  

We would like to show that both maps extend to simplicial maps of $\C_A(S_g)$ and $\C_{\bar A}(S_g)$, respectively.  For the first map, assume that $v$ and $w$ are vertices of $\C_A(S_g)$ that are connected by an edge.   By the definition of a hole, neither $v$ nor $w$ can lie in a hole for the other, and it follows that each is disjoint from the other's filling.  Thus, the images of $v$ and $w$ are connected by an edge, and so the map indeed extends.

For the second map, let us assume that $v$ and $w$ are vertices of $\C_{\bar A}(S_g)$ that are connected by an edge.  The only nontrivial case is where $v$ is a cork.  In this case $w$, being disjoint from $v$, either is equal to or is connected by an edge to the vertex that forms a cork pair with $v$, as desired.  Since the image of a connected complex under a simplicial map is connected, we have thus proven that $\C_{\bar A'}(S_g)$ is connected. 

The second step is to show that $\C_{\bar A'}(S_g)$ has a small vertex.  By assumption $\C_A(S_g)$ has a small vertex, and hence by Lemma~\ref{small filling} the filling $\C_{\bar A}(S_g)$ has a small vertex; call it $v$.  If $v$ does not represent a cork in $\C_{\bar A}(S_g)$ then it survives in $\C_{\bar A'}(S_g)$ and is the desired small vertex.  If $v$ does represent a cork in $\C_{\bar A}(S_g)$, then there is a vertex $w$ of $\C_{\bar A}(S_g)$ that it forms a cork pair with.  This $w$ is represented in $R$ and is the desired small vertex of $\C_{\bar A'}(S_g)$.

The third step is to show that $\C_{\bar A'}(S_g)$ has no corks.  Suppose for contradiction that $\C_{\bar A'}(S_g)$ had a cork pair corresponding to the non-annular region $R$ and the annular region $Q$.  All regions representing vertices of $\C_{\bar A'}(S_g)$ also represent vertices of the intermediate complex $\C_{\bar A}(S_g)$; in particular $R$ and $Q$ do.  By the definition of a cork pair, there is no vertex of $\C_{\bar A'}(S_g)$ represented by a non-peripheral, proper subsurface of $R$.  But since $R$ and $Q$ do not represent a cork pair for $\C_{\bar A}(S_g)$ (otherwise the $Q$-vertex would have been removed), there must be a vertex of $\C_{\bar A}(S_g)$ represented by a non-peripheral, proper subsurface $P$ of $R$.  Since $P$ does not represent a vertex of $\C_{\bar A'}(S_g)$, it must represent a cork in $\C_{\bar A}(S_g)$.  But then there is a region $P'$ so that the $P$- and $P'$-vertices form a cork pair in $\C_{\bar A}(S_g)$.  The region $P'$ must be contained in $R$, for otherwise the annular region $Q$ would prevent $P$ from being a cork.  The region $P'$ is further proper and non-peripheral in $R$, contradicting our assumption about $R$.  This completes the proof that $\C_{\bar A'}(S_g)$ has no corks.

The fourth step is to show that $\C_{\bar A'}(S_g)$ has no holes.   Suppose that $R$ is a non-annular region of $S_g$ representing a vertex of $\C_{\bar A'}(S_g)$.  Then $R$ also represents a vertex of $\C_{\bar A}(S_g)$.  Let $Q$ be a region of $S_g$ complementary to $R$.  We would like to show that $Q$ does not represent a hole for the $R$-vertex of  $\C_{\bar A'}(S_g)$.  By Lemma~\ref{filling} the complex $\C_{\bar A}(S_g)$ has no holes. Thus there is a region $P$ in $Q$ representing a vertex of $\C_{\bar A}(S_g)$.  If $P$ does not represent a cork for $\C_{\bar A}(S_g)$ then $P$ also represents a vertex of $\C_{\bar A'}(S_g)$ and we are done.  Now suppose that $P$ does represent a cork for $\C_{\bar A}(S_g)$.  If the $R$-vertex of $\C_{\bar A}(S_g)$ forms a cork pair with the $P$-vertex, then since $\C_{\bar A'}(S_g)$ has a small vertex the genus of $R$ is less than $g/3$.  As such, the complementary region $Q$ contains a $\MCG(S_g)$-translate of the $R$-vertex of $\C_{\bar A'}(S_g)$ and again we are done.  Finally, we are in the situation where $P$ represents a cork for $\C_{\bar A}(S_g)$ but $R$ does not represent the other vertex in the cork pair; say that $P'$ represents the other vertex of the cork pair.  By the definition of a cork, $R$ is not contained in $P'$.  The annulus $P$ only has two complementary regions, one of which is $P'$.  So it must be that $P'$ is the complementary region to $P$ not containing $R$.  It follows that $P'$ is contained in $Q$.  Hence $Q$ does not represent a hole for the $R$-vertex of $\C_{\bar A'}(S_g)$ and so $\C_{\bar A'}(S_g)$ has no holes. 

We have shown that $\C_{\bar A'}(S_g)$ is connected, has a small vertex, and has no holes or corks.  By Theorem~\ref{main:complex} we have
\[
\Aut \C_{\bar A'}(S_g) \cong \MCG(S_g).
\]

We now proceed to the second part of the proof, which is to show that there is a short exact sequence
\[ 
1 \to \Exchange \C_A(S_g) \to \Aut \C_A(S_g) \to \Aut \C_{\bar A'}(S_g) \to 1.
\]
We treat three steps in turn, namely, to show that there is a well-defined map $\Aut \C_A(S_g) \to \Aut \C_{\bar A'}(S_g)$, to show that this map has a right inverse, and then to show that the kernel is $\Exchange \C_A(S_g)$.

We begin with the first step, which is to show that there is a well-defined map $\Aut \C_A(S_g) \to \Aut \C_{\bar A'}(S_g)$.  Above we described simplicial maps 
\[
\C_A(S_g) \to  \C_{\bar A}(S_g) \to  \C_{\bar A'}(S_g).
\]
We would like to show that these simplicial maps induce well-defined maps
\[ 
\Aut \C_A(S_g) \to \Aut \C_{\bar A}(S_g) \to \Aut \C_{\bar A'}(S_g).
\]
For the first map we need to show that the image of an equal filling set under an automorphism of $\C_A(S_g)$ is another equal filling set.  But this is true by Theorem~\ref{theorem:suff rich}, which implies that a collection of vertices of $\C_A(S_g)$ is an equal filling set if and only if it is either a singleton or an infinite set on which $\Exchange \C_A(S_g)$ acts transitively.

For the second map we need to show that $\Aut \C_{\bar A}(S_g)$ preserves the set of cork pairs.  But again by Theorem~\ref{theorem:suff rich} a collection of vertices of $\C_{\bar A}(S_g)$ is a cork pair if and only if it is a set with two elements upon which $\Exchange \C_{\bar A}(S_g)$ acts transitively.

We will now show that the map $\Aut \C_A(S_g) \to \Aut \C_{\bar A'}(S_g)$ has a right inverse.  We already showed in the first part of the proof that the map $\MCG(S_g) \to \Aut \C_{\bar A'}(S_g)$ is an isomorphism and so the natural homomoprhism $\MCG(S_g) \to \Aut \C_A(S_g)$ is a candidate for a right inverse.  Because the processes of filling vertices and removing corks both commute with the action of $\MCG(S_g)$, this indeed gives the desired right inverse.

To complete the second part of the proof---and hence the theorem---it remains to show that the kernel of the composition $\Aut \C_A(S_g) \to \Aut \C_{\bar A'}(S_g)$ is indeed the group $\Exchange \C_A(S_g)$.  The kernel of the first map $\Aut \C_A(S_g) \to \Aut \C_{\bar A}(S_g)$ is clearly the subgroup of $\Exchange \C_A(S_g)$ generated by the permutations of equal filling sets.  Since $\C_{\bar A}(S_g)$ has no holes the kernel of the second map $\Aut \C_{\bar A}(S_g) \to \Aut \C_{\bar A'}(S_g)$ is equal to the subgroup of  $\Exchange \C_{\bar A}(S_g)$ generated by cork pair swaps and by Theorem~\ref{theorem:suff rich} this is all of $\Exchange \C_{\bar A}(S_g)$.

The statement that the kernel of $\Aut \C_A(S_g) \to \Aut \C_{\bar A'}(S_g)$ is $\Exchange \C_A(S_g)$ now amounts to the following statement: if $v$ and $w$ are vertices of $\C_A(S_g)$ whose images $\bar v$ and $\bar w$ in $C_{\bar A}(S_g)$ form a cork pair---say that $\bar v$ is the cork---then either $v$ and $w$ form a cork pair or they are not exchangeable.  Since $\bar v$ is annular, it follows that $v$ is represented by the same annulus.  Also, the filling of $w$ corresponds to $\bar w$.  If $w$ and $\bar w$  are represented by the same region, then $v$ and $w$ form a cork.  If not, then $w$ has a nontrivial filling, and there are infinitely many other vertices of $\C_A(S_g)$ with the same filling.  Each of these vertices is connected to $v$ by an edge but not to $w$.  It follows that $v$ and $w$ are not exchangeable.  This completes the proof.
\end{proof}


\section{Complexes of separating curves}
\label{sec:sep}

In this section we prove Theorem~\ref{theorem:sep k}. The ideas in this section build on our earlier work.  Indeed, the main tool in this section is a certain configuration of curves called a sharing pair, which was introduced in our previous paper \cite{kg}.  Our approach here puts our earlier work into a more conceptual framework, allowing for the more general argument.

As discussed in the introduction, Theorem~\ref{theorem:sep k} is proven by induction with the base case $k=1$ (again the base case was proved in our earlier work \cite{kg,kgadd}).  The main point of the inductive step is to show that there is a map
\begin{align*}
\Aut \C_k(S_g) &\to \Aut \C_{k-1}(S_g) \\
\phi &\mapsto \hat \phi
\end{align*}
so that each element $\phi$ of $\Aut \C_k(S_g)$ is the restriction to $\C_k(S_g)$ of its image $\hat \phi$ (we regard $\C_k(S_g)$ as a subcomplex of $\C_{k-1}(S_g)$).  Given an automorphism $\phi$ of $\C_k(S_g)$, we thus need to specify the action of $\hat \phi$ on the vertices of $\C_{k-1}(S_g)$ of genus $k-1$.

\begin{figure}
\begin{minipage}{5in}
  \centering
  $\vcenter{\hbox{
  \labellist
\small\hair 2pt
 \pinlabel {$k-1$} [ ] at 110 30
 \endlabellist  
  \includegraphics[scale=.35]{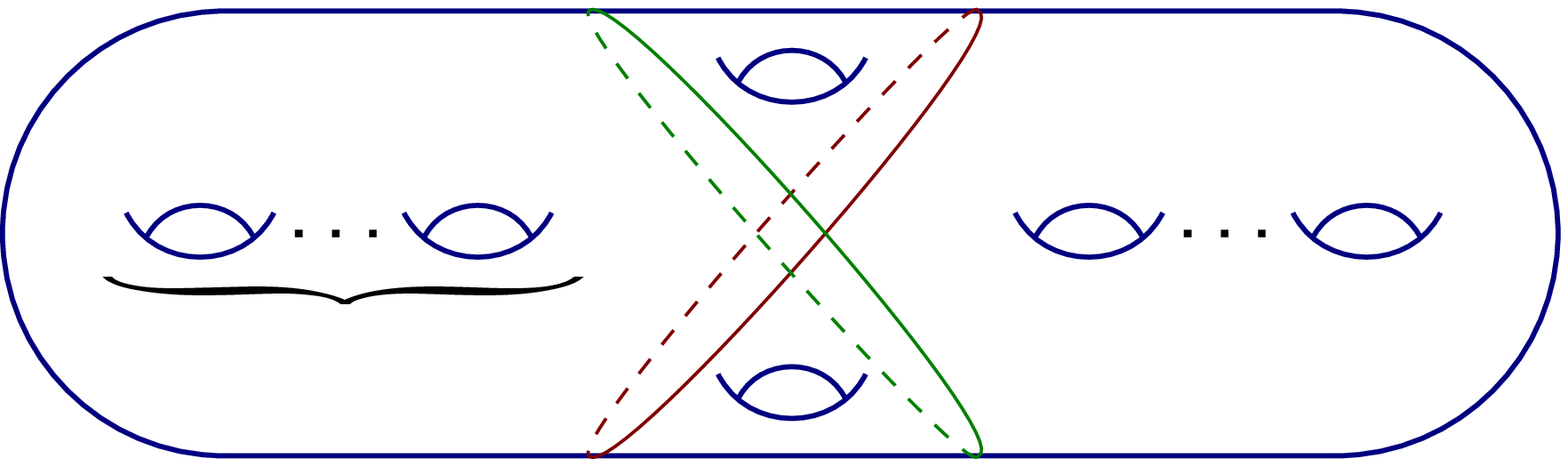}
  
  } 
  }$
  \hspace*{.2in}
  $\vcenter{\hbox{\includegraphics[scale=.35]{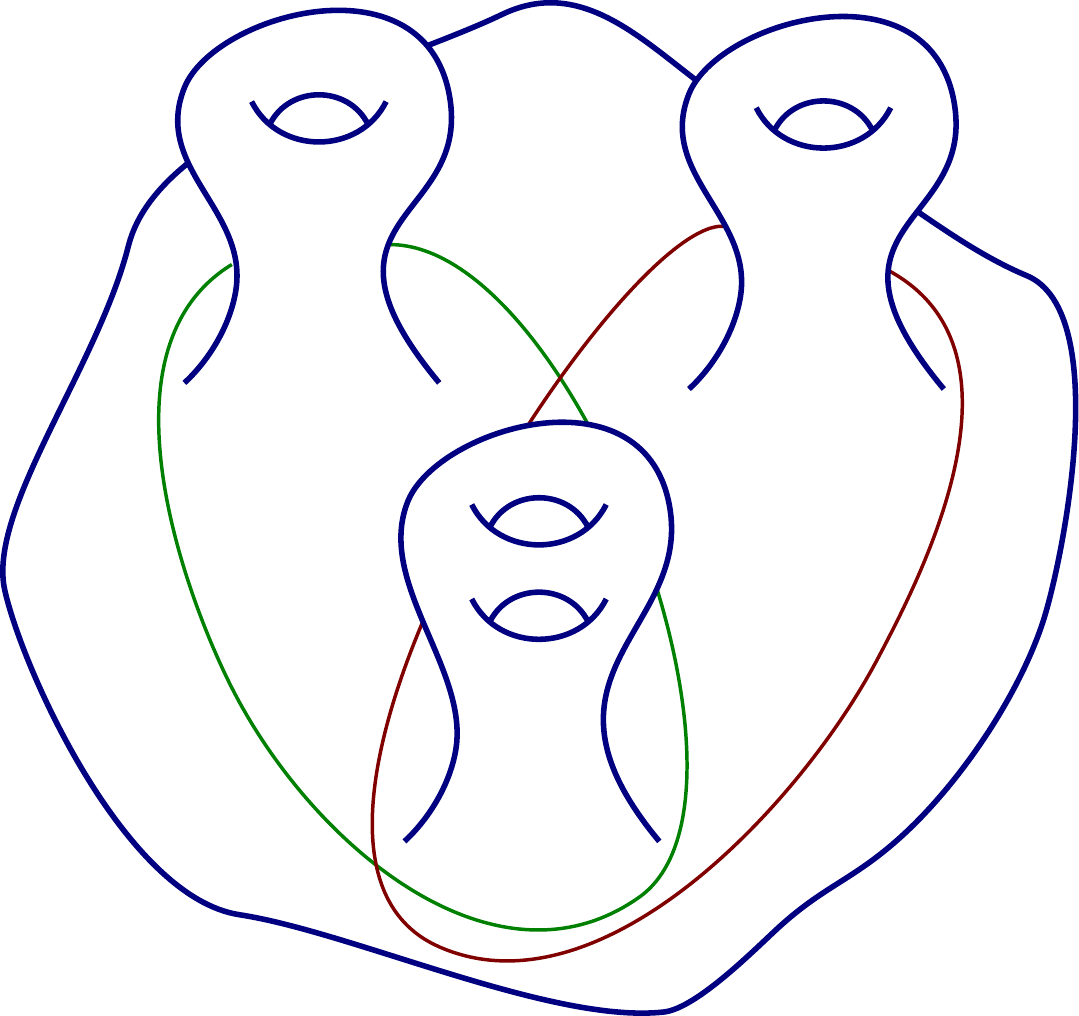}}}$
\end{minipage}
\caption{\emph{Left:} A sharing pair of genus $k$.  \emph{Right:} Another view of a sharing pair of genus three}
\label{fig:sp}
\end{figure}

To this end, we define a \emph{sharing pair of genus $k$} as a pair of vertices of $\C_k(S_g)$ that both have genus $k$ and are configured as in Figure~\ref{fig:sp}.  Specifically, if we choose representative curves in minimal position, then the two curves intersect in two points and  cut the surface into four regions, each with one boundary component, and the genera of the regions are $g-k-1$, $k-1$, 1, and 1.  The key point is that (as long as $g > 2k$) a sharing pair specifies a unique separating curve of genus $k-1$, namely, the boundary of the region of genus $k-1$ defined by the sharing pair.  

The main steps are to show for any automorphism $\phi$ of $\C_k(S_g)$ that
\begin{enumerate}
\item if two vertices of $\C_k(S_g)$ form a sharing pair their $\phi$-images do, and
\item if two sharing pairs in $\C_k(S_g)$ specify the same vertex of $\C_{k-1}(S_g)$ then their $\phi$-images do.
\end{enumerate}
Given these properties, we can define an automorphism $\hat \phi$ of $\C_{k-1}(S_g)$, whereby the action of $\hat \phi$ on a vertex of genus $k-1$ is dictated by the action of $\phi$ on sharing pairs.  

Here is the outline for the section.  First we begin with two preliminaries, dealing with basic separation properties (Lemma~\ref{lemma:joins}) and with subsurface projections (Lemma~\ref{lemma:vertex}).  Then we prove the two properties above (Lemmas~\ref{lemma:sharing} and Lemma~\ref{lemma:sp graph connected}) before finally finishing the proof.

\p{Separation properties.} We begin with some basic characterizations.  If $z$ is a vertex of $\C_k(S_g)$ of genus $h$, then any representative of $z$ divides $S_g$ into two regions, one of genus $h$ and one of genus $g-h$.  Both regions are well defined up to isotopy and refer to them as the \emph{sides} of $z$.  If $a$ is any other vertex connected to $z$ by an edge, then $a$ lies on one side of $z$; this side is well-defined.  

The \emph{join} of two simplicial complexes $X$ and $Y$ is the simplicial complex $X \ast Y$ whose simplices are the disjoint unions of the simplices in $X$ and $Y$. 

\begin{lemma}
\label{lemma:joins}
Let $k \geq 1$.  Let $\phi \in \Aut \C_{k}(S_g)$, let $z$ be a vertex of $\C_k(S_g)$, and let $a$ and $b$ be a vertices connected by edges to $z$.  Then
\begin{enumerate}
\item $z$ and $\phi(z)$ have the same genus, 
\item\label{small side} the sides of $z$ and $\phi(z)$ corresponding to $a$ and $\phi(a)$ have the same genus, and
\item $a$ and $b$ lie on the same side of $z$ if and only if $\phi(a)$ and $\phi(b)$ lie on the same side of $\phi(z)$.
\end{enumerate}
\end{lemma}

\begin{proof}

The link of $z$ in $\C_k(S_g)$ is the join of two subcomplexes, namely, the subcomplexes corresponding to the two sides of $z$.  Moreover, if we write the link of $z$ as the join of two subcomplexes, then those two subcomplexes must be the ones corresponding to the two sides.  This is a consequence of the fact that if $v$ and $w$ are two vertices of $\C_{k}(S_g)$ lying on the same side of $z$ then there is another vertex $x$ in the link of $z$ that is not connected by an edge to either $v$ or $w$ in $\C_{k}(S_g)$.

Say that the genus of $z$ is $h$.  The three statements of the lemma follow from the previous paragraph and the fact that the subcomplexes corresponding to the two sides of $z$ have dimensions $h-k-1$ and $g-h+k+1$.  
\end{proof}

\p{Projections} Before showing that automorphisms preserve sharing pairs, we need to introduce one further tool: subsurface projection maps.  This idea was introduced by Masur and Minsky in their work on the geometry of the complex of curves \cite{masurminsky}. 

Let $z$ be a vertex of $\C_k(S_g)$ whose genus is strictly less than $g/2$.  Let $R$ be a region of $S_g$ whose boundary represents $z$ and whose genus is equal to that of $z$.  Let $v$ be a vertex of $\C_k(S_g)$ that is not in the star of $z$ (that is, $v$ intersects $z$).  The projection $\pi_z(v)$ is a collection of homotopy classes of disjoint arcs in $R$ defined as follows: we choose a representative of $v$ that lies in minimal position with $\partial R$, take the intersection of this representative with $R$, and identify parallel arcs to a single arc.  

The projection $\pi_z(v)$ is a well-defined collection of homotopy classes of arcs in $R$.  We say that $\pi_z(v)$ is a \emph{nonseparating arc} if a representative has a single component and is nonseparating.  Next, we say that $\pi_z(v)$ and $\pi_z(w)$ are \emph{unlinked} if they have representatives that are disjoint and whose endpoints on $\partial R$ alternate.  Finally we say that $\pi_z(v)$ and $\pi_z(w)$ form a \emph{handle pair} if they have representatives that are distinct nonseparating arcs and so that the subsurface of $R$ filled by these projections is a surface of genus one with two boundary components.

\begin{lemma}
\label{lemma:vertex}
Let $k \geq 2$.  Let $z$ be a vertex of $\C_k(S_g)$ of genus $h$ where $k < h < g/2$.  Let $u$, $v$, and $w$ be vertices of $\C_k(S_g)$ that are not in the star of $z$, and suppose that $u$ is connected to $v$ by an edge.
\begin{enumerate}
\item If $\pi_z(v)$ is a nonseparating arc then $\pi_{\phi(z)}(\phi(v))$ is as well; 
\item if $\pi_z(v)$ and $\pi_z(w)$ are distinct nonseparating arcs then $\pi_{\phi(z)}(\phi(v))$ and $\pi_{\phi(z)}(\phi(w))$ are distinct nonseparating arcs;
\item if $\pi_z(v)$ and $\pi_z(w)$ form a handle pair then $\pi_{\phi(z)}(\phi(v))$ and $\pi_{\phi(z)}(\phi(w))$ form a handle pair.
\item if $\pi_z(u)$ and $\pi_z(v)$ are unlinked nonseparating arcs then $\pi_{\phi(z)}(\phi(u))$ and $\pi_{\phi(z)}(\phi(v))$ are unlinked nonseparating arcs; 
\end{enumerate}
\end{lemma}

\begin{proof}

We fix a region $R$ of genus $h$ whose boundary represents $z$.  For the first statement, we claim that $\pi_z(v)$ is a nonseparating arc if and only if $\C_k(S_g)$ has more than one vertex of genus $h-1$ that lies on the genus $h$ side of $z$ and is connected by an edge to $v$.  The first statement will follow from the claim and Lemma~\ref{lemma:joins}.  

For the forward direction, assume that $\pi_z(v)$ is a single nonseparating arc.  If we cut $R$ along a representative of the arc $\pi_z(v)$ the resulting surface has genus $h-1$ and two boundary components.  There are infinitely many isotopy classes of simple closed curves in the cut surface that separate the cut surface into a pair of pants and a surface of genus $h-1$ with one boundary component.  Each of these corresponds to a vertex of genus $h-1$ in $\C_k(S_g)$ that is connected by an edge to $v$.

For the other direction, there are two cases: either $\pi_z(v)$ contains the homotopy class of a separating arc or $\pi_z(v)$ contains more than one homotopy class of nonseparating arcs.  In the first case, if we divide a representative of $R$ along such a separating arc, we obtain two surfaces of genus $h_1$ and $h_2$ with $h_2 > h_1 > 0$ and $h_1+h_2=h$.  In particular, $h_i \leq h-1$.  Any vertex of $\C_k(S_g)$ that has a representative in $R$ and is connected to $v$ by an edge must lie in one of these subsurfaces. If $h_1 =1$ then there is a unique such vertex; otherwise, there is no such vertex.  For the second case, we note that if we cut $R$ along two disjoint nonseparating arcs we either obtain a surface of genus $h-1$ with a single boundary component,  we obtain a surface of genus less than $h-1$, or we obtain two surfaces, each with two boundary components and genus less than $h-1$.  So either there is a single vertex of genus $h-1$ as in the claim or there are none.

The second statement follows from Lemma~\ref{lemma:joins} and the first statement, since $\pi_z(v)$ and $\pi_z(w)$ are determined by the vertices of genus $h-1$ that lie on the genus $h$ side of $z$ and are disjoint from $v_1$ and $v_2$, respectively.

We now proceed to the third statement.  Two distinct, nonseparating projections $\pi_z(v)$ and $\pi_z(w)$ form a handle pair if and only if there exists a vertex of genus $h-1$ in $\C_k(S_g)$ that lies on the genus $h$ side of $z$ and is connected by edges to both $v$ and $w$ in $\C_k(S_g)$.  The third statement then follows from Lemma~\ref{lemma:joins}.

Finally we prove the fourth statement.  Since $u$ and $v$ are connected by an edge, their projections are disjoint.  Also, by the first statement we know that the projections are nonseparating if and only if their images are.  It remains to characterize linking and unlinking for disjoint nonseparating projections.  But disjoint nonseparating projections $\pi_z(u)$ and $\pi_z(v)$ are linked if and only if they form a handle pair, and so an application of the third statement completes the proof.
\end{proof}

\p{Sharing pairs} We now show that automorphisms of $\C_k(S_g)$ preserve sharing pairs. As discussed at the start of this section, this is the main tool used to extend an automorphism $\phi$ of $\C_k(S_g)$ to an automorphism $\hat \phi$ of $\C_{k-1}(S_g)$.

\begin{lemma}
\label{lemma:sharing}
Let $k \geq 2$, let $g \geq 3k+1$, and let $a$ and $b$ be two vertices of $\C_k(S_g)$ that form a sharing pair of genus $k$.  If $\phi$ is an automorphism of $\C_k(S_g)$, then $\phi(a)$ and $\phi(b)$ form a sharing pair of genus $k$.
\end{lemma}

\begin{figure}
  \labellist
\small\hair 2pt
 \pinlabel {$a$} [ ] at 170 185
 \pinlabel {$b$} [ ] at 290 185
 \pinlabel {$z$} [ ] at 315 185
 \pinlabel {$k-1$} [ ] at 90 65
 \pinlabel {$k-1$} [ ] at 585 125
 \pinlabel {$k-1$} [ ] at 585 15
 \pinlabel {$P$} [ ] at 585 180
 \pinlabel {$Q$} [ ] at 585 70
 \endlabellist  
\includegraphics[scale=.525]{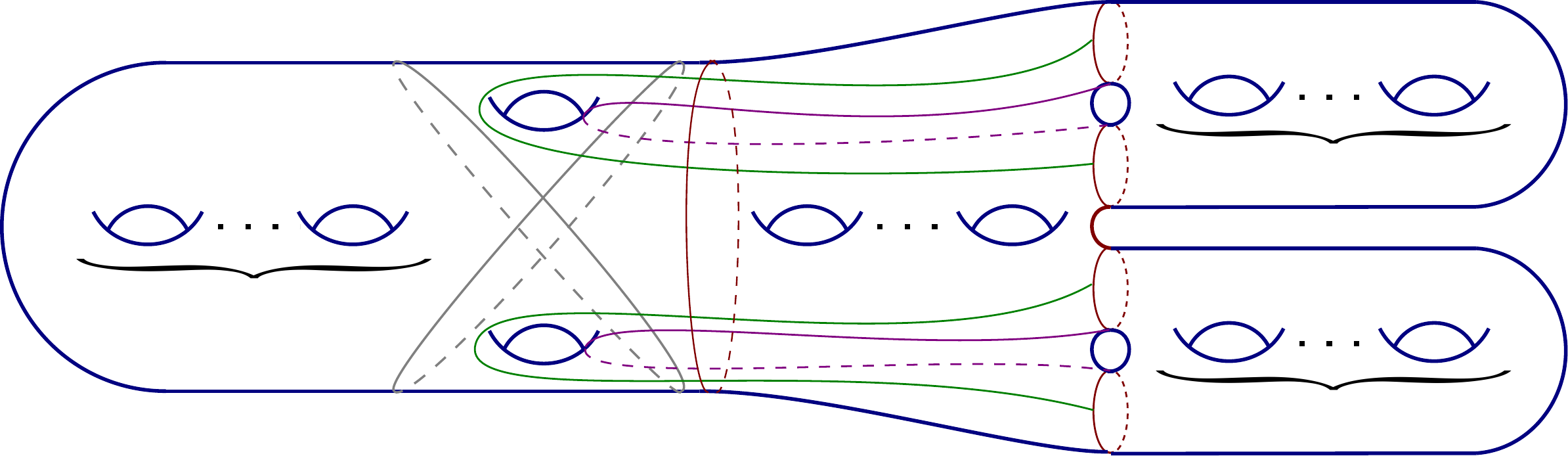}
\caption{Curves and arcs used to characterize a sharing pair}
\label{fig:sp soup}
\end{figure}
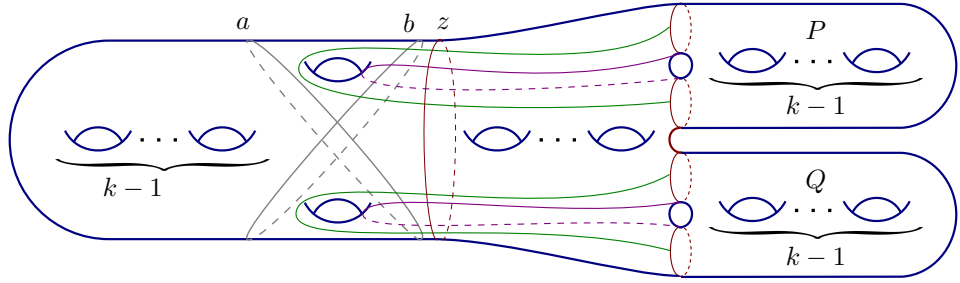

\begin{proof}

We will show that two vertices $a$ and $b$ of $\C_k(S_g)$ form a sharing pair of genus $k$ if and only if there are vertices $x_1$, $x_2$, $y_1$, $y_2$, and $z$ of $\C_k(S_g)$ with the following properties:
\begin{enumerate}
 \item\label{condition:genus} the genus of $z$ is $k+1$;
 \item\label{condition:small} $a$ and $b$ are vertices of genus $k$ lying on the genus $k+1$ side of $z$;
 \item\label{condition:penta} each $x_i$ is connected by edges to $a$, $y_1$, and $y_2$ but not to $b$;
 \item\label{condition:pentb} each $y_i$ is connected by edges to $b$, $x_1$, and $x_2$ but not to $a$; 
 \item\label{condition:handle} $\pi_z(x_1)$ and $\pi_z(x_2)$ form a handle pair and $\pi_z(y_1)$ and $\pi_z(y_2)$ form a handle pair; and
 \item\label{condition:unlink} each $\pi_z(x_i)$ is unlinked with each $\pi_z(y_j)$.
\end{enumerate}
(We have implicitly used the fact that $z$ has only one side of genus $k+1$, but this is implied by the conditions on $g$ and $k$.) Lemmas~\ref{lemma:joins} and ~\ref{lemma:vertex} together imply that $a$, $b$, the $x_i$, the $y_i$, and $z$ satisfy the given conditions if and only if their $\phi$-images do.  Therefore, it remains to prove that the existence of such $a$, $b$, the $x_i$, the $y_i$, and $z$ is equivalent to the condition that $a$ and $b$ form a sharing pair. 

The forward direction is given by explicit construction; refer to Figure~\ref{fig:sp soup}.  There is a unique configuration for $a$ and $b$ as in the figure.  The curve $z$ is shown.  For each $i$ the curve $x_i$ is obtained by taking the boundary of a regular neighborhood of the union of the region $P$ with one of the arcs in the picture with endpoints on $P$ (the order of the arcs does not matter).  Similarly each $y_i$ is obtained as the boundary of a regular neighborhood of the union of the region $Q$ with one of the arcs with endpoints on $Q$.  

We now proceed to the other implication.  Assume that $a$, $b$, $x_1$, $x_2$, $y_1$, $y_2$, and $z$ satisfy the conditions in the claim. Let $R$ be a region of $S_g$ representing the genus $k+1$ side of $z$, as per property \eqref{condition:genus}.  By property~\eqref{condition:small}, the vertices $a$ and $b$ have representatives in $R$.  Let $\bar R$ denote the closed surface obtained by collapsing the boundary of $R$ to a marked point.  Each of $a$ and $b$ separates $\bar R$ into a surface of genus $k$ with one boundary component and a surface of genus one with one boundary component and one marked point.

By property~\eqref{condition:handle}, the vertices $x_1$ and $x_2$ give rise to a pair of nonseparating arcs in $\bar R$ based at the marked point, and these arcs fill a subsurface $Q_x$ of $\bar R$ homeomorphic to a surface of genus one with one boundary component and one marked point.  Similarly, the $y_i$ give a pair of nonseparating arcs in $\bar R$ that fill a subsurface $Q_y$ with the same properties.

Since there is only one separating curve of genus $k$ disjoint from $Q_x$, namely the boundary of $Q_x$, property~\eqref{condition:penta} implies that $a$ is represented by the boundary of $Q_x$.  Similarly, property~\eqref{condition:pentb} implies that $b$ is represented by the boundary of $Q_y$.  Our main goal at this point is to show that the geometric intersection number $i(a,b)$  is 2, and so we have reduced this to a problem about the $x_i$ and $y_i$.

If we consider a small closed disk around the marked point of $\bar R$, the $x_i$-arcs and $y_i$-arcs give a collection of eight disjoint arcs connecting the marked point to the boundary of the disk, and since no triple among the $x_i$ and $y_i$ have pairwise nontrivial intersection, the eight arcs in the disk come in a well-defined cyclic order (depending only on the $x_i$, the $y_i$, and $z$).  Any homotopically distinct based simple loops in $Q_x$ must cross transversely at the base point (this follows from the identification of the set of oriented nonseparating simple closed curves in the punctured torus with the primitive elements of $\Z^2$).  It follows that the two $x_1$-arcs and the two $x_2$-arcs alternate in the cyclic order.  Then, since the $x_i$-arcs fill $Q_x$, it follows from property~\eqref{condition:unlink} that in the cyclic order the four $x_i$-arcs in the disk appear in order, followed by the four $y_i$-arcs.

Since the $x_i$-arcs fill $Q_x$ and since the $x_i$ are disjoint from the $y_i$ it follows from the previous paragraph that if we take the $y_i$-arcs in $\bar R$ and intersect them all with $Q_x$ we obtain a set of four parallel arcs connecting the marked point to the boundary.  As $Q_y$ is obtained from a regular neighborhood of the $y_i$-arcs, it follows that the intersection of $Q_x$ with $Q_y$ is a disk.  Since $a$ and $b$ are identified with the boundary components of $Q_x$ and $Q_y$, it follows that $i(a,b)=2$. There is only one possible configuration for two separating curves of genus $k$ in $R$ with intersection number two, and so $a$ and $b$ form a sharing pair, as desired.
\end{proof}

\p{Sharing triples} As discussed at the beginning of the section, Lemma~\ref{lemma:sharing} suggests a method for extending an automorphism of $\C_k(S_g)$ to an automorphism of $\C_{k-1}(S_g)$: the image of a vertex of genus $k-1$ is determined by the image of a corresponding sharing pair.  We need to show that this rule is well defined.  We will use sharing triples to address this issue.

\begin{figure}
\includegraphics[scale=.35]{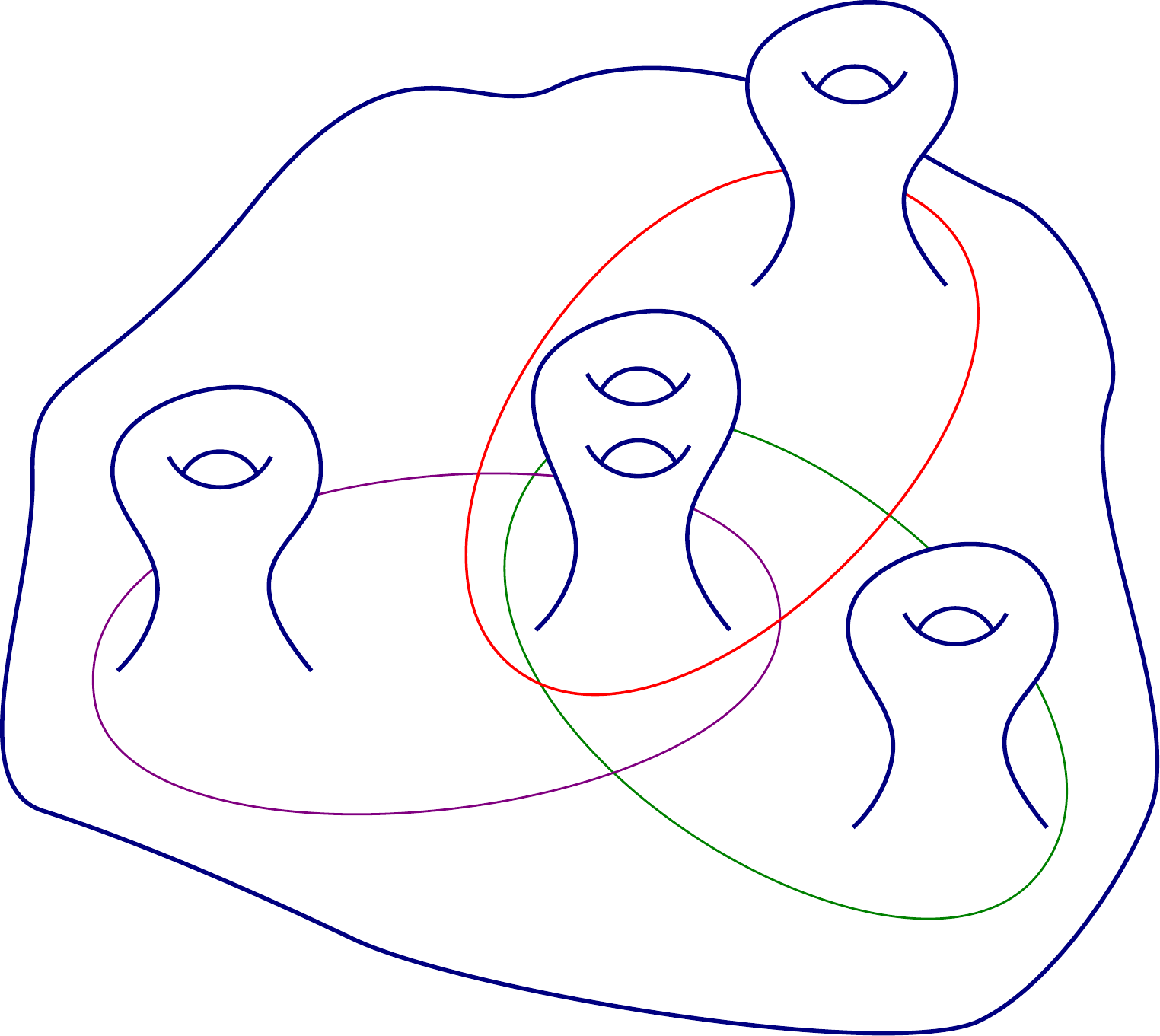}
\caption{A sharing triple of genus three}
\label{fig:triple}
\end{figure}

We say that three vertices of $\C_k(S_g)$ form a \emph{sharing triple of genus $k$} if they are configured as in Figure~\ref{fig:triple}.  Specifically, each pair of vertices in a sharing triple forms a sharing pair for the same vertex of $\C_{k-1}(S_g)$.  It is also true that if three vertices of $\C_k(S_g)$ pairwise form sharing pairs for the same vertex of $\C_{k-1}(S_g)$ then they form a sharing triple but we will not use this.

The next lemma tells us that if two sharing pairs belong to the same sharing triple, then their images under any automorphism of $\C_k(S_g)$ represent the same vertex of genus $k-1$.  Lemma~\ref{lemma:sp graph connected} below then tells us that if two arbitrary sharing pairs represent the same vertex of genus $k-1$, then their images under any automorphism of $\C_k(S_g)$ also represent the same vertex.

\begin{lemma}
\label{lemma:triples}
Let $k \geq 2$, let $g \geq 3k+1$, and let $\phi$ be an automorphism of $\C_k(S_g)$.   If $a$, $b$, and $c$ are three vertices of $\C_k(S_g)$ that form a sharing triple of genus $k$ then $\phi(a)$, $\phi(b)$, and $\phi(c)$ form a sharing triple of genus $k$.
\end{lemma}

\begin{proof}

We claim that three vertices $a$, $b$, and $c$ form a sharing triple of genus $k$ if and only if there are vertices $z$ and $d$ of $\C_k(S_g)$ so that the following conditions hold:
\begin{enumerate}
\item the genus of $z$ is $k+2$, 
\item $a$, $b$, and $c$ all lie on a genus $k+2$ side of $z$,
\item any two of $a$, $b$, and $c$ form a sharing pair of genus $k$,
\item any vertex of $\C_k(S_g)$ that is connected by an edge to each of $a$, $b$, and $c$ must also be connected by an edge to $z$, and
\item $d$ forms a sharing pair with $c$ and is connected by edges to $a$ and $b$.
\end{enumerate}
(In the special case where $k=2$ and $g=7$ the first condition should be interpreted as saying that the genus of $z$ is 3.)  The lemma follows from the claim plus Lemmas~\ref{lemma:joins} and~\ref{lemma:sharing}.  The last condition in the claim is included only to rule out one configuration in the case $k=2$.  On the other hand for any $k \geq 2$ there are fake sharing triples that satisfy the first three conditions but not the fourth.  

\begin{figure}
  \labellist
\small\hair 2pt
 \pinlabel {$a$} [ ] at 370 70
 \pinlabel {$b$} [ ] at 345 305
 \pinlabel {$c$} [ ] at 285 70
 \pinlabel {$z$} [ ] at 240 305
 \pinlabel {$d$} [ ] at 115 70
 \endlabellist  
\includegraphics[scale=.35]{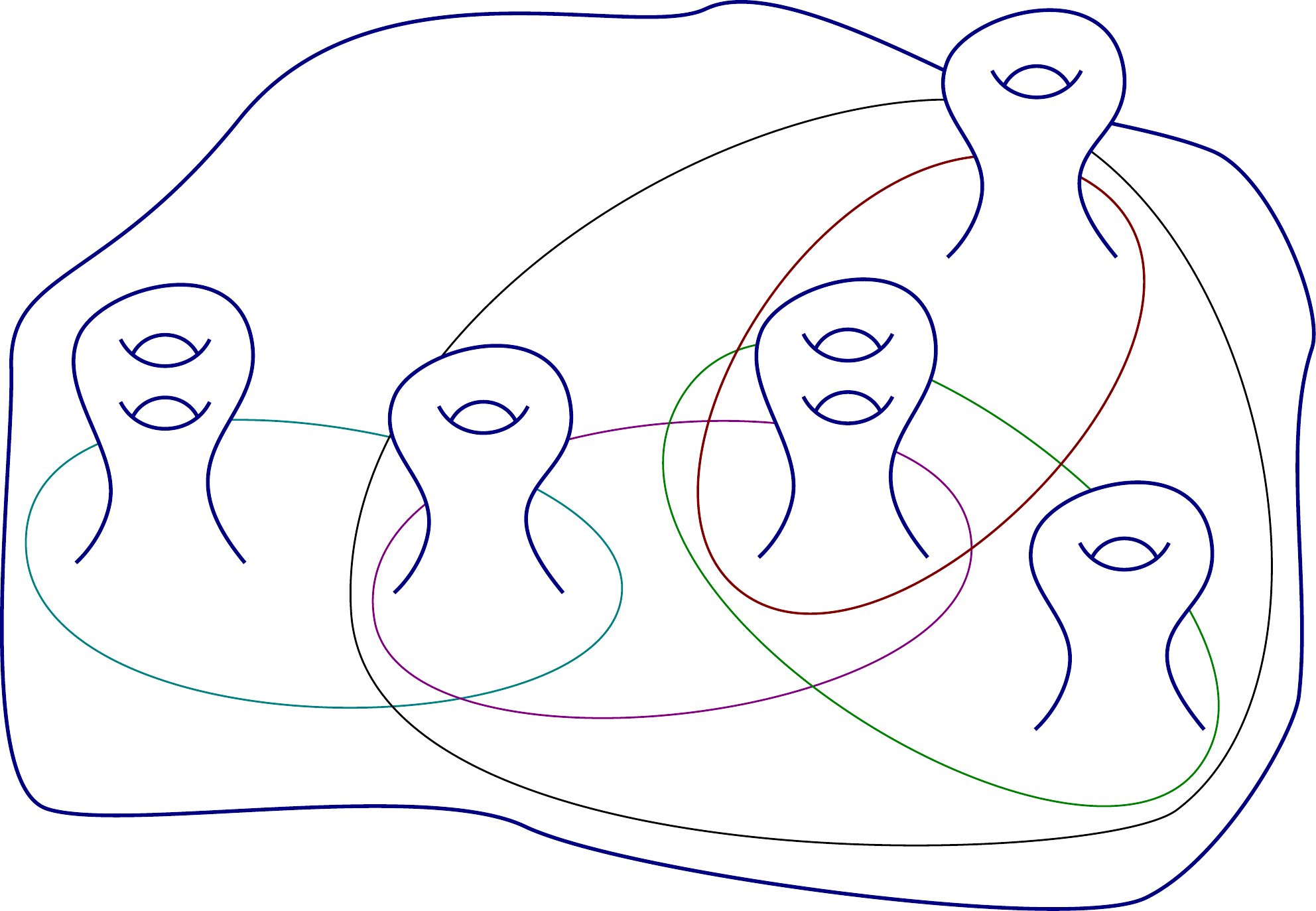}
\caption{The vertices $z$ and $d$ in the characterization of sharing triples.}
\label{fig:sp2}
\end{figure}
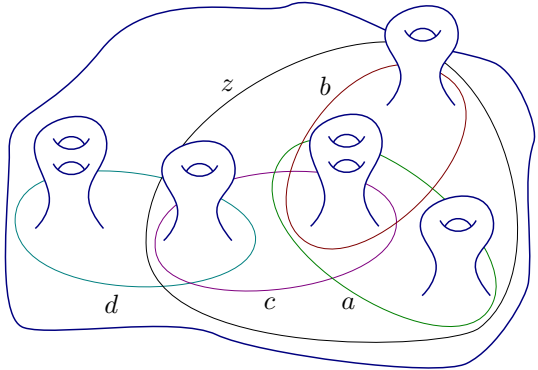

The forward direction of the claim is straightforward.  The construction of the vertices $z$ and $d$ is indicated in Figure~\ref{fig:sp2}.  Note that the vertex $d$ exists because $g \geq 2k+1$.

For the other direction, assume that $a$, $b$, and $c$ are three vertices of $\C_k(S_g)$ that satisfy the conditions of the claim.  We must show that $a$, $b$, and $c$ form a sharing triple of genus $k$.

Choose a vertex $z$ as in the claim, and choose a representative curve in $S_g$.  Let $R$ be the region of $S_g$ that is determined by this curve and contains representatives of $a$, $b$, and $c$.  Choose representative curves for $a$ and $b$ in $R$.  Since $a$ and $b$ form a sharing pair, we can assume that these curves are configured as in Figure~\ref{fig:sp}.

There is a region $Q$ of $R$ that lies between the boundary of $R$ and the union of the $a$-curve and the $b$-curve.  This region $Q$ is homeomorphic to a surface of genus one with two boundary components; one boundary component is the $z$-curve and the other boundary component is made up of one arc of the $a$-curve and one arc of the $b$-curve.  

Choose a curve in $R$ representing $c$, and take this curve to be in minimal position with the $a$-curve and the $b$-curve, and so that there are no triple intersections.  Since $c$ lies on the same side of $z$ as $a$ and $b$ and is not connected by an edge to $a$ or $b$, the intersection of the $c$-curve with $Q$ is a collection of disjoint arcs, each starting and ending on the boundary component coming from $a$ and $b$.  Since two curves in a sharing pair intersect in two points, there are at most two arcs.  We may further assume that no $c$-arc in $Q$ is peripheral, for in this case we can push this arc out of $Q$ without increasing the number of intersections of the $c$-curve with either the $a$-curve or the $b$-curve.

\begin{figure}
\includegraphics[scale=.4]{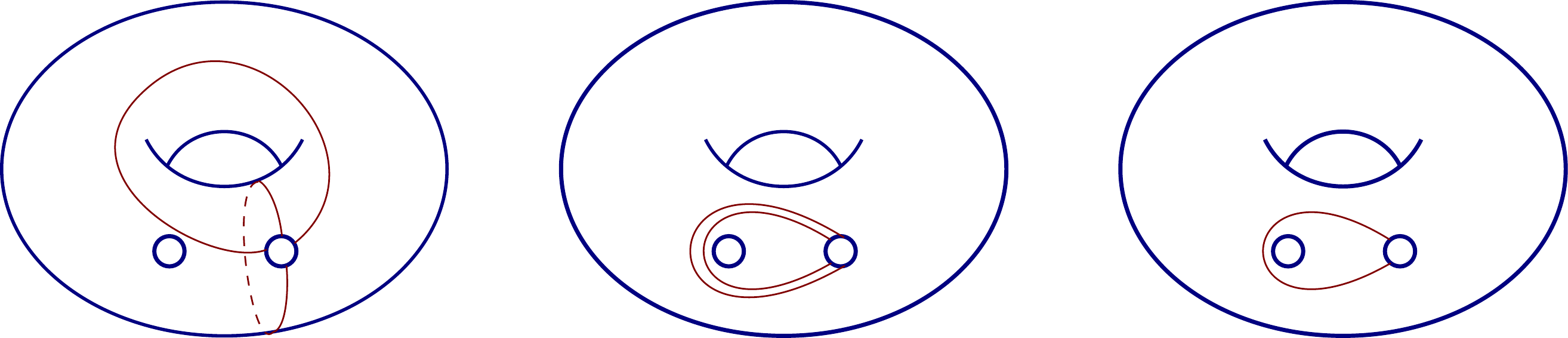}
\caption{Three possibilities for the intersection of $c$ with $Q$ in the proof of Lemma~\ref{lemma:triples}}
\label{fig:q}
\end{figure}

By the fourth condition of the claim, the $c$-arcs in $Q$ must have the following property: if we cut $Q$ along the $c$-arcs the component containing the $z$-curve must be an annulus.  We thus have the following possibilities:
\begin{enumerate}
\item there are two $c$-arcs in $Q$ that are not homotopic and nonseparating,
\item there are two $c$-arcs in $Q$ that are parallel and separating, or
\item there is a single $c$-arc in $Q$ that is separating;
\end{enumerate}
see Figure~\ref{fig:q}.

The first case can be ruled out because in this case the $c$-curve is nonseparating in $S_g$ (we can find a curve in $Q$---hence $S_g$---that intersects it in one point).

The second case can be ruled out as follows.  First, any vertex $v$ of $\C_k(S_g)$ that is connected by edges to $a$ and $b$ but not $c$ is necessarily not connected by an edge to $z$.  This is simply because $v$ has genus $k$ and the genus of $Q$, the region between $z$ and $a \cup b$, is only one.  It follows that the vertex $d$ from the fifth condition of the claim must intersect $z$.  But in the second case any curve that is disjoint from $a$ and $b$ but not $z$ must intersect $c$ in at least four points.  This contradicts the assertion that $c$ and $d$ form a sharing pair.  

We now consider the third case, where there is a single separating $c$-arc in $Q$.  In this case the $c$-arc in $Q$ is configured like the $c$-arc on the outside of $a$ and $b$ in Figure~\ref{fig:sp2}.  It remains to determine the configuration of $c$ in the union of the interiors of $a$ and $b$ in Figure~\ref{fig:sp2}.  The $c$-arcs in the genus $k$ sides of the $a$- and $b$-curves are separating and must cut off a subsurface of genus $k-1$ in each region.  The only possibility then is that $c$ is configured exactly as in Figure~\ref{fig:sp2} and hence forms a sharing triple with $a$ and $b$.  
\end{proof}

\p{The sharing pair graph} Assume $g \geq 2k$ and let $y$ be a vertex of $\C_{k-1}(S_g)$.  We define the \emph{sharing pair graph} for $y$ as the graph whose vertices correspond to sharing pairs of genus $k$ that specify $y$ and whose edges correspond to sharing pairs $\{a,b\}$ and $\{b,c\}$ with the property that $\{a,b,c\}$ is a sharing triple.

Let $\Mod(S_g,y)$ denote the subgroup of $\Mod(S_g)$ consisting of elements represented by homeomorphisms that act by the identity on the region of $S_g$ corresponding to the genus $k-1$ side of $y$.  Note that $\Mod(S_g,y)$ acts on the sharing pair graph for $y$ and acts transitively on the vertices.

\p{Putman's trick} If $G$ is a group with a generating set $\{g_i\}$ and $G$ acts on a graph $X$ with base point $v$ and $G$ acts transitively on the vertices of $X$ then $X$ is connected if for each $i$ the vertices $g_i \cdot v$ and $v$ lie in the same component of $X$.  As Putman explains \cite{putman}, this method is useful in the theory of mapping class groups because one can often choose the $v$ and the $g_i$ so that most of the $g_i \cdot v$ are equal to $v$ and the other $g_i \cdot v$ are very close to $v$.  We refer to this method as Putman's trick.  We will presently apply it to the action of $\Mod(S_g,y)$ on the sharing pair graph for $y$.

\begin{lemma}
\label{lemma:sp graph connected}
Let $k \geq 2$ and let $g \geq 2k$.  Let $y$ be a vertex of $\C_{k-1}(S_g)$.  The sharing pair graph for $y$ is connected.
\end{lemma}

\begin{proof}

Let $y$ be the vertex of $\C_{k-1}(S_g)$ shown in Figure~\ref{fig:hump}.  We enlist the sharing pair $\{a,b\}$ shown in Figure~\ref{fig:hump} to act as a base point $v$ for the sharing pair graph for $y$.  As mentioned, the group $\Mod(S_g,y)$ acts transitively on the vertices of this graph.  

Denote $g-k+1$ by $g'$.  The group $\Mod(S_g,y)$ is isomorphic to the mapping class group of the region of $S_g$ that lies on the genus $g'$ side of the $y$-curve.  As such it is generated by the Dehn twists about the curves $\{d_0,\dots,d_{2g'}\}$ indicated in Figure~\ref{fig:hump} for the case where $g=k+4$; see \cite[Theorem 1]{johnson}.

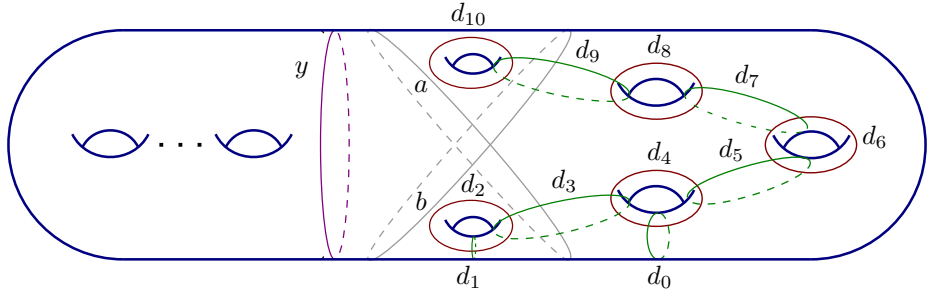
\begin{figure}
  \labellist
\small\hair 2pt
 \pinlabel {$b$} [ ] at 260 38
 \pinlabel {$a$} [ ] at 260 110
 \pinlabel {$y$} [ ] at 185 120
 \pinlabel {$d_0$} [ ] at 410 -10
 \pinlabel {$d_1$} [ ] at 290 -10
 \pinlabel {$d_2$} [ ] at 293 48
 \pinlabel {$d_3$} [ ] at 350 50
 \pinlabel {$d_4$} [ ] at 410 68
 \pinlabel {$d_5$} [ ] at 455 70
 \pinlabel {$d_6$} [ ] at 545 78
 \pinlabel {$d_7$} [ ] at 465 116
 \pinlabel {$d_8$} [ ] at 410 135
 \pinlabel {$d_9$} [ ] at 365 132
 \pinlabel {$d_{10}$} [ ] at 290 155
 \endlabellist  
\includegraphics[scale=.6]{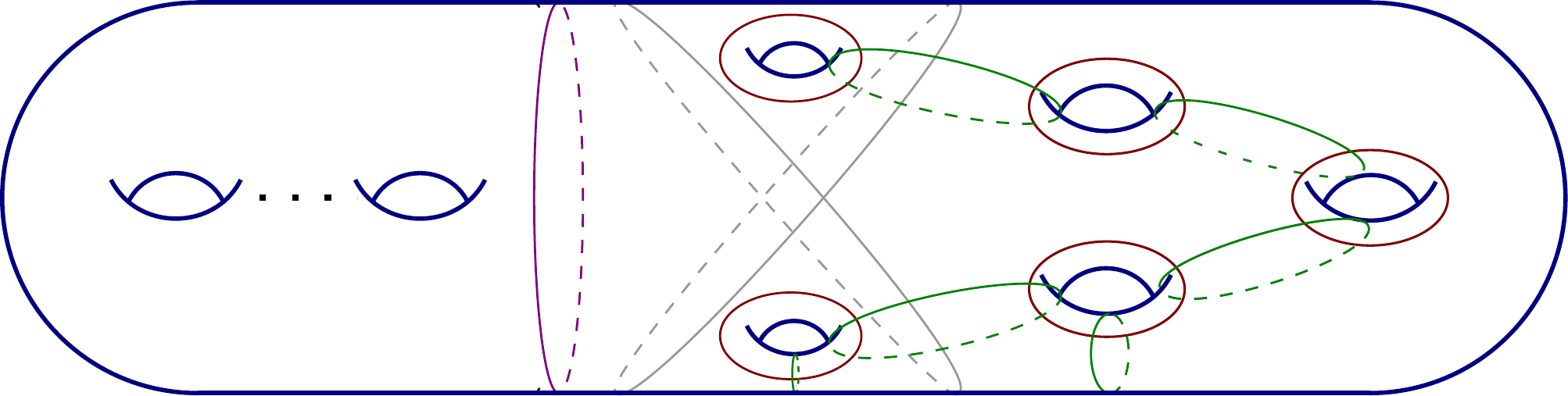}
\caption{Curves $d_i$ whose Dehn twists generate $\Mod(S_g,y)$}
\label{fig:hump}
\end{figure}

By Putman's trick, it is enough to show that each $T_{d_i}(v)$ lies in the same connected component of the sharing pair graph as $v$.  We have arranged things so that when $d_i \notin \{d_3,d_{2g'-1}\}$ we have $T_{v_i} \cdot v = v$ and there is nothing to check.  It remains to check that $T_{d_3} \cdot v$ and $T_{d_{2g'-1}} \cdot v$ lie in the same component as $v$.  The configuration $\{a,b,d_3\}$ is equivalent to the configuration $\{a,b,d_{2g'-1}\}$ and so it suffices to treat the case of $d_3$.

\begin{figure}
  \labellist
\small\hair 2pt
 \pinlabel {$a$} [ ] at 210 70
 \pinlabel {$b$} [ ] at 380 55
 \pinlabel {$c$} [ ] at 480 200
 \pinlabel {$d_3$} [ ] at 120 65
 \endlabellist  
\includegraphics[scale=.4]{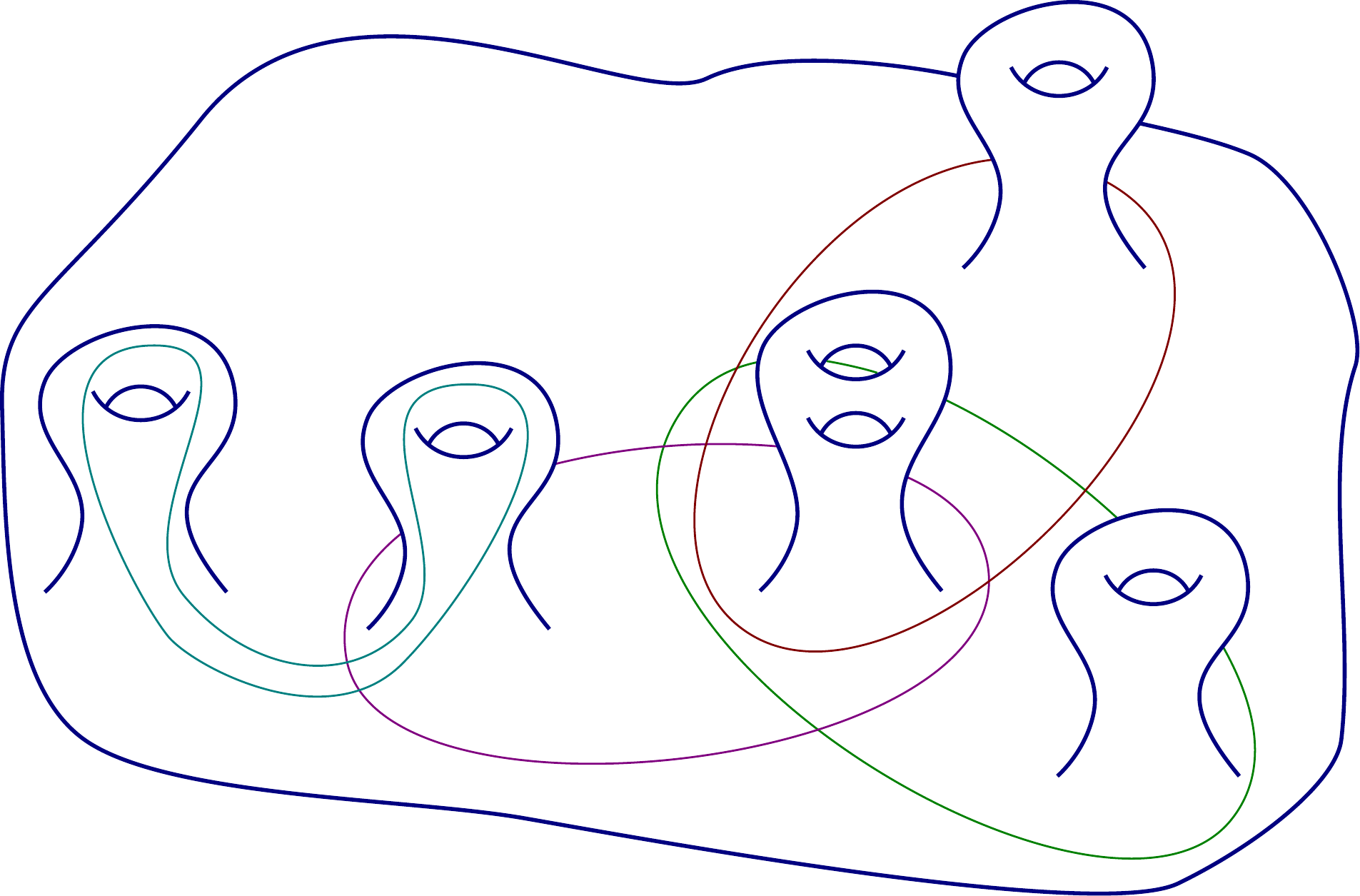}
\caption{The curves $a$, $b$, $c$, and $d_3$}
\label{fig:d3}
\end{figure}
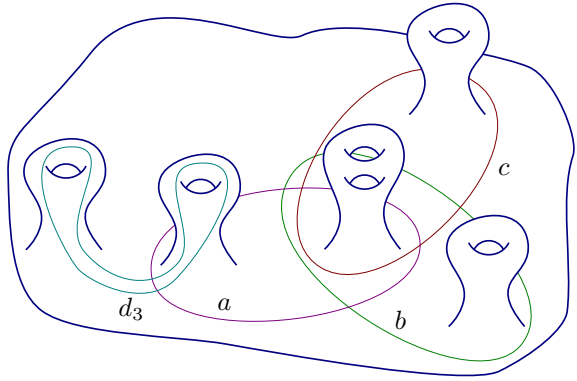

There exists a vertex $c$ of $\C_k(S_g)$ so that $\{a,b,c\}$ forms a sharing triple of genus $k$ for $y$ and so that $i(c,d_3)=0$; this is clear from the point of view of Figure~\ref{fig:d3}.  We also see from Figure~\ref{fig:hump} that $i(b,d_3)=0$.  Since $\{a,b,c\}$ is a sharing triple, its $T_{d_3}$-image $\{T_{d_3}(a),b,c\}$ is a sharing triple as well.  Thus, $T_{d_3} \cdot v = \{T_{d_3}(a),T_{d_3}(b)\} = \{T_{d_3}(a),b\}$ is connected by an edge to $w=\{b,c\}$.  Since $v=\{a,b\}$ is also connected to $w$, we are done.
\end{proof}

\p{Finishing the proof} We are almost ready to finish the proof of Theorem~\ref{theorem:sep k}, which states that for $k \geq 2$ and $g \geq 3k+1$ the natural map $\MCG(S_g) \to \Aut \C_k(S_g)$ is an isomorphism.

\begin{lemma}
\label{inj}
Let $g \geq 3$ and let $X$ be one of the complexes $\C_k(S_g)$, $\C_D(S_g)$, or $\C_A(S_g)$, where $k \geq 1$, $D \subseteq \D(S_g)$ is nonempty, or $A \subseteq \R(S_g)$ is nonempty.  Then the natural map $\MCG(S_g) \to \Aut X$ is injective and the image contains no exchange automorphisms.
\end{lemma}

\begin{proof}

A vertex of $X$ represents a point in $\PMF(S_g)$, the space of projective measured foliations on $S_g$.  Indeed, for $\C_k(S_g)$ and $\C_D(S_g)$ we take the usual inclusion of the set of multicurves in $S_g$ into $\PMF(S_g)$ \cite[Section 5.4]{flp}, and for $\C_A(S_g)$ we first take the boundary of the corresponding subsurface and then take the usual inclusion of the set of multicurves into $\PMF(S_g)$.  The action of $\MCG(S_g)$ on $X$ clearly agrees with the action on $\PMF(S_g)$.

The action of $\MCG(S_g)$ on $\PMF(S_g)$ is continuous and minimal, meaning the orbit of every point is dense; see \cite[Theorem 6.19]{flp}.  It follows that if an element of $\MCG(S_g)$ acts trivially on $X$ then it acts trivially on $\PMF(S_g)$.  But the kernel of the action of $\MCG(S_g)$ on $\PMF(S_g)$ is trivial for $g \geq 3$ (see \cite[Proof of Theorem 3.10]{primer}) and so the first statement follows.  The second statement also follows, since the complement of two points in a dense subset of $\PMF(S_g)$ is still dense.
\end{proof}

\begin{proof}[Proof of Theorem~\ref{theorem:sep k}]

Fix $g \geq 4$.  We would like to show that for each $k$ with $1 \leq k < g/3$ the automorphism group of $\C_k(S_g)$ is isomorphic to $\MCG(S_g)$.  We proceed by induction on $k$.  The case $k=1$ was proven in our earlier paper \cite{kg}.  So suppose $1 < k < g/3$, and assume that the natural map $\MCG(S_g) \to \Aut \C_{k-1}(S_g)$ is an isomorphism.  By Lemma~\ref{inj} it suffices to show that the natural map $\MCG(S_g) \to \Aut \C_k(S_g)$ is surjective.  Let $\phi \in \Aut \C_{k}(S_g)$.  By Lemmas~\ref{lemma:sharing} and~\ref{lemma:sp graph connected} there is a well-defined automorphism $\hat \phi$ of the 0-skeleton of $\C_{k-1}(S_g)$ that agrees with $\phi$ on the 0-skeleton of $\C_k(S_g)$.  Specifically, if $v$ is a vertex of $\C_{k-1}(S_g)$ that does not lie in $\C_k(S_g)$ then $\hat \phi(v)$ is the vertex of $\C_{k-1}(S_g)$ determined by the $\phi$-image of any sharing pair for $v$.

We claim  that $\hat \phi$ extends to the 1-skeleton of $\C_{k-1}(S_g)$.  To this end, let $v$ and $w$ be vertices of $\C_{k-1}(S_g)$.  If $v$ and $w$ both lie in the subcomplex $\C_k(S_g)$ of $\C_{k-1}(S_g)$, there is nothing to show since the restriction of $\hat \phi$ to $\C_k(S_g)$ is equal to $\phi$.  Next suppose that neither $v$ nor $w$ lies in $\C_k(S_g)$; in other words $v$ and $w$ both have genus $k-1$.  In this case since $g \geq 3k+1 \geq 2k+2$ we can find sharing pairs for $v$ and $w$ that are disjoint, and from this---and the fact that two disjoint separating curves of genus $k$ cut off disjoint regions of genus $k$---the result follows.  The final case is where $v$ lies in $\C_k(S_g)$ but $w$ does not.  If $v$ has genus $k$ and $w$ lies on the genus $k$ side of $v$ then $v$ lies in a sharing pair for $w$, and the claim follows.  In all other cases, we can choose a sharing pair for $w$ that is disjoint from $v$ and again the claim follows.

By induction there exists some $f \in \MCG(S_g)$ whose image in $\Aut \C_{k-1}(S_g)$ is precisely $\hat \phi$.  Since the restriction of $\hat \phi$ to $\C_{k}(S_g)$ is equal to $\phi$, it follows that the image of $f$ in $\Aut \C_{k}(S_g)$ is $\phi$, as desired.
\end{proof}


\section{Complexes of dividing sets}
\label{sec:div}

The goal of this section is to prove Theorem~\ref{theorem:multi k}, which states that whenever $D$ is an upper set in $\D(S_g)$ with $g \geq 3\check{g}(D)+1$, the group of automorphisms of $\C_D(S_g)$ is isomorphic to $\MCG(S_g)$.

We will say that a vertex of $\C_D(S_g)$ has \emph{type $(k,n)$} if one of the regions of $S_g$ determined by a representative of $v$ is a surface of genus $k$ with $n$ boundary components.  Because a dividing set determines two complementary regions of $S_g$, a vertex of type $(k,n)$ is also a vertex of type $(g-k-n+1,n)$.

Next, we say that a dividing set has...
\begin{enumerate}
\item \emph{type N} if it is of type $(0,2)$, 
\item \emph{type S} if it is of type $(k,1)$ for some $k$, and 
\item \emph{type M} otherwise
\end{enumerate}
(N, S, and M are for nonseparating curve, separating curve, and multicurve).  Vertices of type N and S correspond to vertices of the complex of curves.  Also, a dividing set is nested with a vertex of type N or S if and only if it is disjoint.  Therefore, the subgraph of $\C_D(S_g)$ spanned by vertices of type N and S is isomorphic to a subgraph of $\C(S_g)$.  The main idea of the proof is to show that the type of a vertex is invariant under automorphisms of $\D(S_g)$, and so an automorphism of $\D(S_g)$ induces an automorphism of either the complex of curves or---in the absence of vertices of type N---an automorphism of the complex of separating curves $\C_{\check g(D)}(S_g)$.

The only upper set in $\D(S_g)$ containing vertices of type N is the entire set $\D(S_g)$ itself.  We deal with this case first since it is especially easy, and also because in the absence of vertices of type N a simplex in $\C_D(S_g)$ has a normal form representative that is unique (Lemma~\ref{lemma:div set normal form} below).


\subsection{Complexes with vertices of type N}

Before dispensing with the case where $D = \D(S_g)$ (equivalently, where $D$ contains an element of type N), we require one new idea.

\p{Subordinacy} We say that a vertex $v$ of $\C_{\D(S_g)}(S_g)$ is \emph{subordinate} to a vertex $w$ if $v$ has a representative that is homotopic to a subset of a representative of $w$ (the homotopy might take distinct components to a single component).  We make the following observations.
\begin{enumerate}
\item If $v$ and $w$ are distinct vertices of $\C_{\D(S_g)}(S_g)$ with $v$ subordinate to $w$ then $v$ has type N and $w$ has type M.
\item If $v$ is a vertex of $\C_{\D(S_g)}(S_g)$ that has type N, then there exists another vertex $w$ so that $v$ is subordinate to $w$.  
\item If $w$ is a vertex of $\C_{\D(S_g)}(S_g)$ that has type M then there is another vertex $v$ subordinate to $w$.
\end{enumerate}
In other words, in $\C_{\D(S_g)}(S_g)$ the type of a given vertex is completely determined by the subordinacy relation.

\begin{lemma}
\label{subord}
Suppose that $g \geq 1$ and let $\phi$ be an automorphism of $\C_{\D(S_g)}$.  If $v$ and $w$ are vertices of $\C_{\D(S_g)}$ with $v$ subordinate to $w$ then $\phi(v)$ is subordinate to $\phi(w)$.
\end{lemma}

\begin{proof}

We make the following claim: a vertex $v$ is subordinate to a vertex $w$ if and only if the star of $w$ is contained in the star of $v$.  Since automorphisms preserve stars, the lemma will follow from this.

For the forward direction of the claim, suppose that $v$ is subordinate to $w$.  As above, $v$ has type N.  Thus a vertex $u$ is connected to $v$ by an edge if and only if $u$ and $v$ have disjoint representatives.  Therefore, if a vertex $u$ is not connected to $v$ by an edge, then $u$ is not connected to $w$ by an edge (it intersects the component of $w$ corresponding to $v$); thus the star of $w$ is contained in the star of $v$.

Now suppose that $v$ is not subordinate to $w$.  If $v$ and $w$ are not joined by an edge, then $w$ is not contained in the star of $v$ and we are done.  If $v$ and $w$ are joined by an edge, then they have disjoint, nested representatives.  Since $v$ is not subordinate to $w$, the representative of $v$ has a component that is not homotopic to any component of $w$, that is, $v$ is represented by a non-peripheral dividing set in one of the two regions of $S_g$ determined by the representative of $w$.  It follows that there is a $\MCG(S_g)$-translate $v'$ of $v$ that is connected by an edge to $w$ and has essential intersection with $v$, so $v'$ is not connected by an edge to $v$.  This completes the proof. 
\end{proof}

\p{Finishing the proof in the easy case} The next proposition constitutes the special case of Theorem~\ref{theorem:multi k} in the case where $D =  \D(S_g)$.  

\begin{prop}
\label{prop:cd easy case}
Suppose that $g \geq 3$.  The natural map
\[  \MCG(S_g) \to \Aut \C_{\D(S_g)}(S_g) \]
is an isomorphism.
\end{prop}

\begin{proof}[Proof of Proposition~\ref{prop:cd easy case}]

We already explained that the type of a vertex in $\C_{\D(S_g)}(S_g)$ is completely determined by the subordinacy relation.  It follows from this and from Lemma~\ref{subord} that the vertices of type N, S, and M form three characteristic subsets of $\C_{\D(S_g)}(S_g)$ (meaning each subset is preserved by automorphisms).  Identifying $\C(S_g)$ with the subcomplex of $\C_{\D(S_g)}(S_g)$ spanned by vertices of type N and S, we thus obtain a homomorphism
\[ \Aut \C_{\D(S_g)}(S_g) \to \Aut \C(S_g) \]
given by restriction.  This map is injective because each vertex of type M is determined by the vertices of type N that are subordinate to it.  Thus the composition
\[ \Aut \C_{\D(S_g)}(S_g) \to \Aut \C(S_g) \stackrel{\cong}{\to} \MCG(S_g) \]
is injective.  It remains to check that the map $\MCG(S_g) \to \Aut \C_{\D(S_g)}(S_g)$ is a right inverse.  But this is true because the composition $\MCG(S_g) \to \Aut \C_{\D(S_g)}(S_g) \to \Aut \C(S_g)$ is equal to the natural map $\MCG(S_g) \to \Aut \C(S_g)$. 
\end{proof}


\subsection{Complexes without vertices of type N}

To prove Theorem~\ref{theorem:multi k}, it remains to deal with the case where $\C_D(S_g)$ has no vertices of type N.  For this case, there are three technical ingredients.  The first ingredient, which takes advantage of the assumption that there are no vertices of type N, is that every simplex of $\C_D(S_g)$ has a normal form representative that is unique up to isotopy in $S_g$ (Lemma~\ref{lemma:div set normal form}).  The second ingredient is the notion of a linear simplex in $\C_D(S_g)$, a special type of simplex in $\C_D(S_g)$.  The third ingredient is the idea of a exceptional edge, a certain type of configuration that only can involve vertices of type M.   Along the way we also give a topological characterization of upper sets in $\D(S_g)$.  Then we determine $\Aut \C_D(S_g)$ by showing that automorphisms of $\C_D(S_g)$ preserve the set of vertices of type S.   

\p{Normal form representatives.}  Let $\sigma$ be a simplex of $\C_D(S_g)$ with vertices $v_0,\dots,v_n$.  A \emph{normal form} representative  for $\sigma$ in $S_g$ is a collection of pairwise nested multicurves $m_0,  \dots, m_n$ where $m_i$ represents $v_i$.

\begin{lemma}\label{lemma:div set normal form}
Suppose that $D \subseteq \D(S_g)$ contains no elements of type N.  Every simplex of $\C_D(S_g)$ has a normal form representative, unique up to isotopy of $S_g$.
\end{lemma}

\begin{proof}

Let $\sigma$ be a simplex in $\C_D(S_g)$ with vertices $v_0, \dots, v_n$.  Choose representatives $m_i$ for the $v_i$ that are pairwise disjoint.  Such a collection $\cup m_i$ exists and is unique up to isotopy of $S_g$ and  reordering the curves in each parallel family; see, e.g., \cite[Section 1.2.4]{primer}.  It remains to show there is a unique choice of ordering of each parallel family so that the resulting representatives of the $v_i$ are pairwise nested, hence are in normal form.

We first deal with the case where $\sigma=\{v_0,v_1\}$ is an edge and then use this to prove the general case.  By the definition of $\C_D(S_g)$, the vertices $v_0$ and $v_1$ have nested representatives $m_0$ and $m_1$.  Such representatives are unique up to ambient isotopy in $S_g$ and reordering the connected components that come in parallel families.  Finally, there is a unique way to order a given pair of parallel curves: since $m_0$ has at least one component $c$ that is not parallel to any component of $m_1$, we must order the curves so that the $m_0$-curves lie on the side of $m_1$ containing $c$.

Now let $\sigma$ be an arbitrary simplex of dimension greater than one and again choose disjoint representatives $m_i$.  Consider one particular parallel family of curves in $\cup m_i$, and let $A$ be a (nonseparating) annulus in $S_g$ containing this family.  Arbitrarily name the two boundary components of $A$ as the left and right boundary components.  

For any $m_i$ and $m_j$ with components in $A$, we choose a normal form representative of the edge spanned by $v_i$ and $v_j$ (possibly modifying $m_i$ and $m_j$ in the process).  We can then declare $v_i$ to be to the left of $v_j$ in $A$ if the $m_i$-curve is closer to the left boundary of $A$ than the $m_j$-curve in $A$.  We need to show that this gives a total ordering on the set of vertices of $\sigma$ incident to $A$, that is, we need to show that the given relation is transitive.  

Suppose that $v_i$ is to the left of $v_j$ in $A$ and $v_j$ is to the left of $v_k$ in $A$.  This means that there are nested representatives $m_i$ and $m_j$ of $v_i$ and $v_j$ so that $m_i$ is to the left of $m_j$ in $A$, and nested representatives $m_j'$ and $m_k$ of $v_j$ and $v_k$ so that $m_j'$ is to the left of $m_k$ in $A$.  By the uniqueness mentioned earlier, we can apply an ambient isotopy to $m_j'$ and $m_k$ so that $m_j'$ is equal to $m_j$.  In particular, $m_j$---hence $m_i$---is to the left of $m_k$ in $A$, as desired.

After rearranging each parallel family of curves according to the above ordering, the resulting representative of $\sigma$ is a normal form representative, since a representative of a simplex is in normal form if and only if it restricts to a normal form representative for each edge.  This completes the proof.
\end{proof}

\p{A characterization of upper sets.} Assume that $\C_D(S_g)$ has no vertices of type N.  Let $\sigma = \{v_0,\dots,v_n\}$ be a simplex of $\C_D(S_g)$.  As a consequence of Lemma~\ref{lemma:div set normal form} it makes sense to say that a vertex $v_j$ lies \emph{between} $v_i$ and $v_k$ in $S_g$; specifically, this means that for any normal form representative $\{m_0,\dots,m_n\}$, the dividing sets $m_i$ and $m_k$ lie on opposite sides of $m_j$.  Equivalently, we can say that $v_i$ and $v_j$ lie \emph{on the same side} of $v_k$.

We say that a subset $D$ of $\D(S_g)$ is \emph{closed under separations} if whenever $u$ and $w$ are homotopy classes of dividing sets representing elements of $D$ and $v$ is a homotopy class of dividing sets lying between $u$ and $w$ we have that $v$ represents an element of $D$.  

\begin{lemma}
\label{upper set}
Let $g \geq 0$. Let $D \subseteq \D(S_g)$  be a subset with no element of type N.  Then $D$ is an upper set if and only if it is closed under separations.
\end{lemma}

\begin{proof}

Let $D \subseteq \D(S_g)$.  Suppose first that $D$ is an upper set.   Let $u$ and $w$ be homotopy classes of dividing sets representing elements of $D$ and so that $v$ lies between $u$ and $w$; for concreteness we imagine that $u$ lies to the left of $v$ and $w$ lies to the right.  We will show that $v$ is either larger than $u$ or larger than $w$ in the partial order on $\D(S_g)$.  Suppose that $v$ is not larger than $w$.  Then the regions of $S_g$ to the left and right of $v$ have genus $g_1$ and $g_2$, respectively, with $g_1 < g_2$.  It follows that the regions of $S_g$ to the left and right of $u$ have genus $h_1$ and $h_2$, respectively, with $h_1 \leq g_1$ and $h_2 \geq g_2$.  It follows that $u \preceq v$.  Thus $v$ represents an element of $D$ and we have shown that $D$ is closed under separations.

For the other direction suppose that $D$ is closed under separations.  Let $u$ and $v$ be homotopy classes of dividing sets with $u \in D$ and with $u \preceq v$.  We would like to show that $v$ represents an element of $D$.  Say that $u$ lies to the left of $v$.  Since $u \preceq v$, there is a $\MCG(S_g)$-translate $w$ of $u$ that lies to the right of $v$.  This $w$ represents the same element of $D$ as $u$.  Since $D$ is closed under separations, $v$ represents an element of $D$, as desired.
\end{proof}

\p{Linear simplices.}  We say that a simplex $\sigma$ of $\C_D(S_g)$ is \emph{linear} if we can label the vertices of $\sigma$ by $v_0,\dots,v_n$ in such a way that $v_j$ lies between $v_i$ and $v_k$ whenever $i < j < k$.  We can think of a linear simplex as a sequence of cobordisms between 1-manifolds, starting and ending with the empty manifold.  We will refer to the vertices $v_0$ and $v_n$ as \emph{extreme} vertices.

\begin{lemma}
\label{lemma:linear}
Let $g \geq 0$.  Suppose that $D \subseteq \D(S_g)$ contains no elements of type N.  Let $D \subseteq \D(S_g)$, let $\sigma$   be a linear simplex of $\C_D(S_g)$ with ordered vertices $v_0,\dots,v_n$ and let $\phi$ be an automorphism of $\C_D(S_g)$.  Then $\phi(\sigma)$ is a linear simplex of $\C_D(S_g)$ with ordered vertices $\phi(v_0),\dots,\phi(v_n)$.
\end{lemma}

\begin{proof}

Let $u$, $v$, and $w$ be three vertices of an arbitrary simplex of $\C_D(S_g)$.  We will show that $v$ lies between $u$ and $w$ if and only if the link of $v$ in $\C_D(S_g)$ is contained in the union of the links of $u$ and $w$; in other words, if a vertex $z$ is not connected by an edge to either $u$ or $w$ then it is not connected by an edge to $v$ either.  This will imply the lemma.

First suppose that $v$ lies between $u$ and $w$ in $S_g$.  If $z$ is any vertex of $\C_D(S_g)$ connected by an edge to $v$, then $z$ is also connected by an edge to one of $u$ and $w$, specifically, the one lying on the opposite side of $v$ from $z$.  This completes the forward direction of the claim.

For the other direction, suppose $v$ does not lie between $u$ and $w$.  Let $\{m_u,m_v,m_w\}$ be a normal form representative for the simplex $\{u,v,w\}$.  Both $m_u$ and $m_w$ must lie on the same side of $m_v$ and each must have a component that is not homotopic into $m_v$.  We can thus find a vertex $u'$ of $\C_D(S_g)$ in the $\MCG(S_g)$-orbit of $u$ that is connected to $v$ but not to $u$ or $w$ (because $u'$ intersects $u$ and $w$ nontrivially).  This completes the proof.
\end{proof}

\begin{figure}
\includegraphics[scale=.5]{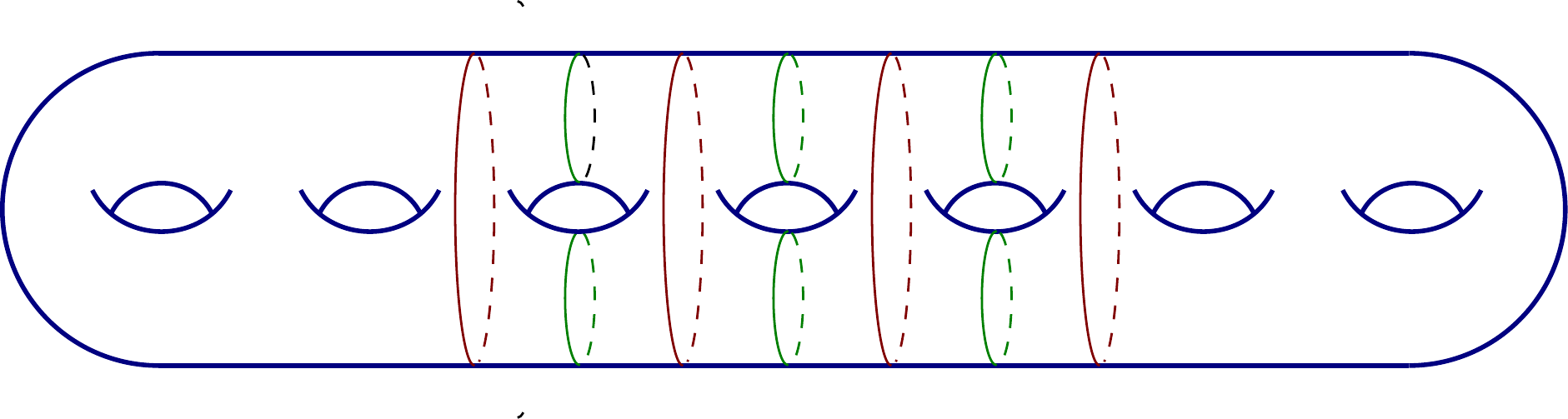}
\caption{A specific type of linear simplex}
\label{fig:linear}
\end{figure}

\begin{lemma}
\label{lemma:four}
Suppose that $D \subseteq \D(S_g)$ is an upper set that contains no elements of type N and suppose that $g \geq \max\{3 \check g(D) + 1,5\}$.  If $v$ is a vertex of $\C_D(S_g)$ of type S then $v$ is an extreme vertex of linear simplex of $\C_D(S_g)$ with five vertices.
\end{lemma}

\begin{proof}

There is a linear simplex $\tau = \{v_0,\dots,v_n\}$ in $\C_D(S_g)$ where $n$ is even, where each $v_i$ with $i$ even is of type $S$, and where $\tau$ is maximal with respect to these properties.  The simplex $\tau$ is unique up to the action of $\MCG(S_g)$; see Figure~\ref{fig:linear}.  Note that the existence of $\tau$ uses the assumption that $D$ is an upper set.

We claim that $n \geq 6$.  Indeed, since $\tau$ is maximal, both $v_0$ and $v_n$ have genus $\check g(D)$.  So the region between $v_0$ and $v_n$ has genus $g-2\check g(D)$.  Thus $\tau$ has $g-2\check g(D)+1$ vertices of type S and $g-2\check g(D)$ vertices of type $M$ (each with two components), for a total of $2g-4\check g(D)+1$ vertices, so $n=2g-4\check g(D)$.  Since $g \geq 3 \check g(D) + 1$ we have $n \geq 2\check g(D) + 2$.  It immediately follows that $n \geq 6$ when $\check g(D) \geq 2$.  When $\check g(D) = 1$ it follows from the assumption that $g \geq 5$ and the equality $n=2g-4\check g(D)$ that $n \geq 6$.

Every vertex of type S lies in the orbit of one of the vertices $v_i$ with $i$ even and $0 \leq i \leq n/2$.  And so it suffices to show that each such $v_i$ satisfies the statement of the lemma.  But since $n \geq 6$ it is the case that for all $i$ even with $0 \leq i \leq n/2$ we have $i + 4 \leq n$ and so $\{v_i,\dots,v_{i+4}\}$ is the desired linear simplex.
\end{proof}

\begin{figure}
  \labellist
\small\hair 2pt
 \pinlabel {$v$} [ ] at 198 -10
 \pinlabel {$w$} [ ] at 246 -10
 \endlabellist  
\includegraphics[scale=.45]{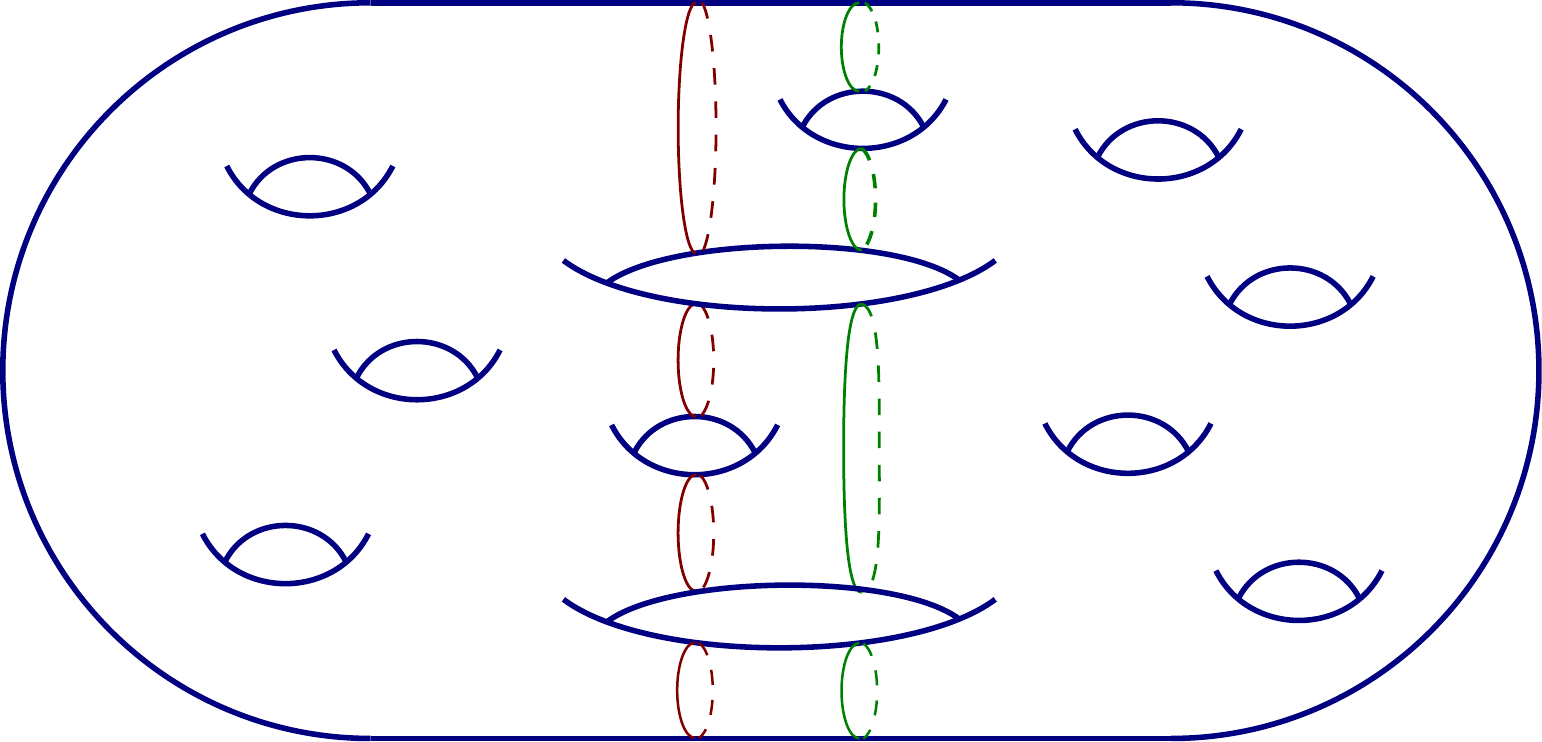}
\caption{A typical exceptional edge}
\label{fig:example}
\end{figure}

\p{Exceptional edges}  We now explain the third tool required for the proof of Theorem~\ref{theorem:multi k}.  We say that vertices $v_1$ and $v_2$ of $\C_D(S_g)$ form an \emph{exceptional edge} if they are connected by an edge in $\C_D(S_g)$ and---after taking a normal form representative for the edge---the subsurface between $v_1$ and $v_2$ is the disjoint union of some number of annuli and at least two pairs of pants.  Note that in an exceptional edge both vertices must have type M.  An example of an exceptional edge is shown in Figure~\ref{fig:example}.

\begin{lemma}
\label{lemma:tight edges}
Suppose $D \subseteq \D(S_g)$ is an upper set that contains no elements of type N.   The image of an exceptional edge under an automorphism of $\C_D(S_g)$ is an exceptional edge.
\end{lemma}

\begin{proof}

We will show that an edge in $\C_D(S_g)$ is exceptional if and only if its link is the join of a nonempty finite graph with some other (possibly empty) graph (see Section~\ref{sec:sep} for the definition of a join).  Since the latter property is clearly preserved by automorphisms, the lemma will follow.

Let $\{v_1,v_2\}$ be an edge in $\C_D(S_g)$.  Normal form representatives for $v_1$ and $v_2$ divide $S_g$ into three regions, $L$, $C$, and $R$ (for left, center, and right).  Say that $C$ is the region lying between $v_1$ and $v_2$.  The regions $L$ and $R$ are connected but $C$ may not be.  We think of $v_1$ as lying to the left of $v_2$ so that $L$ is bounded by $v_1$.  

Suppose first that $\{v_1,v_2\}$ is an exceptional edge, so the subsurface $C$ is the disjoint union of $n \geq 2$ pairs of pants $P_1,\dots,P_n$ and some number of annuli.  There are $2^n$ vertices of the star of $\{v_1,v_2\}$ supported in $C$: for each $i$ we make a choice between the left side of $P_i$ or the right and for each annulus we include the core curve.  If we choose the left side of each pair of pants, we obtain $v_1$ and if we choose the right side in each case we obtain $v_2$.  Thus, there are $2^n-2>0$ vertices of the link of $\{v_1,v_2\}$ supported in $C$; call this set of vertices $F$.  Since all other vertices of the link of $\{v_1,v_2\}$ are supported in $L$ or $R$, the link of $\{v_1,v_2\}$ is the join of $F$ with the graph spanned by those vertices.

Now suppose that $\{v_1,v_2\}$ is an edge of $\C_D(S_g)$ whose link is the join of a finite graph $F$ with some other (possibly empty) graph.  Assume that $\{v_1,v_2\}$ is not an exceptional edge.  We must show that $F$ is empty.  We will repeatedly use the following fact: to show that a vertex $w$ is not contained in $F$ it is enough to show that there are infinitely many vertices of the link of the edge $\{v_1,v_2\}$ to which $w$ is not connected by an edge.  

Suppose that $w$ is a vertex of the link of $\{v_1,v_2\}$ represented in $L$ (or $R$).  Then $w$ must have at least one component that is not peripheral in $L$ (or $R$).  But then there are infinitely many $\MCG(S_g)$-translates of $w$ that lie in the link of $\{v_1,v_2\}$ and have nonzero intersection with $w$.  As above, this implies $w$ does not lie in $F$.  Applying the same argument to the region $C$, we see that all elements of $F$ must be represented by dividing sets that are peripheral in $C$.

If $C$ has fewer than two components that are not annuli then the only vertices of the star of $\{v_1,v_2\}$ realized as peripheral dividing sets in $C$ are $v_1$ and $v_2$ and we have succeeded in showing that $F$ is empty.  So assume that $C$ has at least two components that are not annuli.  

Let $w$ be a vertex of $F$.  Again, we may assume that $w$ is peripheral in $C$.  Since $\{v_1,v_2\}$ is not exceptional, there is some component $C_0$ of $C$ that is not an annulus or a pair of pants.  Since $C_0$ is not an annulus or a pair of pants there is a vertex $u$ of the link of $\{v_1,v_2\}$ so that $u$ is supported in $C$ and some component of a representative of $u$ is non-peripheral in $C_0$.  In $C_0$ it makes sense to say that $u$ lies to the left or right of $w$.  Say it is to the left.  This means that in $C_0$ the vertex $w$ is parallel to the $v_2$-side (the right-hand side) of $C_0$.

Since $w$ is not equal to $v_1$ or $v_2$, it must have some component that lies on the $v_1$-side (the left-hand side) of some non-annular component $C_1$ of $C$.  After possibly changing $u$ we may assume that the intersection of $u$ with $C_1$  is the collection of curves parallel to the $v_2$-side (the right-hand side) of $C_1$.  By construction this $u$ is not connected to $w$ by an edge.  Also, since $u$ has a component in the interior of $C_0$ there are infinitely many $\MCG(S_g)$-translates of $u$ that are not connected to $w$ by an edge.  We have thus shown that $w$ cannot lie in $F$, and so $F$ is empty.
\end{proof}

We already said that the two vertices in an exceptional edge must be of type M.  The next lemma is a partial converse to this statement.  

\begin{lemma}
\label{lemma:char}
Suppose that $D \subseteq \D(S_g)$ contains no elements of type N.  Let $D \subseteq \D(S_g)$ be an upper set with $g \geq 3 \check g (D) + 1$.  Let $v$ be a vertex of $\C_D(S_g)$ that is of type M and is an extreme vertex of a linear simplex $\sigma$ with five vertices.  Then $v$ lies in an exceptional edge.
\end{lemma}

\begin{proof}

Denote the ordered vertices of $\sigma$ by $v=v_0,\dots,v_4$.  Choose a normal form representative for $\sigma$ and let $Q$ denote the region of $S_g$ lying between the dividing sets representing $v_0$ and $v_4$.  We will think of $Q$ as lying to the right of $v$.

We must find a vertex $w$ so that $v$ and $w$ form an exceptional edge.  To do this, we will find two disjoint pairs of pants $P_1$ and $P_2$ that lie in $Q$ and are adjacent to $v$; the vertex $w$ is then obtained from $v$ by replacing the left-hand side of each $P_i$ with the right-hand side.  It follows from Lemma~\ref{upper set} that $w$ will indeed represent a vertex of $\C_D(S_g)$.

We can think of $Q$ as the union of four cobordisms connecting the vertices of $\sigma$.  As such we have $\chi(Q) \leq -4$.  We can assume without loss of generality that $\chi(Q) = -4$: since $D$ is an upper set, we can by Lemma~\ref{upper set} replace $\sigma$ (if necessary) by a linear simplex with five vertices where each of the four corresponding cobordisms is the disjoint union of one pair of pants with some number of annuli.

Let $Q_0$ denote the union of the non-annular components of $Q$.  If $Q_0$ is not connected, then we can find the desired $w$ by choosing pair of pants in two different components of $Q_0$.

\begin{figure}
\includegraphics[scale=.25]{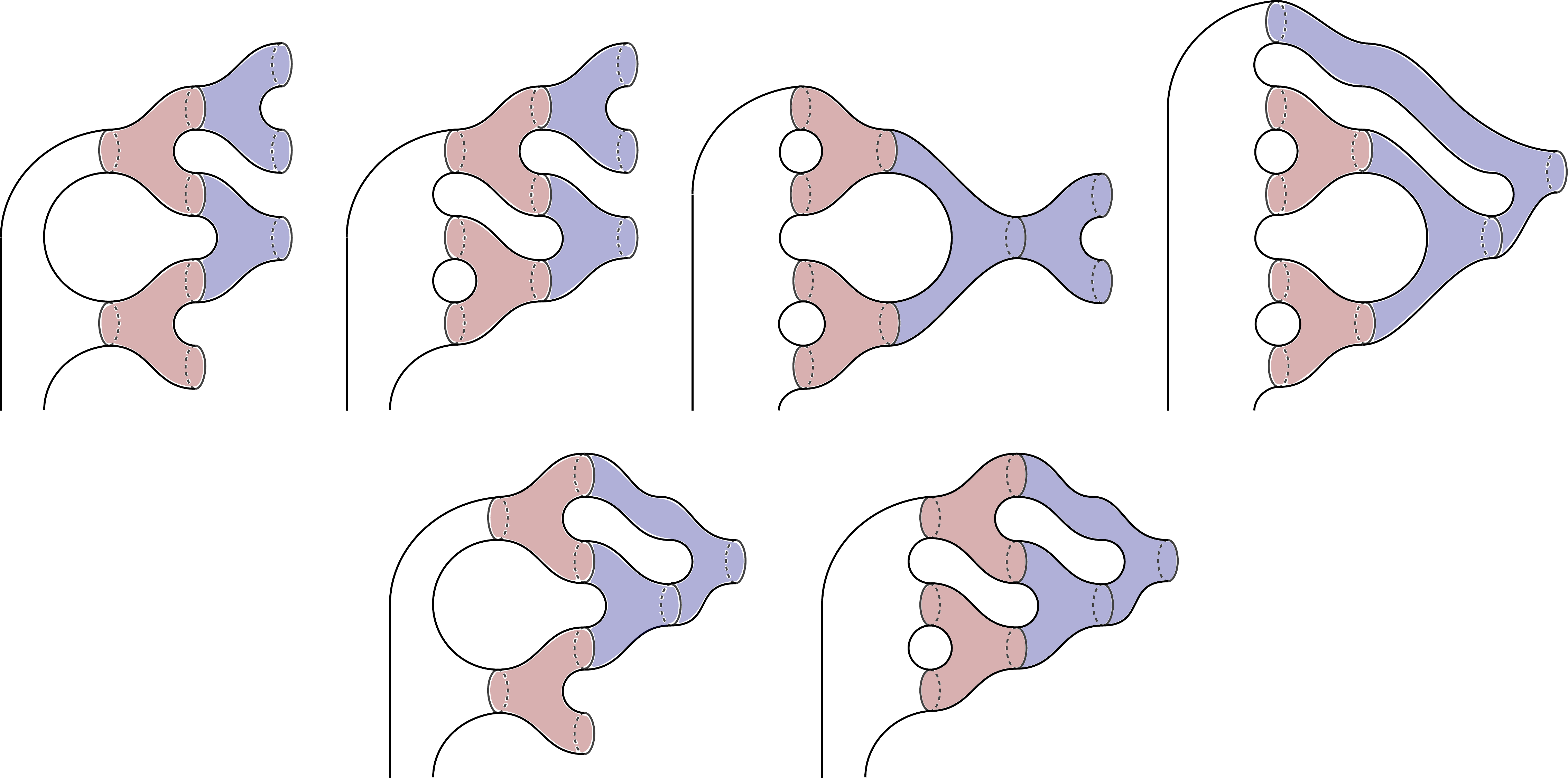}
\caption{\emph{Top:} finding $w$ when $Q$ is a sphere with six boundary components; \emph{Bottom:} finding $w$ when $Q$ is a torus with four boundary components}
\label{fig:w}
\end{figure}

Now suppose that $Q_0$ is connected.  There are three cases for $Q_0$: a surface of genus zero with six boundary components, a surface of genus one with four boundary components, or a surface of genus two with two boundary components.  If $Q_0$ has more than one boundary component parallel to $v$ then we can easily find the desired pairs of pants $P_1$ and $P_2$, hence the desired vertex $w$; see Figure~\ref{fig:w}.  (Note that we do note need to consider here the case where $Q_0$ has genus two since $Q_0$ does not have all of its boundary components parallel to $v$.)

So it remains to consider the case where $Q_0$ is one of the three surfaces described in the previous paragraph and $Q_0$ has a single component of its boundary parallel to $v$.  We treat this case by reducing to the previous cases.

Say that $v$ has type $(k,n)$ and that the region $R$ of $S_g$ that is determined by $v$ and does not contain $Q_0$ has genus $k$.  If $Q_0$ is a sphere with six boundary components, then $v_4$ has type $(k,n+4)$.  Since $n \geq 2$, we can form another vertex of type $(k,n+4)$ by gluing a sphere with five boundary components to one component of the boundary of $R$ and a pair of pants to another component of the boundary of $R$ (the new dividing set is the boundary of the union of $R$ and the two additional spheres).  This new dividing set represents a vertex $v_4'$ of $\C_D(S_g)$ since it has the same type as $v_4$.  Also, since the region between $v$ and $v_4'$ has more than one component with negative Euler characteristic, we have reduced to a previous case.

Similarly, if $Q_0$ is a surface of genus one with four boundary components, then $v_4$ has type $(k+1,n+2)$.  In this case we can obtain a vertex $v_4'$ of the same type by gluing a sphere with six boundary components to $R$ along two of the six boundary components.  Again, this is a case we have already dealt with.

Finally if $Q_0$ is a surface of genus two with two boundary components, then $v_4$ has type $(k+2,n)$, and in this case $v_4'$ is obtained by gluing a sphere with six boundary components to $R$ along three of its boundary components, another case we have previously treated.
\end{proof}

\p{Finishing the proof} For the proof of Theorem~\ref{theorem:multi k}, we say that a vertex $v$ of $\C_D(S_g)$ is \emph{1-sided} if all vertices in the link of $v$ lie to one side of $v$.  Similarly, we say that $v$ is \emph{2-sided} if there are vertices of the link of $v$ lying on different sides of $v$.

\begin{proof}[Proof of Theorem~\ref{theorem:multi k}]

By Proposition~\ref{prop:cd easy case} we may assume that $\C_D(S_g)$ does not contain vertices of type N.  And by Lemma~\ref{inj}, it suffices to show that the natural map $\MCG(S_g) \to \Aut \C_D(S_g)$ is surjective.  Let $\phi \in \Aut \C_D(S_g)$.  

We claim that $\phi$ preserves the set of vertices of type S.  First assume that $g \geq 5$.  By Lemma~\ref{lemma:four}, if a vertex of $\C_D(S_g)$ is not an extreme vertex of some linear simplex with five vertices then it is a vertex of type M.  The set of extreme vertices of linear simplices with five vertices is preserved by automorphisms (Lemma~\ref{lemma:linear}).  Therefore, it suffices to show that if a vertex of type M is an extreme vertex of a linear simplex with five vertices then its image under $\phi$ is a vertex of type M.  Let $v$ be such a vertex.  By Lemma~\ref{lemma:char}, $v$ is a vertex of an exceptional edge.  So by Lemma~\ref{lemma:tight edges} the vertex $\phi(v)$ is a vertex of an exceptional edge.  Thus $\phi(v)$ is of type M, giving the claim in the case $g \geq 5$.

It remains to prove the claim in the case where $g=4$ and $\check g(D) = 1$.  The first step is to show that automorphisms of $\C_D(S_4)$ preserve the set of vertices of type S that have genus one, that is, the ones of type $(1,1)$.  The only vertices than can be extreme vertices of linear simplices with five vertices are those of type $(0,3)$ and of type $(1,1)$.  By the argument of the previous paragraph, these two types are each preserved (although type $(0,3)$ may not be present in the complex).

The second step for the proof of the claim in the $g=4$ case is to show that automorphisms of $\C_D(S_4)$ preserve the set of vertices of type S that have genus two.  These vertices are distinguished by the following property: if we have a simplex with five vertices, four of which are vertices of type S that have genus one, then the fifth vertex must be a vertex of type S that has genus two.  This completes the proof of the claim.  

The claim implies that $\phi$ restricts to an automorphism $\bar \phi$ of $\C_{\check g(D)}(S_g)$, regarded as a subcomplex of $\C_D(S_g)$.  By Theorem~\ref{theorem:sep k}, there exists some $f \in \MCG(S_g)$ so that the automorphism of $\C_{\check g(D)}(S_g)$ induced by $f$ is $\bar \phi$.  We would like to show that the automorphism of $\C_D(S_g)$ induced by $f$ is $\phi$.  

To do this, it suffices to show that an automorphism of $\C_D(S_g)$ that restricts to the identity on the subcomplex $\C_{\check g(D)}(S_g)$ must itself be the identity.  We proceed by induction on distance in $\C_D(S_g)$ from $\C_{\check g(D)}(S_g)$; since $\C_D(S_g)$ is connected by assumption, this will prove the theorem.

We assume for induction that $\phi$ restricts to the identity on all vertices of $\C_D(S_g)$ that have distance at most $k$ from $\C_{\check g(D)}(S_g)$.  The base case $k=0$ is true by assumption.

We perform the inductive step in two stages, first for 1-sided vertices and then for 2-sided vertices.  Let $v$ be a 1-sided vertex of $\C_D(S_g)$ that has distance $k+1$ from $\C_{\check g(D)}(S_g)$.  Let $w$ be any vertex of $\C_D(S_g)$ that is connected by an edge to $v$ and has distance $k$ from $\C_{\check g(D)}(S_g)$.  There are $\MCG(S_g)$-translates of $w$ that fill the region of $S_g$ that lies to one side of $v$ and contains $w$.  Since $\C_{\check g(D)}(S_g)$ is invariant under the action of $\MCG(S_g)$, all of these vertices have distance $k$ from $\C_{\check g(D)}(S_g)$ and so by the inductive hypothesis they are fixed by $\phi$.  As $v$ is 1-sided, it is the unique vertex of $\C_D(S_g)$ connected by edges to all of these translates of $w$ and so $\phi$ must fix $v$ as well.

Now let $v$ be a 2-sided vertex of $\C_D(S_g)$ that has distance $k+1$ from $\C_{\check g(D)}(S_g)$.  Let $w$ be any vertex of $\C_D(S_g)$ that is connected by an edge to $v$ and has distance $k$ from $\C_{\check g(D)}(S_g)$.  The vertex $w$ lies to one side of $v$, and the $w$-side of $v$ is filled by $\MCG(S_g)$-translates of $w$.  As above, all of these vertices have distance $k$ from $\C_{\check g(D)}(S_g)$ and so are fixed by $\phi$.  Let $u$ be any 1-sided vertex of $\C_D(S_g)$ that lies on the other side of $v$.  Note that the distance from $u$ to $\C_{\check g(D)}(S_g)$ is at most $k+1$ since $u$ is connected by an edge to $w$.  By the previous paragraph, $u$ is fixed by $\phi$.  Moreover, there are $\MCG(S_g)$-translates of $u$ that fill the $u$-side of $v$, and all of these vertices are also fixed by $\phi$.  The vertex $v$ is the unique vertex connected by edges to all of these vertices, and so $v$ is also fixed by $\phi$, and we are done.
\end{proof}


\section{Complexes of regions}
\label{sec:reg}

In this section we apply our theorem about dividing sets (Theorem~\ref{theorem:multi k}) in order to prove Theorem~\ref{main:complex}, which states that if $\C_A(S_g)$ is a connected complex of regions with no holes or corks and with a small vertex then the automorphism group of $\C_A(S_g)$ is isomorphic to $\MCG(S_g)$.

The basic idea of the proof of Theorem~\ref{main:complex} is to relate maximal joins in $\C_A(S_g)$ to dividing sets in $S_g$ (Lemma~\ref{lemma:bijection}).  This relationship is somewhat strained in the case where $\C_A(S_g)$ has nonseparating annular vertices.  So similar to Section~\ref{sec:div} we will break this off as a special case.

We say that a vertex of $\C_A(S_g)$ has...
\begin{enumerate}
\item \emph{type N} if it is represented by a nonseparating annulus, 
\item \emph{type S} if it is represented by a separating annulus, and
\item \emph{type P} if it is represented by a nonseparating pair of pants.
\end{enumerate}

There is a partial order on the set of vertices of $\C_A(S_g)$ whereby $a \preceq b$ if the link of $b$ is contained in the link of $a$.  We say that a vertex is \emph{$\preceq$-minimal} if it is minimal with respect to this ordering. 

\subsection{Complexes with vertices of type N}

In this section, we prove Theorem~\ref{main:complex} in the special case where $\C_A(S_g)$ has vertices corresponding to nonseparating annuli.  

\begin{lemma}
\label{nonsep ann}
Let $\C_A(S_g)$ be a complex of regions with vertices of type N. 
\begin{enumerate}
\item Any automorphism of $\C_A(S_g)$  preserves the set of vertices of type N.
\item Any automorphism of $\C_A(S_g)$  preserves the set of vertices of type P.
\end{enumerate}
\end{lemma}

\begin{proof}

We claim that a vertex of $\C_A(S_g)$ is of type N if and only if
\begin{enumerate}
\item it is a $\preceq$-minimal vertex of $\C_A(S_g)$ and
\item its link in $\C_A(S_g)$ is not a join.
\end{enumerate}
The first statement of the lemma follows from the claim since an automorphism of $\C_A(S_g)$ preserves these two properties.

For the forward direction, assume $v$ is of type N.  The vertex $v$ is $\preceq$-minimal because a representative of $v$ contains no proper subsurfaces.  Also, the link of $v$ is not a join because for any two vertices of the link of $v$ we can find a $\MCG(S_g)$-translate of $v$ that is not connected by an edge to either.

For the other direction of the claim, assume that $v$ is not of type N.  Let $R$ be a representative of $v$.  If $R$ is separating then each complementary region to $R$ supports at least one vertex of type N, and so the link of $v$ is a join.  So we may assume that $R$ is nonseparating.  If the boundary of $R$ is connected, then $v$ is not $\preceq$-minimal, as there are vertices of type N that are smaller in the partial order.  So we may further assume that $R$ has disconnected boundary.  If $R$ is the complementary region to a nonseparating annulus, then clearly $v$ is not $\preceq$-minimal.  The remaining case is that $R$ is nonseparating with disconnected boundary and no two components of the boundary are parallel.  In this case, the link of $v$ is a join.  Indeed, any vertex of type N corresponding to a component of the boundary of $R$ is a cone point.  This completes the proof of the claim and hence the first statement of the lemma.

For the second statement, we claim that a vertex $v$ of $\C_A(S_g)$ is of type P if and only if there is a triangle $\sigma = \{w_1,w_2,w_3\}$ so that each $w_i$ is of type N and so that $v$ and $\sigma$ have equal stars.  The second statement follows from this and the first statement.

For the forward direction, let $v$ be a vertex of type P.  We take $\sigma$ to be the triangle corresponding to the boundary of a representative of $v$.  Since the only proper regions in this pair of pants are those corresponding to the vertices of $\sigma$, it follows that $\sigma$ and $v$ have equal stars.  

For the other direction, assume we have a vertex $v$ and a triangle $\sigma$ as in the claim.  Let $Q$ and $R$ be representatives of $\sigma$ and $v$.  Since $v$ and $\sigma$ have equal stars, we can take $Q$ and $R$ to be disjoint.  Also, each component of $Q$ must be parallel to the boundary of $R$, for otherwise we could find a vertex in the star of $v$ but not $\sigma$.  Conversely, each component of the boundary of $R$ must be parallel to some component of $Q$, for otherwise we could find a vertex in the star of $\sigma$ but not $v$.  Thus, $R$ has exactly three boundary components, each of which is a nonseparating curve in $S_g$.  It follows that $R$ is nonseparating.  If $R$ is not a pair of pants, then we could find a vertex of type N that is in the star of $\sigma$ but not $v$.  This completes the proof of the claim, hence the lemma.
\end{proof}

\p{Finishing the proof in the easier case} We are now ready to prove Theorem~\ref{main:complex} in the case where $\C_A(S_g)$ contains vertices of type N.  In the proof we say that a region of $S_g$ is of type P if it is a nonseparating pair of pants (we make the distinction between a region and a vertex here because $\C_A(S_g)$ may or may not have vertices of type P).

Let $\N(S_g)$ denote the \emph{complex of nonseparating curves} for $S_g$, that is, the subcomplex of the complex of curves $\C(S_g)$ spanned by vertices of type N.

\begin{prop}
\label{prop:ca easy}
Let $g \geq 3$ and suppose that $\C_A(S_g)$ contains vertices of type N.  Then the natural map
\[ \MCG(S_g) \to \Aut \C_A(S_g) \]
is an isomorphism.
\end{prop}

\begin{proof}

By Lemma~\ref{inj} the map $\Psi : \MCG(S_g) \to \Aut \C_A(S_g)$ is injective.  So it suffices to show that $\Psi$ is surjective.  We will construct a right inverse.  

The subcomplex of $\C_A(S_g)$ spanned by the vertices of type N is naturally isomorphic to the complex of nonseparating curves $\N(S_g)$ and so we may regard $\N(S_g)$ as a subcomplex of $\C_A(S_g)$.  By Lemma~\ref{nonsep ann}, there is a map
\[
\Aut \C_A(S_g) \to \Aut \N(S_g)
\]
given by restriction.  Also Irmak proved that the natural map
\[
\MCG(S_g) \to \Aut \N(S_g)
\]
is an isomorphism \cite{irmak}.  Consider the composition
\[
\Omega: \Aut \C_A(S_g) \to \Aut \N(S_g) \stackrel{\cong}{\to} \MCG(S_g).
\]
We would like to show that $\Psi \circ \Omega$ is the identity.  Let $\phi \in \Aut \C_A(S_g)$.  Then $\Omega(\phi)$ is a mapping class $f$ whose action on $\N(S_g)$ agrees with the restriction of $\phi$.  To show that $\Psi \circ \Omega(\phi) = \phi$ we must show that $\phi$ is determined by its restriction, or, that any element of $\Aut \C_A(S_g)$ restricting to the identity in $\Aut \N(S_g)$ must itself be the identity.

So suppose $\phi \in \Aut \C_A(S_g)$ restricts to the identity in $\Aut \N(S_g)$.  Let $v$ be a vertex of $\C_A(S_g)$ that is not of type N.  We must show that $\phi(v)=v$.

First suppose that $v$ is of type P.  By the second statement of Lemma~\ref{nonsep ann}, $\phi(v)$ is also of type P.  But since $g \geq 3$ such a vertex is clearly determined by the vertices of type N in its link, and so we are done in this case.

Now suppose that $v$ is not of type P.  Again by Lemma~\ref{nonsep ann} the same is true for $\phi(v)$.  Let $R$ be a representative of $v$.  Each complementary region to $R$ is nonseparating.  If a complementary region is not of type P, then it is filled by vertices of type N.  If a complementary region is of type P, then its three boundary components all correspond to vertices of type N.  Thus, if we take the subsurface of $S_g$ filled by vertices of type N in the link of $v$, then $R$ is the unique complementary region that is not of type P.  Thus $\phi(v)$ is represented in $R$.  In other words, the support of $v$ cannot become larger under $\phi$.  Since $\phi$ is invertible, we conclude that the support of $v$ cannot become smaller under $\phi$ either.  This means that $\phi(v)$ is represented by $R$, that is, $\phi(v)=v$, as desired.
\end{proof}

\subsection{Complexes without vertices of type N}

To complete the proof of Theorem~\ref{main:complex} it remains to treat the case where $\C_A(S_g)$ does not have any vertices of type N.  

\p{From regions to dividing sets.} Let $A \subseteq \R(S_g)$ and let $D \subseteq \D(S_g)$.  We say that vertices of $\C_A(S_g)$ and $\C_D(S_g)$ are \emph{compatible} if they have disjoint representatives.  Also, we say that an element $d$ of $\D(S_g)$ is \emph{compatible} with $A$ if both regions of $S_g$ determined by a representative of $d$ contain representatives of elements of $A$ (the representatives are allowed to be peripheral).  We then define
\[ \delta A = \{ d \in \D(S_g) \mid d \text{ is compatible with } A \}. \]
For any $A$, the set $\delta A$ is clearly closed under separations and so by Lemma~\ref{upper set} it is an upper set.
Most of our work here is devoted to showing that there is a well-defined map 
\[ \Aut  \C_A(S_g) \to \Aut \C_{\delta A} (S_g), \]
so that we may apply Theorem~\ref{theorem:multi k}.

We now define a function
\[ \Phi : \{ \text{vertices of } \C_{\delta A}(S_g) \} \to \{ \text{subcomplexes of } \C_A(S_g) \}. \]
For a vertex $v$ of $\C_{\delta A}(S_g)$, the image $\Phi(v)$ is defined to be the full subcomplex of $\C_A(S_g)$ spanned by all vertices that are compatible with $v$.

\p{Two join decompositions.} For a vertex $v$ of $\C_{\delta A}(S_g)$, we will now define two different join decompositions of $\Phi(v)$, one topological and one combinatorial.  In the next lemma we will show that the two decompositions are the same.

First, we define the \emph{left/right decomposition} $V_L \ast V_M \ast V_R$ of $\Phi(v)$ as follows. A representative of $v$ divides $S_g$ into two complementary regions, which we arbitrarily label as $L$ and $R$ (for left and right).  The complexes $V_L$ and $V_R$ are the subcomplexes of $\C_A(S_g)$ spanned by the vertices corresponding to non-peripheral subsurfaces of $L$ and $R$.  The complex $V_M$ is the subcomplex spanned by the vertices corresponding to annuli parallel to a representative of $v$; since $\C_A(S_g)$ does not have vertices of type N, the complex $V_M$ is either empty or is a single vertex of type S.  It is clear that $V_L \ast V_M \ast V_R$ is a join decomposition of $\Phi(v)$.

Next, the \emph{complete join decomposition} of $\Phi(v)$ is a decomposition $V_1 \ast \cdots \ast V_n$ with the property that no $V_i$ can be decomposed as a join.  We have $n \leq 3g-3$ since there is an upper bound of $3g-4$ to the dimension of a simplex in $\C_A(S_g)$.

We say that a subcomplex of a simplicial complex is a \emph{maximal join} if it is a join that is not contained in any other join.

\begin{lemma}
\label{lemma:bijection}
Let $\C_A(S_g)$ be a complex of regions without vertices of type N and without isolated vertices.  
\begin{enumerate}
\item The map $\Phi$ gives a bijection
\[ \Phi : \{ \text{vertices of } \C_{\delta A}(S_g) \} \to \{ \text{maximal joins in } \C_A(S_g) \}. \]
\item For each vertex $v$ of $\C_{\delta A}(S_g)$ the left/right decomposition of $\Phi(v)$ is the same as the complete join decomposition.
\end{enumerate}
\end{lemma}

\begin{proof}

We start with the first statement.  First we will show that each $\Phi(v)$ is a maximal join and then we will show that each maximal join is in the image of $\Phi$ (it is clear that $\Phi$ is injective).  

Let $v$ be a vertex of $\C_{\delta A}(S_g)$.  We will begin by showing that the left/right decomposition of $\Phi(v)$ is a nontrivial join decomposition.  Assume for contradiction that the left/right decomposition of $\Phi(v)$ is a trivial join.  By the definition of $\delta A$ there are vertices of $\C_A(S_g)$ corresponding to (possibly peripheral) subsurfaces of both sides of $v$, and so the assumption implies that $\Phi(v)$ consists only of vertices corresponding to components of $v$.  Since $\C_A(S_g)$ has no nonseparating annular vertices, it follows that $v$ is represented by a separating curve (it is of type S).  It follows that $\Phi(v)$ consists of a single vertex, and this vertex is isolated in $\C_A(S_g)$, the desired contradiction.

Now that we know $\Phi(v)$ is a join, we would like to show that it is maximal.  Let $V_1 \ast V_2$ be a join in $\C_A(S_g)$ that properly contains $\Phi(v)$.  We would like to show that one of the $V_i$ is trivial.  First note that $V_L$ must be contained in one of $V_1$ or $V_2$; this is because for any two vertices of $V_L$ there is a third vertex of $V_L$ that is not connected by an edge to either.  Similar for $V_R$.  Since $V_M$ has at most one vertex (as mentioned above), it also must be contained in $V_1$ or $V_2$.  Suppose now for contradiction that $w$ is a vertex of $V_1 \ast V_2$ that does not lie in $\Phi(v)$.  If $V_L$ is nonempty, then $w$ has nonempty intersection with some vertex of $V_L$, and similarly for $V_M$ and $V_R$.  Thus $w$ must be contained in the same $V_i$ as each nonempty element of the set $\{V_L,V_M,V_R\}$.  Since $w$ was arbitrary, it follows that one of the $V_i$ is empty, as desired.

We now must show that each maximal join in $\C_A(S_g)$ lies in the image of $\Phi$.  Let $X$ be a maximal join in $\C_A(S_g)$.  Let $V_1 \ast \cdots \ast V_n$ be the complete join decomposition of $X$.  For each $i$ let $R_i$ be the subsurface of $S_g$ determined by $V_i$, that is, $R_i$ is the smallest subsurface containing a representative of each vertex of $V_i$.

We claim that:
\begin{enumerate}
\item the $R_i$ are connected,
\item the $R_i$ are disjoint, and
\item at least one $R_i$ is not an annulus.
\end{enumerate}
The $R_i$ are connected because otherwise the corresponding $V_i$ could be written as the join of the complexes corresponding to the components.  If the $R_i$ were not disjoint, then since each $R_i$ is filled by the vertices of the corresponding $V_i$ we could find vertices of $V_i$ and $V_j$ that had essential intersection.  For the last statement, suppose for contradiction that each $R_i$ were an annulus.  Consider the subgraph of $\C_A(S_g)$ spanned by the vertex corresponding to $R_1$ and the vertices represented in the complement of $R_1$.  This graph is a join (it is a cone on the $R_1$-vertex) and it clearly contains $X$.  It is strictly larger than $X$ because it contains $\MCG(S_g)$-translates of $R_2$ that do not lie in $X$.  This is a contradiction and the claim is proven.

Suppose then that $R_1$ is non-annular.  We claim that $R_1$ is nonseparating.  Suppose to the contrary that $R_1$ has complementary regions $P_1,\dots,P_k$ with $k \geq 2$.  Since $n \geq 2$, there is an $R_i$ contained in some $P_j$, say $P_1$.  The complement of $P_1$ is a region $Q$ containing $R_1$ as a proper subsurface.  Consider the subcomplex of $\C_A(S_g)$ spanned by the vertices represented in either $Q$ or $P_1$.  This is the join of the subcomplexes corresponding to $Q$ and to $P_1$, and it is a nontrivial join since both are nonempty by assumption.  Moreover this subcomplex properly contains $X$ since there are $\MCG(S_g)$-translates of vertices of $X$ that are not contained in any $R_i$.  This is a contradiction, and we conclude that $R_1$ is indeed nonseparating.

Since $R_1$ is nonseparating, its boundary is a dividing set.  Also, since $X$ is a nontrivial join, this dividing set represents a vertex $v$ of $\C_{\delta A} (S_g)$.  Since $R_1$ is compatible with $v$, we must have that $\Phi(v)$ contains $X$.  Since $X$ is maximal, it follows that $\Phi(v)$ is equal to $X$, and so $X$ is in the image of $\Phi$.  This completes the proof of the first statement of the lemma.

Since we chose $R_1$ to be an arbitrary non-annular $R_i$, it follows that $v$ corresponds to the boundary of each non-annular $R_i$.  In particular, there are at most two non-annular $R_i$ and if there are two then they are complementary in $S_g$.  In other words, it makes sense to think of the non-annular $R_i$ as the left and right sides of $v$.  The second statement of the lemma follows.
\end{proof}

\p{Compatible subcomplexes.} For us, the main consequence of Lemma~\ref{lemma:bijection} is that an automorphism of $\C_A(S_g)$ induces an automorphism of the 0-skeleton of $\C_{\delta A}(S_g)$: the image of a vertex of $\C_{\delta A}(S_g)$ is determined by the image of the corresponding maximal join in $\C_A(S_g)$.  The next lemma tells us that this automorphism extends to the 1-skeleton of $\C_{\delta A}(S_g)$.  We say that a subcomplex $V$ of a complex is \emph{compatible} with a subcomplex $W$ if $V$ can be written as $V_1 \ast V_2$ with $V_1$ nonempty and $V_1 \subseteq W$.  

\begin{lemma}
\label{lemma:ca edges}
Suppose $\C_A(S_g)$ has no isolated vertices.  Let $v$ and $w$ be vertices of $\C_{\delta A}(S_g)$.  Then $v$ and $w$ are connected by an edge in $\C_{\delta A}(S_g)$ if and only if $\Phi(v)$ is compatible with $\Phi(w)$.  In particular, compatibility of maximal joins in $\C_A(S_g)$ is a symmetric relation.
\end{lemma} 

\begin{proof}

Denote the left/right decompositions of $\Phi(v)$ and $\Phi(w)$ by $V_L \ast V_M \ast V_R$ and $W_L \ast W_M \ast W_R$.  By Lemma~\ref{lemma:bijection} these are both nontrivial join decompositions.

Suppose first that $v$ and $w$ are connected by an edge in $\C_{\delta A}(S_g)$.  This means that there are representatives of $v$ and $w$ in $S_g$ that are nested dividing sets.  We can choose the left and right subsurfaces of $S_g$ associated to $v$ and $w$ in a compatible way and so that $v$ lies to the left of $w$.  It then follows from the definition of the left/right decomposition that $V_L \ast V_M$ is contained in $W_L$.  Since the left/right decomposition is the same as the complete join decomposition (Lemma~\ref{lemma:bijection}) it must be that $V_L \ast V_M$ is nonempty.  So $\Phi(v)$ is compatible with $\Phi(w)$, as desired.

For the other direction, suppose that $v$ and $w$ are not connected by an edge in $\C_{\delta A}(S_g)$.  Any pair of representatives for $v$ and $w$ are dividing sets that are not nested.  Thus each complementary subsurface for the $v$-dividing set has essential intersection with each complementary subsurface for $w$.  It follows that no component of the left/right decomposition of $\Phi(v)$ can be contained in $\Phi(w)$.  But by Lemma~\ref{lemma:bijection} the left/right decomposition is the same as the complete join decomposition and so $\Phi(v)$ is not compatible with $\Phi(w)$, as desired.
\end{proof}

\p{Finishing the proof.} We are almost ready to prove Theorem~\ref{main:complex}.  

\begin{lemma}
\label{lemma:sr punchline}
Let $A$ be a subset of $\R(S_g)$ so that $\C_A(S_g)$ is connected and has no holes.
\begin{enumerate}
\item\label{bound} Let $R$ be a representative of a vertex of $\C_A(S_g)$; then $\partial R$ represents a simplex in $\C_{\delta A}(S_g)$.
\item The complex $\C_{\delta A}(S_g)$ is connected.
\end{enumerate}
\end{lemma}

\begin{proof}

We begin with the first statement.  Clearly $\partial R$ is a disjoint union of dividing sets; the individual dividing sets are in bijection with the regions of $S_g$ complementary to $R$.  If $R$ is not an annulus, then the statement follows from the no holes condition.  If $R$ is an annulus, then the statement is automatic since there are representatives of the $R$-vertex on both sides of $\partial R$.  This completes the proof of the first statement.  (In the case where $R$ is a separating annulus, $\partial R$ is two parallel separating curves that both represent the same vertex of $\C_{\delta A}(S_g)$.)

We now proceed to the second statement.  Let $v$ and $w$ be vertices of $\C_{\delta A}(S_g)$.  We would like to find a path between $v$ and $w$ in $\C_{\delta A}(S_g)$.  By the definition of $\delta A$ we can choose a vertices $a$ and $b$ of $\C_A(S_g)$ so that $v$ is compatible with $a$ and $w$ is compatible with $b$.  Let $a=a_0,\dots,a_n=b$ be a path in $\C_A(S_g)$.  By the first statement of the lemma the boundary of each $a_i$ represents a simplex $\delta a_i$ of $\C_{\delta A}(S_g)$.  The vertices $v$ and $w$ are connected by edges to $\delta a_0$ and $\delta a_n$, respectively.  Moreover, since each $a_i$ corresponds to a connected subsurface of $S_g$, each component of $\delta a_{i+1}$  lies in a single region of $S_g$ determined by any given component of $\delta a_i$; this is to say that each $\delta a_i \cup \delta a_{i+1}$ is also a simplex of $\C_{\delta A}(S_g)$.  The second statement now follows, as the desired path between $v$ and $w$ lies in the sequence of simplices $v,\delta a_0,\dots, \delta a_n,w$.
\end{proof}

In the proof of Theorem~\ref{main:complex} we will use the partial order on vertices of $\C_A(S_g)$ defined above.  Also, we say that a vertex $v$ of $\C_A(S_g)$ is \emph{1-sided} if for any representative of $v$ in $S_g$ there is exactly one complementary region that contains a representative of a vertex of the link of $v$ in $\C_A(S_g)$.

\begin{proof}[Proof of Theorem~\ref{main:complex}]

As in the statement of the theorem, we have a connected complex of regions $\C_A(S_g)$ with a small vertex and no holes or corks.  By Proposition~\ref{prop:ca easy} we may further assume that $\C_A(S_g)$ has no vertices of type N.  Let $\eta$ denote the natural map $\MCG(S_g) \to \Aut \C_A(S_g)$.  We would like to show that $\eta$ is an isomorphism.  By Lemma~\ref{inj}, $\eta$ is injective.  It remains to show that $\eta$ is surjective.

By Lemmas~\ref{lemma:bijection} and~\ref{lemma:ca edges} there is a well-defined map 
\[
\delta : \Aut \C_A(S_g) \to \Aut \C_{\delta A}(S_g);
\]
for any $\phi \in \Aut \C_A(S_g)$ the image under $\delta \phi$ of a dividing set is determined by the image under $\phi$ of the corresponding maximal join.  The main step in the proof is to show that $\delta$ is injective.

Let $\phi$ be an automorphism of $\C_A(S_g)$ and suppose that $\delta\phi$ is the identity.  We would like to show that $\phi$ is the identity.  To this end, let $v$ be a vertex of $\C_A(S_g)$.  We would like to show that $\phi(v)=v$.  We proceed in three cases.
\begin{enumerate}
\item the case where $v$ is a 1-sided annular vertex,
\item the case where $v$ is a non-annular, $\preceq$-minimal, 1-sided vertex, and
\item the general case.
\end{enumerate}

\medskip \noindent {\em Case 1.} First assume that $v$ is a 1-sided annular vertex.  A representative of $v$ divides $S_g$ into two regions; let $R$ be a region of smallest  genus, and let $Q$ be the other region.  We would like to show that there is a vertex of $\C_{\delta A}(S_g)$ represented by a multicurve in $Q$ that is not peripheral in $Q$ (in other words, this vertex should not be the dividing set that is parallel to $v$).  Since $\C_A(S_g)$ is connected there must be some vertex $w$ of $\C_A(S_g)$ represented by a subsurface of $Q$.  If $w$ is annular then it is not parallel to $v$ and so the vertex of $\C_{\delta A}(S_g)$ parallel to $w$ is the desired vertex.  If $w$ is not annular and is not represented by the entire region $Q$ then we can apply Lemma~\ref{lemma:sr punchline}\eqref{bound} in order to find the desired vertex of $\C_{\delta A}(S_g)$.  So the only remaining possibility is that $w$ corresponds to $Q$ and no other subsurface of $Q$ represents a vertex of $\C_A(S_g)$.  But this pair would be a cork pair, a contradiction.

Since there is one vertex of $\C_{\delta A}(S_g)$ corresponding to a non-peripheral dividing set in $Q$, it follows that $Q$ is filled by representatives of vertices of $\C_{\delta A}(S_g)$.  By assumption, all of these vertices are fixed by $\delta\phi$.  By the definition of $\delta\phi$, the vertex $\phi(v)$ must be disjoint from all of these vertices of $\C_{\delta A}(S_g)$, and hence $\phi(v)$ must be represented by a subsurface of $R$.  Since $v$ is 1-sided it must be that $\phi(v)=v$, as desired.

\medskip \noindent {\em Case 2.} Next assume that $v$ is non-annular, $\preceq$-minimal, and 1-sided.  The argument is very similar to the previous case.  Let $R$ be a representative of $v$.  Since $\C_A(S_g)$ has no holes and $v$ is 1-sided, the complement of $R$ in $S_g$ is connected; denote this region by $Q$.  We again would like to show that there is a vertex of $\C_{\delta A}(S_g)$ represented by a dividing set in $Q$ that is not peripheral in $Q$.  Let $w$ be a vertex of $\C_A(S_g)$ in the link of $v$, so $w$ is represented by a subsurface of $Q$.  If $w$ corresponds to a separating annulus that is not peripheral in $Q$ then the vertex of $\C_{\delta A}(S_g)$ parallel to $w$ is the desired vertex.  If $w$ corresponds to an annulus parallel to the boundary of $Q$, then since $\C_A(S_g)$ has no corks there must be a vertex of $\C_A(S_g)$ represented by a subsurface of $R$.  And since $\C_A(S_g)$ has no holes there must be a further vertex of $\C_A(S_g)$ that is represented by a subsurface of $R$ and that is smaller than $v$ in the partial order, a contradiction.  If $w$ is a non-annular vertex corresponding to a proper subsurface of $Q$ then again we apply Lemma~\ref{lemma:sr punchline}\eqref{bound} to find the desired vertex of $\C_{\delta A}(S_g)$.  So the only remaining case to consider is where $w$ corresponds to $Q$ and no other vertex of $\C_A(S_g)$ corresponds to a subsurface of $Q$.  By the minimality of $v$ it follows that $Q$ and $R$ are homeomorphic and hence that $v$ and $w$ span an isolated edge in $\C_A(S_g)$, a contradiction.

Again, we must have that $Q$ is filled by vertices of $\C_{\delta A}(S_g)$ and so again $\phi(v)$ must correspond to a subsurface of $R$.  By the minimality of $v$ and the assumption that $\C_A(S_g)$ has no holes or corks, it must be that $\phi(v)=v$.

\medskip \noindent {\em Case 3.} We now attack the case of an arbitrary vertex $v$ of $\C_A(S_g)$.  Let $Q$ be a region of $S_g$ that is complementary to some representative of $v$.  We would like to analyze the set of vertices of the link of $v$ represented in $Q$. Since we already dealt with the case of 1-sided annular vertices and since $\C_A(S_g)$ has no holes we may assume that this set of vertices is nonempty.

We put a partial order on this set of vertices of the link of $v$ represented in $Q$.  For any such vertex $u$, we write $\hat u$ for the smallest nonseparating region of $Q$ that contains a representative of $u$.  Then we say $u \leq w$ if $\hat u \subseteq \hat w$.  Since there are finitely many vertices of $\C_A(S_g)$ represented in $Q$ up to homeomorphism of $Q$ the set of $\leq$-minimal elements is nonempty.  We would like to show that each $\leq$-minimal element is either a 1-sided annular vertex or a non-annular, $\preceq$-minimal, 1-sided vertex.

Let $u$ be a vertex of the link of $v$ that is represented in $Q$.  Assume that $u$ is neither a 1-sided annular vertex or a non-annular, $\preceq$-minimal, 1-sided vertex.   We must show that $u$ is not $\leq$-minimal.  

Suppose first that $u$ is non-annular and 2-sided.  Let $P$ be a complementary region to $u$ that does not contain the boundary of $Q$ (there is such a region since the complement of $Q$ is connected).  Since $\C_A(S_g)$ has no holes, there exists a vertex $w$ of $\C_A(S_g)$ supported in $P$.  We have $w < u$, as desired.

Next suppose that $u$ is non-annular, 1-sided vertex that is not $\preceq$-minimal.  Since $u$ is non-annular and $\C_A(S_g)$ has no holes, $u$ is represented by a nonseparating subsurface $P$ of $S_g$.  Since $u$ is not $\preceq$-minimal and is 1-sided and since $\C_A(S_g)$ has no holes, it again follows that there is a vertex $w$ with $w < u$.

Finally suppose that $u$ is annular and 2-sided.  Let $P$ denote the region of $Q$ that is determined by $u$ and does not contain the boundary of $Q$.  Since $\C_A(S_g)$ has no corks, there is a vertex $w$ of $\C_A(S_g)$ represented by a proper subsurface of $P$.  If we do not have $w < u$ then $w$ is non-annular and has one boundary component parallel to $u$.  Since $\C_A(S_g)$ has no holes, there is a vertex $x$ of $\C_A(S_g)$ with $x < u$ and $u$ is again not $\leq$-minimal.  We have succeeded in characterizing the $\leq$-minimal elements.

We will now show that one of the following conditions holds:
\begin{enumerate}
\item $Q$ is filled by the 1-sided annular vertices and non-annular, $\preceq$-minimal, 1-sided vertices of $\C_A(S_g)$ represented in $Q$, or
\item the boundary of $Q$ is parallel to a 1-sided annular vertex of $\C_A(S_g)$ and no non-peripheral subsurfaces of $Q$ represent vertices of $\C_A(S_g)$.
\end{enumerate}
If there is a vertex of $\C_A(S_g)$ that is $\leq$-minimal and is non-peripheral in $Q$ then by our classification of $\leq$-minimal elements, we are in the first case.  If all $\leq$-minimal elements are peripheral in $Q$, then since there are no vertices of type N it must be that the boundary of $Q$ is connected.   Then since there are no holes or corks we must be in the second case.

We are finally ready to complete the proof that $\phi(v)=v$, and hence that $\delta$ is injective.   From our analysis of the vertices of $\C_A(S_g)$ represented in the complementary regions for $v$, we conclude that $v$ is completely determined as follows: it is the unique vertex of $\C_A(S_g)$ represented by a region that is complementary to the supports of the 1-sided annular vertices and non-annular, $\preceq$-minimal, 1-sided vertices of $\C_A(S_g)$ that lie in the link of $v$.  Indeed, no other complementary region contains a representative of a vertex of $\C_A(S_g)$ and---since $\phi$ is invertible---$\phi(v)$ cannot correspond to a proper subregion of $v$.  Since we already showed that 1-sided annular vertices and non-annular, $\preceq$-minimal, 1-sided vertices are fixed by $\phi$, it follows that $\phi(v)=v$.  We have thus proven $\delta$ is injective.

\medskip

Having shown that $\delta$ is injective, we now proceed to complete the proof of the theorem.  Since $\bar g(A) < g/3$, it follows that $\check g(\delta A) < g/3$.  Also it follows from Lemma~\ref{lemma:sr punchline} and the assumption that $\C_A(S_g)$ is connected that $\C_{\delta A}(S_g)$ is connected.  As we already mentioned, Lemma~\ref{upper set} implies that $\delta A$ is an upper set.  Thus by Theorem~\ref{theorem:multi k} the natural map $\MCG(S_g) \to \Aut \C_{\delta A}(S_g)$ is an isomorphism.  We thus have the following diagram:
\[
\xymatrix{
\Aut \C_A(S_g)\  \ar@{^{(}->}[rr]^\delta & & \ \Aut \C_{\delta A}(S_g) \\
& \MCG(S_g) \ar[ur]_{\cong} \ar@{_{(}->}[ul]^{\eta} & 
}
\]
It follows from the definition of $\delta$ that the diagram is commutative.  It follows then from the injectivity of $\delta$ and $\eta$ that both are isomorphisms.  This completes the proof of the theorem.
\end{proof}


\section{Normal subgroups}
\label{section:normal}

In this section we prove Theorem~\ref{main:normal}, which describes the automorphism group and the abstract commensurator group of a normal subgroup of $\Mod(S_g)$ or $\MCG(S_g)$ that has a pure element with a small component.  We will define a complex of regions where the vertices correspond to the supports of certain subgroups of $G$, called basic subgroups.  Theorem~\ref{main:normal} will then be derived from Theorem~\ref{main:complex}, our theorem about automorphisms of complexes of regions.

It might seem more intuitive to relate regions of $S_g$ to elements as opposed to subgroups.  One immediate advantage of using subgroups is that typically the centralizer of a subgroup has its support disjoint from that of the subgroup; this is not true for individual elements.  We were inspired to take the subgroup approach after reading a paper by Hensel \cite{Hensel}, although this idea already appears in the work of  Ivanov, cf. \cite[Section 7.20]{ivanovshiny}.

\subsection{Nielsen--Thurston normal forms and pure elements}

Before getting to the proof of Theorem~\ref{main:normal}, we must recall some ideas related to the Nielsen--Thurston classification for elements of $\Mod(S_g)$.  For basics on this theory, including the definition of a pseudo-Anosov mapping class, see the book by Farb and the second author of this paper \cite[Chapter 13]{primer}.

For a compact surface $R$ with marked points, let $\PMod(R)$ denote the subgroup of $\Mod(R)$ consisting of elements that induce the trivial permutation of the marked points (for us the mapping class group of a surface with boundary is the group of homotopy classes of orientation-preserving homeomorphisms that restrict to the identity on the boundary, so the components of the boundary are automatically not permuted).

A \emph{partial pseudo-Anosov} element of $\Mod(S_g)$ is the image of a pseudo-Anosov element of $\PMod(R)$ under the map $\PMod(R) \to \Mod(S_g)$ induced by the inclusion of a region $R$ of $S_g$; here $R$ has no marked points and we allow $R=S_g$.  By the work of Thurston and Birman--Lubotzky--McCarthy \cite{blm} the region $R$---the \emph{support}---is canonically defined up to isotopy.  Similarly, the support of a power of a Dehn twist is well defined.  

Following Ivanov \cite{ivanovshiny}, a mapping class $f \in \MCG(S_g)$ is \emph{pure} if it is equal to a product $f_1 \cdots f_k$ where 
\begin{enumerate}
 \item each $f_i$ is a partial pseudo-Anosov element or a power of a Dehn twist,
 \item for $i \neq j$ the supports of $f_i$ and $f_j$ have disjoint representatives, and
 \item for $i \neq j$ the support of $f_i$ is not homotopic into the support of $f_j$.
\end{enumerate}
The $f_i$ are called the \emph{Nielsen--Thurston components}, or simply \emph{components}, of $f$.  Each component can be characterized as a \emph{pseudo-Anosov component} or a \emph{Dehn twist component}.  Note that the third condition is vacuous if $f_i$ is a pseudo-Anosov component.  Also note that pure elements of $\MCG(S_g)$ lie in $\Mod(S_g)$.  

Under these definitions, the Nielsen--Thurston components of a pure mapping class are not canonical when the supports of pseudo-Anosov components $f_i$ and $f_j$ have a pair of parallel boundary components.  Indeed, if the homotopy class of these boundary components is $c$ then we may replace $f_i$ and $f_j$ with $f_iT_c$ and $f_jT_c^{-1}$.  Still, the support of a pure element is a well-defined subsurface, invariant under taking nontrivial powers.

A subgroup of $\Mod(S_g)$ is \emph{pure} if each element is pure.  Ivanov proved that there is a subgroup of finite index in $\Mod(S_g)$ that is pure \cite[Corollary 1.8]{ivanovshiny}.  
The support of a pure subgroup of $\Mod(S_g)$ is again a well-defined subsurface of $S_g$ (sometimes called the active subsurface; cf. \cite{Mosher}).  The support of a pure subgroup is invariant under passage to finite-index subgroups.

If $R$ is a component of the support of the pure subgroup $H$, then there is a well-defined \emph{reduction homomorphism}
\[ 
H \to \PMod(R^\circ), 
\]
where $R^\circ$ is the surface obtained from $R$ by collapsing each component of the boundary to a marked point.  By definition, the image is an irreducible subgroup of $\PMod(R^\circ)$.  (The `P' in $\PMod$ also stands for pure, but in general $\PMod(R)$ is not pure in the sense of Ivanov.) 

If we have two partial pseudo-Anosov elements with equal support $R$, we say that the elements are \emph{independent} if their corresponding pseudo-Anosov foliations are distinct from each other.  If a partial pseudo-Anosov mapping class $f$ has support $R$ then its image under the reduction homomorphism $\langle f \rangle \to \Mod(R^\circ)$ is pseudo-Anosov, and so we can say that two partial pseudo-Anosov elements with support $R$ are independent if their reductions are independent pseudo-Anosov elements.

We will repeatedly use the following fact, usually without mention. 

\begin{fact}
\label{commute fact}
Two elements of a pure subgroup of $\Mod(S_g)$ commute if and only if
\begin{enumerate}
\item the supports of their components are pairwise disjoint or equal and
\item in the case where two pseudo-Anosov components have equal support $R$ the components are dependent.
\end{enumerate}
In particular if two pure elements commute then all of their nontrivial powers commute and if two pure elements do not commute then all of their nontrivial powers fail to commute.
\end{fact}

Fact~\ref{commute fact} follows from the fact that commuting elements have compatible canonical reduction systems \cite[Lemma 3.1(1)]{blm} and the fact that in a pure subgroup the centralizer of a pseudo-Anosov element is cyclic \cite[Lemma 5.10]{ivanovshiny}.  

\begin{lemma}
\label{irred}
Let $H$ be a non-abelian pure subgroup of $\Mod(S_g)$.  Then there is a component $R$ of the support of $H$ so that the reduction homomorphism $H \to \PMod(R^\circ)$ 
has non-abelian image.  For any such $R$, the centralizer of $H$ is supported in the complement of $R$.
\end{lemma}

\begin{proof}

For the first statement, let $f$ and $h$ be two elements of $H$ that do not commute.  By Fact~\ref{commute fact} they have components whose supports $R_1$ and $R_2$ have essential intersection and if $R_1=R_2$ then these components are independent partial pseudo-Anosov elements.  The regions $R_1$ and $R_2$ must lie in the same component $R$ of the support of $H$.  The images of $R_1$ and $R_2$ in $R^\circ$ still have essential intersection and if $R_1=R_2$ then the images of the chosen components of $f$ and $h$ are still independent.  By Fact~\ref{commute fact} the images of $f$ and $h$ in $\PMod(R^\circ)$ do not commute.

We now prove the second statement.  Let $R$ be as in the statement.  By the definition of a component of the support of $H$, the image of $H$ under the reduction homomorphism is irreducible; since the image is also non-abelian it contains a pseudo-Anosov element \cite[Theorem 5.9]{ivanovshiny}.  Since (the image of) $H$ is not abelian, it follows that its image in $\Mod(R^\circ)$ contains two independent pseudo-Anosov elements (conjugate the first one by any element that does not commute with it).  The preimages in $H$ of these pseudo-Anosov elements of  $\Mod(R^\circ)$ are two elements of $H$ that have independent partial pseudo-Anosov components with support $R$.  The second statement now follows from Fact~\ref{commute fact}.
\end{proof}


\subsection{The commutator trick}

Fix $g$ and let $N$ be some fixed pure, normal subgroup of $\Mod(S_g)$ that contains a pure element with a small component.  Also let $G$ be a subgroup of $N$ of finite index. 

We would like to define a complex of regions where the vertices correspond to the supports of certain subgroups of $N$.  One of the basic difficulties we need to overcome is that a typical element (or subgroup) of $N$ has disconnected support, but in a complex of regions the vertices must correspond to connected subsurfaces.  Further, if an element with multiple Nielsen--Thurston components lies in $N$ then the individual components may or may not lie in $N$.  The next lemma deals with this problem.  The key point is that if $N$ contains an element $f$ so that one component of the support of $f$ is a non-annular region $R$, then there is a different element $f'$ of $N$---not equal to a component of $f$---whose support is $R$.  The element $f'$ is obtained as a commutator of $f$ with an appropriately chosen element of $\Mod(S_g)$.    

\begin{lemma}
\label{comm trick}
Let $g \geq 0$, let $N$ be a pure, normal subgroup of $\Mod(S_g)$ and let $G$ be a finite-index subgroup of $N$.

Let $f$ be an element of $G$ and let $R$ be a region of $S_g$ so that some component of $f$ has support contained as a non-peripheral subsurface of $R$ and all other components of $f$ have support that is either contained in or is  disjoint from $R$.  Let $J$ be the subgroup of $G$ consisting of all elements supported in $R$.  Then
\begin{enumerate}
\item\label{comm trick 1} $J$ is not abelian,  
\item\label{comm trick 2} $J$ contains an element with support $R$, and
\item\label{comm trick 3} the centralizer $C_G(J)$ is supported in the complement of $R$.
\end{enumerate}
\end{lemma}

\begin{proof}

The first step is to show that $J$ contains a nontrivial element $j$ whose support is a non-peripheral subsurface of $R$.  To this end, consider the reduction homomorphism $G \to \PMod(R^\circ)$.  Since the support of $f$ is not peripheral in $R$, the image of $f$ is nontrivial.  Also, since $f$ is pure it follows that the image of $f$ is not central in $\PMod(R^\circ)$ (cf. \cite[page 77]{primer}).  Thus there is an $\bar h \in \PMod(R^\circ)$ that does not commute with the image of $f$.  Let $h$ be any element of the preimage of $\bar h$ in $\PMod(R)$ (the reduction homomorphism $\PMod(R) \to \PMod(R^\circ)$ is surjective).  As above we may identify $h$ as an element of $\Mod(S_g)$.  Since $N$ is normal in $\MCG(S_g)$, the commutator $[f,h]$ lies in $N$. By construction $[f,h]$ is supported in $R$.  Let $j$ be a nontrivial power of $[f,h]$ that lies in $G$.  The kernel of the map $\PMod(R) \to \PMod(R^\circ)$ is precisely the set of elements with peripheral support, so $j$ is the desired element.    

The second step is to show that $J$ contains two independent partial pseudo-Anosov elements with support $R$.  All three statements of the lemma follow from this and Fact~\ref{commute fact}.  Let $j$ be the element found in the first step.  Any conjugate of $j$ by an element of $\Mod(S_g)$ has a power in $G$.  It follows that the image of $J$ under the reduction homomorphism $\PMod(R) \to \PMod(R^\circ)$ is irreducible and not abelian (apply Fact~\ref{commute fact}).  Any such subgroup contains two independent pseudo-Anosov elements (irreducible subgroups contain pseudo-Anosov elements \cite[Theorem 5.9]{ivanovshiny} and any nontrivial conjugate of a pseudo-Anosov element is an independent pseudo-Anosov element and has a power in the image of $J$).  Any preimages of these elements in $J$ are the desired elements and the proof is complete.
\end{proof}

\subsection{Basic subgroups} 

As in the previous section let $N$ be a pure, normal subgroup of $\Mod(S_g)$ that contains a pure element with a small component, and let $G$ be a finite-index subgroup of $N$.  We now define basic subgroups of $N$.  We will show that these subgroups have connected supports, and so they can used to build a complex of regions for $N$. 

We define a strict partial order on subgroups of $G$ by the following rule:
\[
H \prec H' \ \ \text{ if } \ \ C_G(H') \subsetneq C_G(H).
\]
(by a strict partial order we mean a binary relation that is irreflexive---meaning that no element is related to itself---and transitive).  A subgroup of $G$ is \emph{basic} if among non-abelian subgroups of $G$ it is minimal with respect to this strict partial order; specifically $B$ is basic if there is no non-abelian subgroup $B'$ of $G$ with $B' \preceq B$.

Consider for example the case where $N$ is the Torelli group $\I(S_g)$.  Let $R$ be a sphere with four separating boundary components in $S_g$.  Then the subgroup $B$ of $N$ consisting of elements that are supported in $R$ is a basic subgroup.  Indeed, it is not abelian since it contains Dehn twists about curves that intersect and it is minimal because all proper subsurfaces of $R$ have abelian mapping class group.

\begin{lemma}
\label{lemma:basic}
Let $g \geq 0$, let $N$ be a pure, normal subgroup of $\Mod(S_g)$ that contains a pure element with a small component, and let $G$ be a finite-index subgroup of $N$.
\begin{enumerate}
\item\label{basic region} The support of a basic subgroup of $G$ is a non-annular region of $S_g$.
\item\label{basic passage} If $B$ is a basic subgroup of $N$ then $B \cap G$ is a basic subgroup of $G$; similarly, any basic subgroup of $G$ is also a basic subgroup of $N$.
\item\label{basic small} $N$ contains a basic subgroup with small support.
\item\label{basic action} $\MCG(S_g)$ acts on the set of supports of basic subgroups of $G$.
\end{enumerate}
\end{lemma}

\begin{proof}

We begin with the first statement.   Let $H$ be a basic subgroup of $G$.  Since $H$ is not abelian, the support of $H$ is clearly not empty and not an annulus.  It is also easy to see that the support of $H$ is proper.  Indeed if the support of $H$ were $S_g$, then by Lemma~\ref{irred} the centralizer $C_G(H)$ would be trivial.  On the other hand, since $N$---hence $G$---contains an element $f$ with a small component, we can apply  Lemma~\ref{comm trick} to $f$ and a small region $Q$ in order to produce a non-abelian subgroup $J$ with nontrivial centralizer $C_G(J)$ (since $Q$ is small there is an $h \in \Mod(S_g)$ so that  $h J h^{-1} \cap G$ lies in $C_G(J)$).  Any such $J$ would be strictly smaller than $H$ in the strict partial order.

Suppose for contradiction that the support of $H$ is not connected.  By Lemma~\ref{irred} there is a component $R$ of the support of $H$ so that the reduction homomorphism $H \to \PMod(R^\circ)$ has non-abelian image.  In particular, there must be an element $h$ of $H$ with a Nielsen-Thurston component whose support is non-peripheral in $R$.   By Lemma~\ref{comm trick}\eqref{comm trick 1}, the subgroup 
\[ J = \{ h \in G \mid \Supp(h) \subseteq R \}. \]
 of $G$ is not abelian.

We will show that $J \prec H$.  By Lemma~\ref{irred}  the centralizer $C_G(H)$ of $H$ is supported in the complement of $R$.  It follows that $C_G(H) \subseteq C_G(J)$.  We must now produce an element of $C_G(J) \setminus C_G(H)$.  Let $h$ be an element of $H$ whose support is not contained in $R$.  Let $Q$ be a non-annular region of $S_g$ that is disjoint from $R$, that contains a component of the support of $h$ as a non-peripheral subsurface, and is disjoint from all other components of the support of $h$.  Applying Lemma~\ref{comm trick}\eqref{comm trick 1} to $h$ and $Q$ we find the desired element of $C_G(J) \setminus C_G(H)$.  Thus $J \prec H$, a contradiction. This completes the proof of the first statement.  

The second statement has two parts.  For the first part, let $B$ be a basic subgroup of $N$.  We would like to show that $B \cap G$ is a basic subgroup of $G$.  It follows from Fact~\ref{commute fact} that $B \cap G$ is non-abelian, and it follows from the first statement and Lemma~\ref{irred} that the centralizer of $B \cap G$ in $G$ consists of all elements whose support lies outside the support of $B \cap G$.  Thus, if $B \cap G$ were not basic in $G$ there would be a non-abelian subgroup $H$ of $G$ whose centralizer contains the centralizer of $B \cap G$ and also contains at least one extra element.  The support of this extra element would have to intersect the support of $B \cap G$ and hence by the first statement and Lemma~\ref{irred} the support of $H$ would be a proper subsurface of the support of $B \cap G$.  Since $B \cap G$ has finite index in $B$ the latter is equal to the support of $B$ and we can conclude that $H$ is smaller than $B$ in the strict partial order on subgroups of $N$, contradicting the assumption that $B$ was basic in $N$.

For the second part of the second statement, let $B$ be a basic subgroup of $G$.  We would like to show that $B$ is basic in $N$.  Suppose $B'$ is a non-abelian subgroup of $N$ that is smaller than $B$ in the strict partial order on subgroups of $N$.  Consider the subgroup $B' \cap G$ of $G$.  Again by Fact~\ref{commute fact} the subgroup $B' \cap G$ is not abelian.  Since $B' \prec B$ (in $N$) the centralizer of $B$ in $N$ is strictly contained in the centralizer of $B'$ in $N$.  Thus the centralizer of $B$ in $G$ is contained in the centralizer of $B'$ in $G$.  Since the containment is strict there is an element $f$ of $N$ that lies in the centralizer in $N$ of $B'$ but not of $B$.  Some power of $f$ lies in $G$.  It then follows from Fact~\ref{commute fact} that this power of $f$ lies in the centralizer of $B'$ in $G$ but not the centralizer of $B$ in $G$, contradicting the assumption that $B$ is basic in $G$. 

We now treat the third statement.  By assumption, $N$---hence $G$---contains a nontrivial pure element $f$ with a small component $f_1$.  
Let $R_1$ be a fitting region for $f$ corresponding to $f_1$, as per the definition of $\hat g(f)$ in Section~\ref{sec:complex}; the component $f_1$ and the region $R_1$ satisfy the hypotheses of Lemma~\ref{comm trick}.  There is a region $Q$ of $S_g$ that has genus less than $g/3$, has connected boundary, and contains $R_1$.  We define
\[ 
J_{R_1} = \{ h \in G \mid \Supp h \subseteq R_1 \}.
\]
It follows from parts \eqref{comm trick 1} and \eqref{comm trick 3} of Lemma~\ref{comm trick} that $J_{R_1}$ is not abelian and that $C_G(J_{R_1})$ is the set of elements of $G$ with support in the complement of $R_1$.  We would like to show that $J_{R_1}$ contains a basic subgroup.  If $J_{R_1}$ itself is not minimal, then it contains a non-abelian subgroup $J_{R_1}'$ with $J_{R_1}' \prec J_{R_1}$.  It follows that $C_G(J_{R_1}')$ contains an element whose support intersects $R_1$.  By Lemma~\ref{irred} the support of $J_{R_1}'$ has a component that is a proper subsurface $R_2$ of $R_1$.  If $J_{R_2}$ is the subgroup of $G$ consisting of all elements with support in $R_2$ we still have $J_{R_2} \prec J_{R_1}$.  If $J_{R_2}$ is not minimal we can repeat the process.  Since the Euler characteristics of the $R_i$ are strictly increasing and negative, the process must eventually terminate at a basic subgroup $H$ with support contained in $Q$.

For the fourth statement, suppose that $R$ is the support of a basic subgroup $B$ in $G$ and let $f \in \MCG(S_g)$.  We would like to show that $f(R)$ is the support of a basic subgroup of $G$.  Since $\Mod(S_g)$ acts transitively on each $\MCG(S_g)$-orbit of regions in $S_g$, we may assume without loss of generality that $f$ lies in $\Mod(S_g)$.

Let $J_R$ denote the subgroup of $G$ consisting of all elements with support in $R$.  Let $B' = (f J_R f^{-1}) \cap G$.  As $N$ is normal in $\Mod(S_g)$ and $G$ has finite index in $N$ the subgroup $B'$ has finite index in $f J_R f^{-1}$.  Since the support of a pure subgroup of $\Mod(S_g)$ is invariant under taking finite index subgroups, the support of $B'$ is $f(R)$.  We would like to show that $B'$ is basic.  It follows from Fact~\ref{commute fact} that $B'$ is not abelian.  Finally we must show that $B'$ is minimal.  By part \eqref{comm trick 3} of Lemma~\ref{comm trick} the support of $J_R$ is $R$ (not a proper subsurface of $R$) and so the support of $f J_R f^{-1}$, hence $B'$, is $f(R)$.  Since $B'$ is not abelian it then follows from Fact~\ref{commute fact} that $C_G(B')$ consists of exactly the elements of $G$ with support outside $f(R)$.  If there were a subgroup $B''$ of $G$ with $B'' \prec B'$ then there would be an element of $C_G(B'')$ whose support has essential intersection with $f(R)$ and so the support of $B''$ would be contained in $f(R)$.   But then---again using Fact~\ref{commute fact}---the subgroup $(f^{-1}B''f) \cap G$ would be strictly smaller than $B$ in the strict partial order, contradicting the minimality of $B$.
\end{proof}

\subsection{The complex}

Again, let $N$ be some fixed pure, normal subgroup of $\Mod(S_g)$ that contains a pure element with a small component.  We are now ready to construct the desired complex of regions $\C_N(S_g)$ for $N$.  By statements~\eqref{basic region} and~\eqref{basic action} of Lemma~\ref{lemma:basic} there is a complex of regions $\C_N^\sharp(S_g)$ whose set of vertices is in bijection with the set of supports of basic subgroups of $N$.  One point to note here is that there are many basic subgroups of $N$ corresponding to a given vertex of $\C_N^\sharp(S_g)$.  

By Lemma~\ref{lemma:basic}\eqref{basic region} the complex of regions $\C_N^\sharp(S_g)$ has no annular vertices and so it has no corks.  Also by Lemma~\ref{lemma:basic}\eqref{basic small} it has a small vertex.  On the other hand, $\C_N^\sharp(S_g)$ does not necessarily satisfy the other hypotheses of Theorem~\ref{main:complex}: it may have holes and it may be disconnected.  

To illustrate the first point, we again consider the above example where $N=\I(S_g)$.  Again, let $R$ be a four-holed sphere where each component of the boundary is separating in $S_g$.  Also assume that one of the complementary regions $Q$ is a handle.  We already explained why there is a basic subgroup with support $R$, and so $R$ represents a vertex of $\C_N^\sharp(S_g)$.  But in this particular case $R$ represents a vertex of $\C_N^\sharp(S_g)$ with a hole.  Indeed, the subgroup of $\I(S_g)$ consisting of elements supported in $Q$ is cyclic (it is generated by the Dehn twist about the boundary of $Q$) and so there are no basic subgroups of $N$ supported in $Q$.  Thus $Q$ represents a hole for the $R$-vertex.

We can also imagine an example where $\C_N^\sharp(S_g)$ is not connected.  Let $N$ be the normal closure in $\MCG(S_g)$ of two elements $f$ and $h$ of $\Mod(S_g)$, where $f$ is a partial pseudo-Anosov element supported on a handle $Q$ and $h$ is a partial pseudo-Anosov element supported on a subsurface $R$ of genus zero with $g+1$ boundary components.  Using Lemma~\ref{comm trick} we can find non-abelian subgroups of $N$ with supports $Q$ and $R$.  For the typical choices of $f$ and $h$ we would expect these subgroups to be basic, and so $Q$ and $R$ will represent vertices $v$ and $w$ of $\C_N^\sharp(S_g)$.  It is also possible that all vertices of $\C_N^\sharp(S_g)$ lie in the orbit of $v$ and $w$.  Since no vertex in the orbit of $v$ is connected to $w$ the complex $\C_N^\sharp(S_g)$ is disconnected in this case.

In light of these issues, we now set about modifying $\C_N^\sharp(S_g)$ so that it satisfies all of the hypotheses of Theorem~\ref{main:complex}.  First, let $\C_N^\flat(S_g)$ be the filling of $\C_N^\sharp(S_g)$  (cf. Section~\ref{sec:sr}).  By Lemma~\ref{filling}, the complex $\C_N^\flat(S_g)$ has no holes and by Lemma~\ref{small filling} it has a small vertex.  Since the filling of a non-annular vertex is non-annular, and since $\C_N^\sharp(S_g)$ has no annular vertices, $\C_N^\flat(S_g)$ has no corks.  

In summary, the complex $\C_N^\flat(S_g)$ has no holes or corks and it has a small vertex, but it might be disconnected.  We have the following fact, which is a straightforward application of the Putman trick from the proof of Lemma~\ref{lemma:sp graph connected}.  

\begin{lemma}
\label{small component}
Let $\C_A(S_g)$ be a complex of regions.  The small vertices of $\C_A(S_g)$ all lie in the same connected component of $\C_A(S_g)$.  
\end{lemma}

Finally, we may define $\C_N(S_g)$ as the connected component of $\C_N^\flat(S_g)$ containing the small vertices.  Clearly $\C_N(S_g)$ is connected and has a small vertex.  Also since $\C_N^\flat(S_g)$ has no annular vertices, $\C_N(S_g)$ has no corks.  To check that $\C_N(S_g)$ satisfies the hypotheses of Theorem~\ref{main:complex}, it remains to check that $\C_N(S_g)$ has no holes.  

Suppose that $R$ is a region of $S_g$ that represents a vertex of $\C_N(S_g)$.  Then $R$ also represents a vertex of $\C_N^\flat(S_g)$.  Since the latter has no holes, each complementary region of $R$ supports a vertex of $\C_N^\flat(S_g)$.  Each of these vertices clearly lies in the connected component of $\C_N^\flat(S_g)$ containing the $R$-vertex and so they all correspond to vertices of $\C_N(S_g)$.  So $\C_N(S_g)$ has no holes.  We thus have the following consequence of Theorem~\ref{main:complex}.

\begin{prop}
\label{prop:defining cn}
Let $N$ be a pure, normal subgroup of $\Mod(S_g)$ that contains a pure element with a small component.  Then the natural map
\[
\MCG(S_g) \to \Aut \C_N(S_g)
\]
is an isomorphism.
\end{prop}

\subsection{Action of the commensurator groups on the complex}

For the complex of regions $\C_N(S_g)$ to be useful, we would like to know that an automorphism---or an abstract commensurator---of $N$ gives rise to an automorphism of $\C_N(S_g)$.  In order to obtain a well-defined action, we must deal with the issue that there are many basic subgroups of $N$ giving rise to the same vertex of $\C_N(S_g)$.  

In what follows, we will denote by $v_B$ the vertex of $\C_N(S_g)$ arising from the basic subgroup $B$ of $N$.  As mentioned, we may have two basic subgroups $B$ and $B'$ with $v_B = v_{B'}$.  Also, for $G$ a finite-index subgroup of $N$ and $B$ a basic subgroup of $G$ we define the \emph{basic centralizer} of $B$ in $G$ to be the subgroup of $G$ generated by the basic subgroups of $G$ in the centralizer of $B$; we denote this group by $BC_G(H)$.  

\begin{lemma}
\label{lemma:bc}
Let $N$ be a pure, normal subgroup of $\Mod(S_g)$ that contains an element with a small component.  Let $G$ be a subgroup of $N$ of finite index.  Let $h \in \MCG(S_g)$ and let $h_\star$ denote its image under the natural map $\MCG(S_g) \to \Aut \C_N(S_g)$.  Let $B$ and $B'$ be two basic subgroups of $G$.  Then
\begin{enumerate}
\item\label{bc v} $v_B = v_{B'}$ if and only if $BC_G(B) = BC_G(B')$, 
\item\label{bc wd} $v_B = v_{B'}$ if $B'$ is a finite index subgroup of $B$,
\item\label{bc e} $v_B$ is connected by an edge to $v_{B'}$ if and only if $B' \leqslant BC_G(B)$, and
\item\label{bc act}  $h_\star(v_B) = v_{hBh^{-1}}$.
\end{enumerate}
\end{lemma}

\begin{proof}

Let $R$ denote the support of $B$.  By Lemma~\ref{lemma:basic}\eqref{basic region}, the subsurface $R$ is a non-annular region of $S_g$.  Denote by $P_1,\dots,P_m$ the complementary regions that do contain the supports of other basic subgroups of $G$ and denote by $Q_1,\dots,Q_n$ the complementary regions that do not.  By the definition of $\C_N(S_g)$, the vertex $v_B$ is represented by the union of $R$ with the $Q_i$; call this region $R'$.  By Lemma~\ref{irred}, we have that $BC_G(B)$ is the subgroup of $G$ generated by the basic subgroups of $G$ with support in the complement of $R'$.  All statements of the lemma follow (for the second statement apply Fact~\ref{commute fact}).  
\end{proof}

\begin{prop}
\label{prop:comm to aut}
Let $N$ be a pure, normal subgroup of $\Mod(S_g)$ that contains an element with a small component.  There is a map
\[
\Comm N \to \Aut \C_N(S_g)
\]
defined as follows: if $\alpha : G_1 \to G_2$ is an isomorphism between finite-index subgroups of $N$ and $[\alpha]_\star$ is the image in $\Aut \C_N(S_g)$ of the equivalence class of $\alpha$, then for any basic subgroup $B$ of $N$ we have
\[
[\alpha]_\star (v_B) = v_{\alpha(B \cap G_1)}.
\]
\end{prop}

\begin{proof}

Our first objective is to show that the formula given in the statement of the proposition makes sense.  Let $\alpha : G_1 \to G_2$ be an isomorphism between finite-index subgroups of $N$ and let $B$ be a basic subgroup of $N$.  By Lemma~\ref{lemma:basic}\eqref{basic passage}, the group $B \cap G_1$ is a basic subgroup of $G_1$.  Since $\alpha$ is an isomorphism from $G_1$ to $G_2$ it follows that $\alpha(B \cap G_1)$ is a basic subgroup of $G_2$.  Again by Lemma~\ref{lemma:basic}\eqref{basic passage} the group $\alpha(B \cap G_1)$ is a basic subgroup of $N$.  Thus $v_{\alpha(B \cap G_1)}$ is indeed a vertex of $\C_N(S_g)$.  

Next we must show that the formula in the statement gives a well-defined action of $\Comm N$ on the set of vertices of $\C_N(S_g)$.  There are two issues, namely, that an element of $\Comm N$ has many representatives and also that there are many basic subgroups of $N$ giving rise to the same vertex of $\C_N(S_g)$.

Let  $\alpha : G_1 \to G_2$ be an isomorphism between finite-index subgroups of $N$ and let $B$ be a basic subgroup of $N$.  Let $\alpha' : G_1' \to G_2'$  be another isomorphism of finite-index subgroups of $N$ that represents the same element of $\Comm N$ as $\alpha$.  We must show that 
\[ 
v_{\alpha(B \cap G_1)} = v_{\alpha'(B \cap G_1')}.
\]
Since $\alpha$ and $\alpha'$ agree on a finite-index subgroup of $N$ it suffices to treat the case where $G_1'$ is a finite-index subgroup of $G_1$ and $\alpha|(B \cap G_1') = \alpha'|(B \cap G_1')$.  In this case $B \cap G_1'$ has finite index in $B \cap G_1$ and so $\alpha'(B \cap G_1')=\alpha(B \cap G_1')$ has finite index in $\alpha(B \cap G_1)$.  It follows that the supports of $\alpha'(B \cap G_1')$ and $\alpha(B \cap G_1)$ are equal, which is to say that $v_{\alpha(B \cap G_1)} = v_{\alpha'(B \cap G_1')}$, as desired.

To deal with the second ambiguity, suppose that $B'$ is another basic subgroup of $N$ with $v_B = v_{B'}$.  With $\alpha$ as above, we must show that
\[
v_{\alpha(B \cap G_1)} = v_{\alpha(B' \cap G_1)}.
\]
By Lemma~\ref{irred} the centralizer of a basic subgroup is invariant under passage to finite-index subgroups.  It follows from this and Lemma~\ref{lemma:basic}\eqref{basic passage} that $v_{B \cap G_1} = v_{B' \cap G_1}$.  It further follows from Lemma~\ref{lemma:bc}\eqref{bc v} that $BC_{G_1}(B \cap G_1)$ is equal to $BC_{G_1}(B' \cap G_1)$.  As basic centralizers are preserved by isomorphisms, we have that $BC_{G_2}(\alpha(B \cap G_1))$ is equal to $BC_{G_2}(\alpha(B' \cap G_1))$.  Again by Lemma~\ref{lemma:bc}\eqref{bc v} we have the desired equality $v_{\alpha(B \cap G_1)} = v_{\alpha(B' \cap G_1)}$.  

Having now shown that $[\alpha]$ induces a well-defined permutation of the set of vertices of $\C_N(S_g)$, it remains to check that this permutation preserves the set of edges.  To this end, we claim that if $B$ and $B'$ are basic subgroups of $N$ then $v_B$ and $v_{B'}$ are connected by an edge if and only if $B$ and $B'$ commute.  The subgroups $B$ and $B'$ commute if and only if the subgroups $B \cap G_1$ and $B' \cap G_1$ commute, and the latter holds if and only if $\alpha(B \cap G_1)$ and $\alpha(B' \cap G_1)$ commute.  It then follows from Lemma~\ref{lemma:bc}\eqref{bc e} that the given permutation of vertices extends to an automorphism of $\C_N(S_g)$.
\end{proof}


\subsection{Proof of the theorem}
\label{sec:finishing}

We are almost ready to prove Theorem~\ref{main:normal}.  Let us first introduce some notation.  For $f \in \MCG(S_g)$ denote by $\alpha_f$ the automorphism of $\MCG(S_g)$ given by conjugation by $f$, that is, $\alpha_f(h) = fhf^{-1}$ for all $h \in \MCG(S_g)$.  If $f$ lies in the normalizer of $N$ then we may consider $\alpha_f$ as an element of $\Aut N$ (technically, the restriction of $\alpha_f$ to $N$ gives an element of $\Aut N$).  Similarly, if there is a restriction of $\alpha_f$ that is an isomorphism between finite-index subgroups of $N$ then we may regard $[\alpha_f]$ as an element of $\Comm N$ (this is an abuse of notation: we should more properly write $[\bar \alpha_f]$ where $\bar \alpha_f$ is the restriction).

\begin{proof}[Proof of Theorem~\ref{main:normal}]

For simplicity we first deal with the case where $N$ is normal in $\MCG(S_g)$ (this is the first statement of Theorem~\ref{main:normal}).  Let $P$ be a pure normal subgroup of finite index in $\MCG(S_g)$.  We will begin by describing a sequence of five maps $\Phi_1,\dots,\Phi_5$ as follows:
\begin{align*}
\MCG(S_g) \stackrel{\Phi_1}{\to}& \Aut N \stackrel{\Phi_2}{\to} \Comm N \stackrel{\Phi_3}{\to} \Comm N \cap P  \\
&  \stackrel{\Phi_4}{\to} \Aut \C_{N \cap P}(S_g) \stackrel{\Phi_5}{\to} \MCG(S_g).
\end{align*}
Here are the definitions of the maps:
\begin{itemize}
\item $\Phi_1$ is the conjugation map, that is, $\Phi_1(f)= \alpha_f$, 
\item $\Phi_2$ maps an element of $\Aut N$ to its equivalence class in $\Comm N$, 
\item $\Phi_3$ maps the equivalence class of an isomorphism between finite index subgroups of $N$ to the equivalence class of any restriction that is an isomorphism between finite-index subgroups of $N \cap P$,
\item $\Phi_4$ is the map from Proposition~\ref{prop:comm to aut}, and
\item $\Phi_5$ is the isomorphism from Proposition~\ref{prop:defining cn}.
\end{itemize}
To prove the theorem in the case where $N$ is normal in $\MCG(S_g)$ we will show that $\Phi_1$, $\Phi_2$, $\Phi_3$, $\Phi_4$, and $\Phi_5$ are all injective and that the composition 
\[
\Phi_5 \circ \Phi_4 \circ \Phi_3 \circ \Phi_2 \circ \Phi_1
\]
is the identity.  The injectivity of the $\Phi_i$ and the surjectivity of the composition together imply that the $\Phi_i$ are surjective, and hence are isomorphisms.  That $\Phi_1$ and $\Phi_2$ are isomorphisms is the content of the first statement of  Theorem~\ref{main:normal}.  The map $\Phi_5 \circ \Phi_4 \circ \Phi_3$ is the natural map $\Comm N \to \MCG(S_g)$ from the statement of the theorem.  

We begin by showing that the $\Phi_i$ are injective.  The map $\Phi_1$ is injective by an argument similar the one given in Lemma~\ref{inj}; indeed, if $f \in \MCG(S_g)$ commutes with $h \in N$ then $f$ fixes the canonical reduction system of $h$.  Any element of $N$ with a small component has a nonempty canonical reduction system, and as in the proof of Lemma~\ref{inj} the orbit under $\MCG(S_g)$ of this canonical reduction system is dense in $\PMF(S_g)$.  Since $N$ is normal in $\MCG(S_g)$ the canonical reduction system for $khk^{-1}$ is the $k$-image of the canonical reduction system for $h$, the injectivity follows.
  
We now show that $\Phi_2$ is injective.  Let $\alpha \in \Aut N$ be an element of the kernel.  Let $f$ be an element of $N$.  We would like to show that $\alpha(f)=f$.  Let $h$ be a pseudo-Anosov element of $N$ (all infinite normal subgroups of $\Mod(S_g)$ contain such elements).  Since $\Phi_2(\alpha)$ is the identity, there is a finite-index subgroup $G$ of $N$ so that $\alpha|G$ is the identity.  There is an $m > 0$ so that $h^{m}$ and $(fhf^{-1})^{m}$ lie in $G$ and so we have $\alpha(h^{m})=h^{m}$ and $\alpha((fhf^{-1})^{m}) = (fhf^{-1})^{m}$.  We have
\begin{align*}
f h^m f^{-1} = (fhf^{-1})^m = \alpha((fhf^{-1})^m) = \alpha(f) \alpha(h^m) \alpha(f)^{-1} = \alpha(f) h^m \alpha(f)^{-1}.
\end{align*}
In other words $f^{-1}\alpha(f)$ commutes with $h^m$ and so $f^{-1}\alpha(f)$ fixes the point in $\PMF(S_g)$ corresponding to the unstable foliation of $h$.  Since $h$ was arbitrary, $f^{-1}\alpha(f)$ fixes all the points in $\PMF(S_g)$ corresponding to the unstable foliations of pseudo-Anosov elements of $N$.  But since $N$ is normal in $\Mod(S_g)$, these points are dense in $\PMF(S_g)$.  As in the proof of Lemma~\ref{inj} we conclude that $f^{-1}\alpha(f)$ is the identity, which is to say that $\alpha(f)=f$.  Thus, $\alpha$ is the identity, as desired. 

The map $\Phi_3$ is an isomorphism since a finite-index subgroup of $N \cap P$ also has finite index in $N$.

Next, we will show that $\Phi_4$ is injective.   To this end, fix some isomorphism $\alpha : G_1 \to G_2$ representing an element $[\alpha]$ of $\Comm N \cap P$.  As in Proposition~\ref{prop:comm to aut} denote the image of $[\alpha]$ in $\Aut \C_{N \cap P}(S_g)$ by $\alpha_\star$.  Assume that $\alpha_\star$ is the identity.  We must show that $[\alpha]$ is the identity.  We will show that in fact $\alpha$ is the identity (in particular $G_1=G_2$).  So let $h \in G_1$.  We would like to show that $\alpha(h)=h$.  Let $h_\star$ and $\alpha(h)_\star$ denote the images of $h$ and $\alpha(h)$ under natural map $\MCG(S_g) \to \Aut \C_{N \cap P}(S_g)$.  Since the latter is injective (Lemma~\ref{inj}) it suffices to show that $h_\star = \alpha(h)_\star$.

So let $B$ be an arbitrary basic subgroup of $N$.  By Lemma~\ref{lemma:basic}\eqref{basic passage} and Lemma~\ref{lemma:bc}\eqref{bc wd} we may assume without loss of generality that $B$ is contained in $G_1$.    We have
\begin{align*}
h_\star (v_B) &= v_{hBh^{-1}} = \alpha_\star v_{hBh^{-1}} = v_{\alpha(hBh^{-1})}  = v_{\alpha(h)\alpha(B)\alpha(h)^{-1}}  \\
&\quad = \alpha(h)_\star v_{\alpha(B)} = \alpha(h)_\star \alpha_\star v_{B} = \alpha(h)_\star v_{B}.
\end{align*}
In order, the equalities use Lemma~\ref{lemma:bc}\eqref{bc act}, the assumption that $\alpha_\star$ is the identity, Proposition~\ref{prop:comm to aut}, the assumptions that $B$ and $h$ both lie in $G_1$, again Lemma~\ref{lemma:bc}\eqref{bc act}, again Proposition~\ref{prop:comm to aut}, and again the assumption that $\alpha_\star$ is the identity.  It follows that $h_\star = \alpha(h)_\star$ and so $\alpha(h)=h$, as desired.

The fifth and final map $\Phi_5 : \Aut \C_{N \cap P}(S_g) \to \MCG(S_g)$ is an isomorphism by Proposition~\ref{prop:defining cn}; in particular it is injective.

As in Theorem~\ref{main:complex} the isomorphism $\Phi_5$ is the inverse of the natural map $\MCG(S_g) \to \Aut \C_{N \cap P}(S_g)$.  It follows that the composition $\Phi_5 \circ \Phi_4 \circ \Phi_3 \circ \Phi_2 \circ \Phi_1$ is the identity.  Indeed, given $f \in 
\Mod(S_g)$ the image in $\Comm N \cap P$ is the element given by conjugation by $f$.  Thus the image in $\Aut \C_{N \cap P}(S_g)$ is $f_\star$ and so the image in $\MCG(S_g)$ must again be $f$.  This completes the proof of the first statement of the theorem.

We now prove the second statement of the theorem.  Assume that $N$ is normal in $\Mod(S_g)$ but not in $\MCG(S_g)$.  Again let $P$ be a pure normal subgroup of finite index in $\MCG(S_g)$.  We will consider a collection of homomorphisms $\Phi_i$ analogous to the ones from the proof of the first statement:
\[ 
\xymatrix{
& \MCG(S_g) \ar@{-->}[ddd]^{\Phi_6} & \\
\Aut \C_{N \cap P}(S_g) \ar[ur]^{\Phi_5}_{\cong} & &  \Mod(S_g) \ar@{_{(}->}[ul] \ar[d]^{\Phi_1} \\
\Comm N \cap P \ar[u]^{\Phi_4} & & \Aut N \ar[dl]^{\Phi_2} \\
 & \Comm N \ar[ul]^{\Phi_3}_{\cong} & 
}
\]
The maps $\Phi_1,\dots,\Phi_5$ are all defined in the same way as in the first case, except that the domain of $\Phi_1$ has changed.  The map $\Phi_6$ is the natural map $\MCG(S_g) \to \Comm N$ from the statement of the theorem.  It maps $f \in \MCG(S_g)$ to the element of $\Comm N$ given by conjugation by $f$.  For $f$ not in $\Mod(S_g)$ there may be no subgroup of finite index $H$ in $N$ so that $fHf^{-1}$ has finite index in $N$; if such $f$ exists the map $\Phi_6$ is not well defined.\footnote{One is tempted to think that since $N$ is normal in $\Mod(S_g)$ and since $\Mod(S_g)$ has index two in $\MCG(S_g)$, there is a subgroup of finite index in $N$ that is normal in $\MCG(S_g)$.  If this were true it would imply that $\Phi_6$ is always well defined.  However, it is not true.  Consider for example the group $A = \Z^2 \rtimes \Z/2$ where $\Z/2$ acts on $\Z^2$ by interchanging the two factors; the subgroup $N$ of $A$ corresponding to the first factor of $\Z^2$ is normal in $\Z^2$ but there is a conjugate of $N$ in $A$ with which $N$ has trivial intersection.}

By the same arguments as in the proof of the first statement of the theorem, the maps $\Phi_1$, $\Phi_2$, $\Phi_3$, $\Phi_4$, and $\Phi_5$ are all injective and the composition
\[
\Phi_5 \circ \Phi_4 \circ \Phi_3 \circ \Phi_2 \circ \Phi_1
\]
is the inclusion map.

We consider two cases, according to whether the image of $\Phi_5 \circ \Phi_4 \circ \Phi_3$ is $\Mod(S_g)$ or $\MCG(S_g)$.  Let us first assume that the image is $\Mod(S_g)$.  In this case $\Comm N$ is isomorphic to $\Mod(S_g)$ under the natural map $\Phi_5 \circ \Phi_4 \circ \Phi_3$.  Further $\Phi_1$ and $\Phi_2$ are isomorphisms since $\Phi_1$, $\Phi_2$, and $\Phi_5 \circ \Phi_4 \circ \Phi_3$ are injective and their composition is the inclusion, as desired.

We now proceed to the case where the image of $\Phi_5 \circ \Phi_4 \circ \Phi_3$ is $\MCG(S_g)$.  Since we already know that $\Phi_1$, $\Phi_2$, and $\Phi_5 \circ \Phi_4 \circ \Phi_3$ are injective, and that their composition is the inclusion, the only remaining statements to prove are that $\Phi_6$ is the inverse to $\Phi_5 \circ \Phi_4 \circ \Phi_3$ and that $\Phi_1$ is surjective, hence an isomorphism.  We treat each of these statements in turn.

We first show that $\Phi_6$ is a left inverse to $\Phi_5 \circ \Phi_4 \circ \Phi_3$ (hence is the inverse).  The proof of this statement follows along similar lines as in the proof of the injectivity of $\Phi_4$.  We fix some isomorphism $\alpha : G_1 \to G_2$ representing an element $[\alpha]$ of $\Comm N \cap P$ (which under $\Phi_3$ is canonically isomorphic to $\Comm N$).  As above we denote the $\Phi_4$-image of $[\alpha]$ in $\Aut \C_{N \cap P}(S_g)$ by $\alpha_\star$.  Assume that $\alpha_\star$ maps to $f \in \MCG(S_g)$ under $\Phi_5$.  We would like to show that $[\alpha]$ is given by the restriction of the conjugation map $\alpha_f$.  This is the same as saying that $[\alpha]=\Phi_6 \circ \Phi_5 \circ \Phi_4 \circ \Phi_3 ([\alpha])$, as desired.

Let $h \in G_1$.  We would like to show that $\alpha(h)=fhf^{-1}$.  For a mapping class $j \in \MCG(S_g)$ denote by $j_\star$ the image of $j$ under natural map $\Phi_5^{-1} : \MCG(S_g) \to \Aut \C_{N \cap P}(S_g)$ (so by definition $\alpha_\star=f_\star$).  Since $\Phi_5^{-1}$ is injective (Lemma~\ref{inj}) it suffices to show that $(fhf^{-1})_\star = \alpha(h)_\star$, or $(fh)_\star = (\alpha(h)f)_\star$.  Let $B$ be an arbitrary basic subgroup of $N$.  Again we may assume without loss of generality that $B$ is contained in $G_1$.    We have
\begin{align*}
(fh)_\star v_B &= f_\star(h_\star v_B) = \alpha_\star(h_\star v_B) = \alpha_\star v_{hBh^{-1}}  = v_{\alpha(hBh^{-1})} \\ &=   v_{\alpha(h)\alpha(B)\alpha(h^{-1})}  = \alpha(h)_\star v_{\alpha(B)} = \alpha(h)_\star \alpha_\star (v_B) \\ &= \alpha(h)_\star f_\star v_B = (\alpha(h)f)_\star v_B.
\end{align*}
In order, the equalities use the fact that $\Phi_5^{-1}$ is a homomorphism, the assumption that $f_\star= \alpha_\star$, Lemma~\ref{lemma:bc}\eqref{bc act}, Proposition~\ref{prop:comm to aut}, the assumptions that $B$ and $h$ both lie in $G_1$, again Lemma~\ref{lemma:bc}\eqref{bc act}, again Proposition~\ref{prop:comm to aut}, again the assumption that $f_\star= \alpha_\star$, and again the fact that $\Phi_5^{-1}$ is a homomorphism.  It follows that $(fh)_\star = (\alpha(h)f)_\star$ and so $\alpha(h) = fhf^{-1}$, as desired.

We have proven that the left inverse of the natural map $\Phi_5 \circ \Phi_4 \circ \Phi_3 : \Comm N \to \MCG(S_g)$ is the natural map $\Phi_6 : \MCG(S_g) \to \Comm N$, as in the statement of the theorem.

To complete the proof of the theorem, it remains to show that $\Phi_1$ is an isomorphism, in other words that $\Phi_1$ is surjective.  Let $\alpha \in \Aut N$.  As usual, denote $\Phi_2(\alpha)$ by $[\alpha]$.  Let $f$ denote $\Phi_5 \circ \Phi_4 \circ \Phi_3([\alpha])$.  Since $\Phi_6$ is the inverse of $\Phi_5 \circ \Phi_4 \circ \Phi_3$, we have $\Phi_6(f) = [\alpha]$.  In other words $[\alpha] = [\alpha_f]$.  To show that $\alpha$ lies in the image of $\Phi_1$ we will show that $\alpha = \alpha_f$.  The proof will be similar to the proof of the injectivity of $\Phi_2$.  

Let $n \in N$.  We would like to show that $\alpha(n) = fnf^{-1}$.  Let $h$ be a pseudo-Anosov element of $N$.  Since $[\alpha] = [\alpha_f]$, there is a finite-index subgroup $G$ of $N$ so that $\alpha|G = \alpha_f|G$.  There is an $m > 0$ so that $h^{m}$ and $nh^mn^{-1}$ lie in $G$.  So $\alpha(h^{m})=fh^{m}f^{-1}$ and $\alpha(nh^mn^{-1}) = fnh^mn^{-1}f^{-1}$.  We have
\begin{align*}
fnh^mn^{-1}f^{-1} = \alpha(nh^mn^{-1}) = \alpha(n)\alpha(h^m)\alpha(n)^{-1} = \alpha(n)fh^mf^{-1}\alpha(n)^{-1}.
\end{align*}
From the equality of the first and last expressions we deduce that
\[
f^{-1}\alpha(n)^{-1}fnh^m = h^mf^{-1}\alpha(n)^{-1}fn,
\]
in other words that $f^{-1}\alpha(n)^{-1}fn$ commutes with $h^m$.  Since $h$ was arbitrary, it follows as in the proof of the injectivity of $\Phi_2$ that $f^{-1}\alpha(n)^{-1}fn$ is the identity, which is to say that $\alpha(n) = f nf^{-1}$, as desired.  This completes the proof of the theorem.
\end{proof}

\bibliographystyle{plain}
\bibliography{meta}

\begin{thebibliography}{10}

\bibitem{TeichTrans}
Annette A'Campo-Neuen.
\newblock English translation of {T}eichm\"uller's paper {\it {v}er\"anderliche
  {r}iemannsche {f}l\"achen} ({V}ariable {R}iemann surfaces) {D}eutsche {M}ath.
  7 (1944), 344--359.
\newblock In {\em Handbook of {T}eichm\"uller theory. {V}ol. {IV}}, volume~19
  of {\em IRMA Lect. Math. Theor. Phys.}, pages 787--803. Eur. Math. Soc.,
  Z\"urich, 2014.

\bibitem{aramayona}
Javier Aramayona.
\newblock Simplicial embeddings between pants graphs.
\newblock {\em Geom. Dedicata}, 144:115--128, 2010.

\bibitem{al}
Javier Aramayona and Christopher~J. Leininger.
\newblock Finite rigid sets in curve complexes.
\newblock {\em J. Topol. Anal.}, 5(2):183--203, 2013.

\bibitem{bdr}
Juliette Bavard, Spencer Dowdall, and Kasra Rafi.
\newblock Isomorphisms between big mapping class groups.
\newblock 08 2017.

\bibitem{bbm}
Joan Birman, Nathan Broaddus, and William Menasco.
\newblock Finite rigid sets and homologically nontrivial spheres in the curve
  complex of a surface.
\newblock {\em J. Topol. Anal.}, 7(1):47--71, 2015.

\bibitem{blm}
Joan~S. Birman, Alex Lubotzky, and John McCarthy.
\newblock Abelian and solvable subgroups of the mapping class groups.
\newblock {\em Duke Math. J.}, 50(4):1107--1120, 1983.

\bibitem{bowditch}
Brian~H. Bowditch.
\newblock Rigidity of the strongly separating curve graph.
\newblock {\em Michigan Math. J.}, 65(4):813--832, 2016.

\bibitem{survey}
Tara Brendle and Dan Margalit.
\newblock Ivanov's metaconjecture: a survey.

\bibitem{kg}
Tara~E. Brendle and Dan Margalit.
\newblock Commensurations of the {J}ohnson kernel.
\newblock {\em Geom. Topol.}, 8:1361--1384 (electronic), 2004.

\bibitem{kgadd}
Tara~E. Brendle and Dan Margalit.
\newblock Addendum to: ``{C}ommensurations of the {J}ohnson kernel'' [{G}eom.
  {T}opol. {\bf 8} (2004), 1361--1384; mr2119299].
\newblock {\em Geom. Topol.}, 12(1):97--101, 2008.

\bibitem{bps}
Martin Bridson, Alexandra Pettet, and Juan Souto.
\newblock The abstract commensurator of the {J}ohnson kernels, in preparation.

\bibitem{lei}
Lei Chen.
\newblock Personal communication, October 2017.

\bibitem{CLM}
Matt Clay, Christopher~J. Leininger, and Dan Margalit.
\newblock Abstract commensurators of right-angled {A}rtin groups and mapping
  class groups.
\newblock {\em Math. Res. Lett.}, 21(3):461--467, 2014.

\bibitem{cmm}
Matt Clay, Johanna Mangahas, and Dan Margalit.
\newblock Normal right-angled {A}rtin subgroups of mapping class groups.
\newblock In preparation.

\bibitem{dgo}
F.~Dahmani, V.~Guirardel, and D.~Osin.
\newblock Hyperbolically embedded subgroups and rotating families in groups
  acting on hyperbolic spaces.
\newblock {\em Mem. Amer. Math. Soc.}, 245(1156):v+152, 2017.

\bibitem{dehn}
Max Dehn.
\newblock {\em Papers on group theory and topology}.
\newblock Springer-Verlag, New York, 1987.
\newblock Translated from the German and with introductions and an appendix by
  John Stillwell, With an appendix by Otto Schreier.

\bibitem{disarlo}
Valentina Disarlo.
\newblock Combinatorial rigidity of arc complexes, 2015.

\bibitem{farb}
Benson Farb.
\newblock Some problems on mapping class groups and moduli space.
\newblock In {\em Problems on mapping class groups and related topics},
  volume~74 of {\em Proc. Sympos. Pure Math.}, pages 11--55. Amer. Math. Soc.,
  Providence, RI, 2006.

\bibitem{farbivanovannounce}
Benson Farb and Nikolai~V. Ivanov.
\newblock The {T}orelli geometry and its applications: research announcement.
\newblock {\em Math. Res. Lett.}, 12(2-3):293--301, 2005.

\bibitem{farbivanov}
Benson Farb and Nikolai~V. Ivanov.
\newblock Torelli buildings and their automorphisms.
\newblock 10 2014.

\bibitem{primer}
Benson Farb and Dan Margalit.
\newblock {\em A primer on mapping class groups}, volume~49 of {\em Princeton
  Mathematical Series}.
\newblock Princeton University Press, Princeton, NJ, 2012.

\bibitem{flp}
Albert Fathi, Fran{\c{c}}ois Laudenbach, and Valentin Po{\'e}naru.
\newblock {\em Thurston's work on surfaces}, volume~48 of {\em Mathematical
  Notes}.
\newblock Princeton University Press, Princeton, NJ, 2012.
\newblock Translated from the 1979 French original by Djun M. Kim and Dan
  Margalit.

\bibitem{funartqft}
Louis Funar.
\newblock On the {TQFT} representations of the mapping class groups.
\newblock {\em Pacific J. Math.}, 188(2):251--274, 1999.

\bibitem{funar}
Louis Funar.
\newblock On power subgroups of mapping class groups.
\newblock \url{https://arxiv.org/pdf/0910.1493v1}, 2009.

\bibitem{grothendieck}
Alexander Grothendieck.
\newblock Techniques de construction et th\'eor\`emes d'existence en
  g\'eom\'etrie alg\'ebrique. {I}.
\newblock In {\em Description axiomatique de l'espace de {T}eichm\"uller et de
  ses variantes, tome 13}, number~1 in S\'eminaire {H}enri {C}artan. 1960-1961.

\bibitem{Hensel}
Sebastian Hensel.
\newblock Rigidity and flexibility for handlebody groups.
\newblock 08 2016.

\bibitem{jhh2}
Jes{\'u}s~Hern{\'a}ndez Hern{\'a}ndez.
\newblock Edge-preserving maps of curve graphs.
\newblock 11 2016.

\bibitem{jhh3}
Jes{\'u}s~Hern{\'a}ndez Hern{\'a}ndez and Jos{\'e} Ferr{\'a}n~Valdez Lorenzo.
\newblock Automorphism groups of simplicial complexes of infinite type
  surfaces.
\newblock 02 2014.

\bibitem{irmak1}
Elmas Irmak.
\newblock Superinjective simplicial maps of complexes of curves and injective
  homomorphisms of subgroups of mapping class groups.
\newblock {\em Topology}, 43(3):513--541, 2004.

\bibitem{irmak}
Elmas Irmak.
\newblock Complexes of nonseparating curves and mapping class groups.
\newblock {\em Michigan Math. J.}, 54(1):81--110, 2006.

\bibitem{irmak2}
Elmas Irmak.
\newblock Superinjective simplicial maps of complexes of curves and injective
  homomorphisms of subgroups of mapping class groups. {II}.
\newblock {\em Topology Appl.}, 153(8):1309--1340, 2006.

\bibitem{irmakmac}
Elmas Irmak and John~D. McCarthy.
\newblock Injective simplicial maps of the arc complex.
\newblock {\em Turkish J. Math.}, 34(3):339--354, 2010.

\bibitem{ivanovshiny}
Nikolai~V. Ivanov.
\newblock {\em Subgroups of {T}eichm\"uller modular groups}, volume 115 of {\em
  Translations of Mathematical Monographs}.
\newblock American Mathematical Society, Providence, RI, 1992.
\newblock Translated from the Russian by E. J. F. Primrose and revised by the
  author.

\bibitem{ivanov}
Nikolai~V. Ivanov.
\newblock Automorphism of complexes of curves and of {T}eichm\"uller spaces.
\newblock {\em Internat. Math. Res. Notices}, (14):651--666, 1997.

\bibitem{ivanov15}
Nikolai~V. Ivanov.
\newblock Fifteen problems about the mapping class groups.
\newblock In {\em Problems on mapping class groups and related topics},
  volume~74 of {\em Proc. Sympos. Pure Math.}, pages 71--80. Amer. Math. Soc.,
  Providence, RI, 2006.

\bibitem{im}
Nikolai~V. Ivanov and John~D. McCarthy.
\newblock On injective homomorphisms between {T}eichm\"uller modular groups.
  {I}.
\newblock {\em Invent. Math.}, 135(2):425--486, 1999.

\bibitem{johnson}
Dennis Johnson.
\newblock The structure of the {T}orelli group. {I}. {A} finite set of
  generators for {${\mathcal I}$}.
\newblock {\em Ann. of Math. (2)}, 118(3):423--442, 1983.

\bibitem{jones}
Gareth~A. Jones.
\newblock Personal communication.
\newblock 2018.

\bibitem{korkmaz}
Mustafa Korkmaz.
\newblock Automorphisms of complexes of curves on punctured spheres and on
  punctured tori.
\newblock {\em Topology Appl.}, 95(2):85--111, 1999.

\bibitem{korkpap}
Mustafa Korkmaz and Athanase Papadopoulos.
\newblock On the arc and curve complex of a surface.
\newblock {\em Math. Proc. Cambridge Philos. Soc.}, 148(3):473--483, 2010.

\bibitem{luo}
Feng Luo.
\newblock Automorphisms of the complex of curves.
\newblock {\em Topology}, 39(2):283--298, 2000.

\bibitem{magnus}
Wilhelm Magnus.
\newblock On a theorem of {M}arshall {H}all.
\newblock {\em Ann. of Math. (2)}, 40:764--768, 1939.

\bibitem{pants}
Dan Margalit.
\newblock Automorphisms of the pants complex.
\newblock {\em Duke Math. J.}, 121(3):457--479, 2004.

\bibitem{margulis}
G.~A. Margulis.
\newblock {\em Discrete subgroups of semisimple {L}ie groups}, volume~17 of
  {\em Ergebnisse der Mathematik und ihrer Grenzgebiete (3) [Results in
  Mathematics and Related Areas (3)]}.
\newblock Springer-Verlag, Berlin, 1991.

\bibitem{masurminsky}
Howard~A. Masur and Yair~N. Minsky.
\newblock Geometry of the complex of curves. {I}. {H}yperbolicity.
\newblock {\em Invent. Math.}, 138(1):103--149, 1999.

\bibitem{mccarthy}
John~D. McCarthy.
\newblock Automorphisms of surface mapping class groups. {A} recent theorem of
  {N}. {I}vanov.
\newblock {\em Invent. Math.}, 84(1):49--71, 1986.

\bibitem{mcpap}
John~D. McCarthy and Athanase Papadopoulos.
\newblock Simplicial actions of mapping class groups.
\newblock In {\em Handbook of {T}eichm\"uller theory. {V}olume {III}},
  volume~17 of {\em IRMA Lect. Math. Theor. Phys.}, pages 297--423. Eur. Math.
  Soc., Z\"urich, 2012.

\bibitem{alan2}
Alan McLeay.
\newblock Normal subgroups of mapping class groups of punctured surfaces.
\newblock 01 2018.

\bibitem{alan}
Alan McLeay.
\newblock Normal subgroups of the braid group and the metaconjecture of ivanov,
  01 2018.

\bibitem{taylor}
R.~Taylor McNeill.
\newblock {A new filtration of the Magnus kernel of the Torelli group}, 2013.

\bibitem{Mosher}
Lee Mosher.
\newblock {MSRI} course on mapping class groups.
\newblock {\em Lecture note, Fall}, 2007.

\bibitem{nielsen}
Jakob Nielsen.
\newblock Untersuchungen zur {T}opologie der geschlossenen zweiseitigen
  {F}l\"achen.
\newblock {\em Acta Math.}, 50(1):189--358, 1927.

\bibitem{ftpg}
Andrew Putman.
\newblock The fundamental theorem of projective geometry.
\newblock \url{http://www3.nd.edu/~andyp/notes/FunThmProjGeom.pdf}.

\bibitem{putman}
Andrew Putman.
\newblock A note on the connectivity of certain complexes associated to
  surfaces.
\newblock {\em Enseign. Math. (2)}, 54(3-4):287--301, 2008.

\bibitem{pss}
Paul Schmutz~Schaller.
\newblock Mapping class groups of hyperbolic surfaces and automorphism groups
  of graphs.
\newblock {\em Compositio Math.}, 122(3):243--260, 2000.

\bibitem{shane}
Shane Scott.
\newblock {\em Automorphisms of some complexes}.
\newblock PhD thesis, Georgia Institute of Technology, 2018.

\bibitem{shack}
Kenneth~J. Shackleton.
\newblock Combinatorial rigidity in curve complexes and mapping class groups.
\newblock {\em Pacific J. Math.}, 230(1):217--232, 2007.

\bibitem{tchangang}
Roger~Tambekou Tchangang.
\newblock Le groupe d'automorphismes du groupe modulaire.
\newblock {\em Ann. Inst. Fourier (Grenoble)}, 37(2):19--31, 1987.

\bibitem{teich}
Oswald Teichm\"uller.
\newblock Ver\"anderliche {R}iemannsche {F}l\"achen.
\newblock {\em Deutsche Math.}, 7:344--359, 1944.

\end{thebibliography}

\end{document}